\documentclass{amsart}
\usepackage{amsfonts}
\usepackage{amssymb}
\usepackage{amsmath}
\usepackage{amsthm}
 \usepackage{graphicx,xcolor}
\usepackage{moreverb}
\usepackage{fancybox}
\usepackage{fancyvrb}
\usepackage{comment}
\usepackage{color}
\usepackage{multirow}
\usepackage{stmaryrd}

\usepackage[all]{xy}
\usepackage{tikz-cd}
\usetikzlibrary{positioning,arrows.meta,calc}
\usepackage{fullpage}

\theoremstyle{plain}
\newtheorem{theorem}{Theorem}[section]
\newtheorem{lemma}[theorem]{Lemma}

\newtheorem{cor}[theorem]{Corollary}
\newtheorem{prop}[theorem]{Proposition}

\newtheorem{question}{Question}

\newtheorem{fact}{Fact}
\newtheorem{obs}[theorem]{Observation}

\theoremstyle{definition}
\newtheorem{definition}[theorem]{Definition}
\newtheorem{example}[theorem]{Example}

\theoremstyle{definition}
\newtheorem*{remark}{Remark}

\newtheorem*{terminology}{Terminology}
\newtheorem*{notation}{Notation}

\newtheorem*{construction}{Construction}

\newcommand{\upto}{\upharpoonright}
\newcommand{\fr}{\mbox{}^\smallfrown}
\newcommand{\N}{\mathbb{N}}
\newcommand{\om}{\omega}
\newcommand{\pcolon}{\colon\!\!\subseteq}
\newcommand{\ep}{\varepsilon}

\newcommand{\U}{\mathcal{U}}

\newcommand{\sfN}{\mathsf{N}}

\newcommand{\dom}{{\rm dom}}
\newcommand{\cod}{{\rm cod}}

\newcommand{\tgame}{\mathfrak{G}^{\sf t}}

\newcommand{\down}{{\downarrow}}
\newcommand{\upar}{{\uparrow}}
\newcommand{\tW}{\leq_{\sf tW}}
\newcommand{\tgw}{\leq_{\sf tW}^{\sf G}}
\newcommand{\gw}{\leq_{\sf W}^{\sf G}}
\newcommand{\eqtgw}{\equiv_{\sf tW}^{\sf G}}

\newcommand{\bflogic}[1]{\underline{{\bf #1}}}

\newcommand{\pair}[1]{\langle{#1}\rangle}
\newcommand{\eval}[1]{[\hspace{-.14em}[{#1}]\hspace{-.14em}]}

\newcommand{\tto}{\rightrightarrows}
\newcommand{\xedge}[1]{\xrightarrow{#1}}
\newcommand{\xarr}[1]{\overset{#1}{\rightarrowtriangle}}

\newcommand{\totarr}{\arr}

\newcommand{\sii}{\Sigma_1\text{-}}
\newcommand{\pii}{\Pi_1\text{-}}

\newcommand{\aaa}{{\bf a}.}
\newcommand{\ppp}{{\bf p}.}

\newcommand{\arr}{\rightarrowtriangle}

\newcommand{\solve}{\searrow}

\newcommand{\leaf}{{}^\times\!}
\newcommand{\internal}{{}^\circ\!}

\newcommand{\paset}[1]{\underline{#1}}

\newcommand{\tpot}{\Upsilon}
\newcommand{\tactual}{\mathbb{T}}

\title{Modified realizability subtoposes and total Weihrauch reducibility}

\author{Akihito Kajikawa}
\author{Masamori Kaku}
\author{Takayuki Kihara}
\author{Satoshi Nakata}
\thanks{Graduate School of Informatics and Graduate School of Mathematics, Nagoya University, Japan}

\date{}

\begin{document}
\maketitle

\begin{abstract}
In recent years, there has been rapid development in the foundational study of oracle computability from the perspective of Lawvere-Tierney topologies and their sheaves.
In this article, we formulate and analyze the notion of reducibility within the framework of total computability.
Then, using sheaf subtoposes derived from oracles in the total computable setting, we establish separations between various hierarchies of logical principles, including the hierarchies of the weak law of excluded middle $\mathbf{WLEM}$, the lessor limited principle of omniscience $\mathbf{LLPO}$, and Markov's principle $\mathbf{MP}$.
\end{abstract}

\section{Introduction}

\subsection{Summary}
One of the most fundamental areas of study in computability theory is the investigation of the degrees of difficulty of various problems; this requires the notion of oracle computability to formalize the comparison of such problems.
In recent years, a new theoretical framework for oracle computability has been rapidly developing, based on notions from topos theory -- specifically, Lawvere-Tierney (LT) topologies and their associated sheaves.

At an early stage, Hyland \cite{Hyl82} pointed out that Turing oracles (single-valued oracles) could be viewed as LT-topologies on the effective topos, yet subsequent development of this perspective was rather slow; see e.g.~\cite{vOBook}.
However, in recent years, this field has seen rapid development, driven by its {\bf co-development with concrete theory}.
Key turning points include the concrete presentation and combinatorial analysis of LT-topologies on the effective topos by Lee-van Oosten \cite{LvO13}, and the connection established by Kihara \cite{Kih23} to the concrete theory of Weihrauch degrees \cite{BGP21} that has been extensively studied in computable analysis.
This evolution has also shifted the primary research focus from {\em single-valued (deterministic) oracles} to {\em multi-valued (non-deterministic) oracles}.

There are numerous important points regarding the new foundation for oracle computability.
\begin{enumerate}
\item[A.] 
First, the abstraction led to the discovery of notions that appropriately extend the standard notion of oracle computability:
the extended Weihrauch reducibility by Bauer \cite{Bau22} and its multi-query version by Kihara \cite{Kih23}, known as the Arthur-Nimue-Merlin reducibility.
This incorporates not only {\em demonic} non-determinism but also {\em angelic} non-determinism; see also Abou Samra-Madore \cite{ASM26}.
\begin{itemize}
\item[--] 
As a byproduct, Kihara-Ng \cite{KiNg26} incorporated ``computability by majority'' as a type of oracle computability and discovered connections to existing research on the strength of ultrafilters and ideals (i.e., the Rudin-Keisler order and the Kat\v{e}tov order).
\end{itemize}
\item[B.] Next, based on extended oracle computability, it becomes possible to perform the concrete analysis of the structure of LT-topologies and associated sheaf subtoposes on the effective topos; see e.g.~Kihara \cite{Kih23}, Kihara-Ng \cite{KiNg25C}, Kihara-Ng \cite{KiNg26,KiNg26b}.
\item[C.] 
Furthermore, the abstraction allows for a unified treatment of various oracle notions -- such as not only oracles for partial computability (the standard setting in computability theory) but also oracles for total computability, which is the subject of this article.
\item[D.] 
Finally, a LT-topology induces a sheaf subtopos of a base topos, and a topos provides a certain kind of mathematical universe.
Consequently, oracle computability automatically yields a mathematical model; by utilizing this, one can achieve unprovability proofs relying solely on arguments concerning oracle computability/reducibility.
See e.g.~Kihara \cite{Kih20}, Kihara-Ng \cite{KN-Church}, and Nakata \cite{Nak24,Nak26}
\end{enumerate}

To demonstrate the power of this idea, this article addresses the problem of separating the hierarchy of logical principles as a test case.
We deal here with the hierarchy of the weak law of excluded middle $\bflogic{WLEM}$, the lessor limited principle of omniscience $\bflogic{LLPO}$, and Markov's principle $\bflogic{MP}$.
These are also known as the propositional de Morgan's law, $\Sigma_1$-de Morgan's law, and the $\Sigma_1$-double negation elimination principle, respectively.
The disjunctive Markov's principle $\bflogic{MP}^\lor$ is also known as the $\Pi_1$-de Morgan's law.

We provide a complete separation of the principles appearing in the following figure:
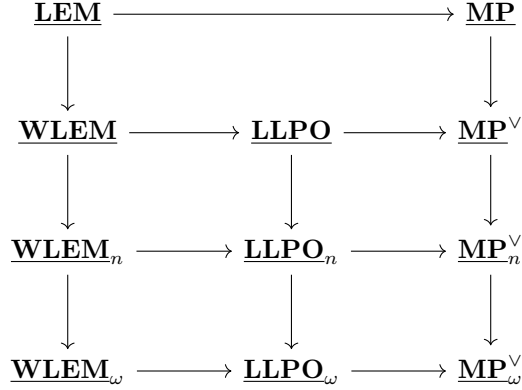
\[
\begin{tikzcd}[row sep=large, column sep=large]
\bflogic{LEM}
  \arrow[rr]
  \arrow[d]
&
&
\bflogic{MP}
  \arrow[d]
\\
\bflogic{WLEM}
  \arrow[r]
  \arrow[d]
&
\bflogic{LLPO}
  \arrow[r]
  \arrow[d]
&
\bflogic{MP}^{\vee}
  \arrow[d]
\\
\bflogic{WLEM}_n
  \arrow[r]
  \arrow[d]
&
\bflogic{LLPO}_n
  \arrow[r]
  \arrow[d]
&
\bflogic{MP}^{\vee}_n
  \arrow[d]
\\
\bflogic{WLEM}_{\omega}
  \arrow[r]
&
\bflogic{LLPO}_{\omega}
  \arrow[r]
&
\bflogic{MP}^{\vee}_{\omega}
\end{tikzcd}
\]

The precise definitions of these principles are given in Section \ref{sec:main-realizability}.
Similar separations have also been achieved by Hendtlass-Lubarsky \cite{HeLu16}.
On the one hand, their separation models are combinations of Kripke semantics with set-theoretic models.
On the other hand, our separation models are certain kinds of realizability models based on the connection between oracle computability and LT topology, thereby achieving separation while satisfying several variants of constructive Church's thesis $\bflogic{CT}_0$ (every total relation has a computable choice).
The variants we consider include $\bflogic{CT}_0!$ (every total function is computable), $\bflogic{wCT}_0$ (a double negated variant of Church's thesis), $\bflogic{CT}_0^{\rm maj}$ (every total relation has a computable bound).
For the precise definitions, see Sections \ref{sec:Church-thesis} and \ref{sec:major-Church-thesis}.

\begin{theorem}\label{thm:main-theorem}
Assume $n\geq2$.
Each of the following items is true in some subtopos of the modified realizability topos.
\begin{enumerate}
\item $\neg\bflogic{MP}^\lor_\om+\bflogic{CT}_0$.
\item $\bflogic{MP}_{n+1}^\lor+\neg\bflogic{MP}_n^\lor+\bflogic{wCT}_0+\bflogic{CT}_0!+\neg\bflogic{CT}_0$.
\item $\bflogic{MP}^\lor+\neg\bflogic{MP}+\bflogic{wCT}_0+\neg\bflogic{CT}_0!$.
\item $\bflogic{MP}+\bflogic{CT}_0$.
\end{enumerate}

Here, since $\bflogic{LLPO}_\om\to\neg\bflogic{wCT}_0$ holds intuitionistically, all of the above models validate $\neg\bflogic{LLPO}_\om$.
Moreover, each of the following items is true in some subtopos of the modified realizability topos.
\begin{enumerate}\setcounter{enumi}{4}
\item $\bflogic{LLPO}_{n+1}+\neg\bflogic{MP}_n^\lor+\bflogic{CT}_0!+\bflogic{CT}_0^{\rm maj}$.
\item $\bflogic{LLPO}+\neg\bflogic{MP}+\bflogic{wCT}_0!+\neg\bflogic{CT}_0!+\bflogic{CT}_0^{\rm maj}$.
\item $\bflogic{MP}+\bflogic{LLPO}_{n+1}+\neg\bflogic{LLPO}_{n}+\bflogic{CT}_0!$.
\end{enumerate}

Here, all of the above models validate $\neg\bflogic{wCT}_0$ and $\neg\bflogic{WLEM}_\om$.
Moreover, each of the following items is true in some subtopos of the modified realizability topos.
\begin{enumerate}\setcounter{enumi}{7}
\item $\bflogic{WLEM}_{n+1}+\neg\bflogic{MP}_n^\lor+\bflogic{CT}_0!+\bflogic{CT}_0^{\rm maj}$.
\item $\bflogic{WLEM}+\neg\bflogic{MP}+\neg\bflogic{CT}_0!+\bflogic{CT}_0^{\rm maj}$.
\item $\bflogic{MP}+\bflogic{LLPO}_{n}+\bflogic{WLEM}_{n+1}+\neg\bflogic{LLPO}_{n-1}+\neg\bflogic{WLEM}_{n}+\bflogic{CT}_0!$.
\item $\bflogic{MP}+\bflogic{LLPO}+\bflogic{WLEM}_{n+1}+\neg\bflogic{WLEM}_n+\bflogic{CT}_0!$.
\item $\bflogic{MP}_{\rm PR}+\bflogic{WLEM}+\neg\bflogic{LEM}$.
\end{enumerate}

Here, $\bflogic{MP}_{\rm PR}$ is a primitive recursive variant of $\bflogic{MP}$.

In addition, our models for (1), (5), (6), (8) and (9) also satisfy the independence of premise $\bflogic{IP}$ (see Section \ref{sec:inde-premi}).
While our model for (12) does not satisfy any of the variants of Church's thesis mentioned so far, it does satisfy some weakened principle $\bflogic{sCT}_0^{\rm maj}$ (every total relation has a lower semicomputable bound; see Section \ref{sec:major-Church-thesis}).
\end{theorem}

\begin{remark}
Each of the items (4), (7), (10), (11), and (12) can be satisfied even in a subtopos of the effective topos, but the others cannot (since they refute $\bflogic{MP}$).
\end{remark}

The proof of this theorem utilizes all of the key points A., B., C.\ and D.\ mentioned above.
\begin{enumerate}
\item[A.] 
An extra axiom in logic translates to an oracle in computability theory.
However, the hierarchy of the weak law of excluded middle $\bflogic{WLEM}_n$ does not fit within the framework of the standard notion of oracle.
This issue is resolved by an extended notion of oracle (Section \ref{sec:bilayer-computability}), which was recently discovered through the aforementioned abstract theory.
\item[B.] 
The analysis of the strength of oracles (i.e., the degree structure) is reflected in the analysis of the strength of LT topologies (Section \ref{sec:LT-topology}).
Our proof begins by establishing separations in the strength of oracles (Section \ref{sec:computable-separation}).
All separations are achieved using crystal-clear arguments familiar to computability theorists.
\item[C.]
Markov's principle $\bflogic{MP}$ is a partial computable principle (i.e., realizable by an unbounded search), so the oracle computability based on standard partial computability cannot distinguish the strength of variants $\bflogic{MP}_n^\lor$ of Markov's principle.
We achieve separations of variants of Markov's principle by newly introducing notions of oracle and reducibility based on total computability.
This new notion of reducibility is introduced automatically from the abstract theory (Sections \ref{sec:many-one-query-compu} and \ref{sec:main-many-query-reducibility}).
\item[D.] 
Each oracle induces an LT-topology, which in turn determines a sheaf subtopos (Section \ref{sec:LT-topology}).
Through this abstract framework, the separation of oracle strengths is translated into the construction of models separating logical principles (Section \ref{sec:main-realizability}).
\end{enumerate}

\subsection{Contents}

In Section \ref{sec:many-one-query-compu}, we introduce the notion of total Weihrauch reducibility, which a notion of one-query reducibility with respect to total computability.
We also reformulate various logical principles as computability-theoretic problems -- specifically, as choice problems -- and investigate the one-query reducibility relations among them.
In Section \ref{sec:main-many-query-reducibility}, we introduce the notion of total multi-query reducibility and investigate its fundamental properties.
In Section \ref{sec:bilayer-computability}, we generalize this idea to the setting of extended oracles, i.e., bilayer problems.
In Section \ref{sec:computable-separation}, we establish separations of strength for various selection problems with respect to total multi-query reducibility.
In Section \ref{sec:LT-topology}, we clarify the precise relationship between an oracle and a LT-topology, and between total computable realizability with multiple queries to an oracle and the semantics of the correspoding sheaf subtopos on the topos for total computability (the Grayson topos).
Section \ref{sec:main-realizability}, we then demonstrate the precise correspondence between logical principles and selection problems, thereby linking logic and computability.
By combining these results, we establish the main theorem in Section \ref{sec:proof-main}.

%
%
%
%

For the basics of computability theory, we refer the reader to \cite{OdiBook}.
For the basics of constructive mathematics, we refer the reader to \cite{BeeBook,HandbookCM}.
For the basics of realizability theory, we refer the reader to \cite{Tro73}.

In this article, we allow the use of classical principles for the meta-logic.
However, with careful argumentation, the use of classical principles can be kept to a necessary minimum.

%
%

\section{One-Query Computability}\label{sec:many-one-query-compu}

\subsection{Total Weihrauch reducibility}

The primary subject of study in computability theory is the analysis of the complexity of problems.
Here, a ``problem'' is actually a collection of problems, which is often displayed as a partial multifunction $F \pcolon X \tto Y$.
An input $x \in \mathrm{dom}(F)$ is called an {\em instance} of the problem $F$, and each $y \in F(x)$ is a solution to this instance.
That is, $F(x)$ is the solution set for the $x$th instance of the problem $F$.

In computability theory, we typically do not require information about the underlying type $X$, but only requires information about $\dom(F)$.
However, in the theory of total computability that we are dealing with, we require information about the type $X$ to define the notion of totality.
Therefore, we treat the notion of a problem as a data equipped with the type $X$ of potential instances.

We assume that all objects are encoded by natural numbers, so formally, we introduce the notion of a problem as follows.

\begin{definition}
A {\bf potted problem} is a triple $(X,Y; F)$ consisting of sets $X,Y\subseteq\om$ containing $0$ and a subset $F\subseteq X\times Y$.
\end{definition}

We use the following notations and terminologies:
\begin{itemize}
\item We write $F\pcolon X\tto Y$ for such a triple.
\item Such an $X$ is denoted by ${\rm dom}_+(F)$, which is called the {\em potential domain} (or the \ppp domain).
\item Such an $Y$ is denoted by ${\rm cod}_+(F)$, which is called the {\em potential codomain} (or the \ppp codomain)
\item For each $x\in X$, $F(x)$ is defined as $\{y\in Y:(x,y)\in F\}$.
\item $\dom(F)$ is defined as $\{x\in X:F(x)\not=\emptyset\}$, which is called the {\em actual domain} (or the \aaa domain).
\item If $F(x)$ is a singleton $\{y\}$, then we write $F(x)=y$.
\item If $F(x)$ is a singleton for any $x\in\dom(F)$, then we say that $F$ is {\em single-valued} and write $F\pcolon X\to Y$.
\end{itemize}

Note that under this formulation, $F$ is always a nonempty-valued.
(Although some approaches allow for empty-valued problems, this article does not deal with such problems.)
In later sections, we will frequently use the following notion to deal with the domain of a potted problem.

\begin{definition}\label{def:potted-set}
A {\bf potted set} is a pair $\paset{A}=(A_-;A_+)$ such that $A_-\subseteq A_+\subseteq \om$ and $0\in A_+$.
\end{definition}

This is a structure required to formalize the notion of a ``{\em total function on $A$}.''
The underlying set itself is $A_-$, but it is necessary to specify the scope of {\em totality}, and that is precisely $A_+$.
That is, we think of a potted set as a set $A_-$ equipped with the scope $A_+$ of totality.
This notion is inspired by Grayson's version of modified assembly/tripos \cite{vO97}.
The details will be explained in Section \ref{sec:LT-topology}.

\begin{notation}
If $\paset{X}=(X_-;X_+)$ is a potted set, we also use the notation $F\colon\paset{X}\tto Y$ for a potted problem $F$ such that $\dom_+(F)=X_+$ and $\dom(F)=X_-$.
\end{notation}


\begin{definition}\label{def:total-Weihrauch-reducibility}
Let $F$ and $G$ be potted problems.
We say that $F$ is {\bf total Weihrauch reducible to} $G$ if there are partial computable functions $\nu,\psi$ such that the following hold:
\begin{enumerate}
\item $\nu\colon{\rm dom}_+(F)\to{\rm dom}_+(G)$ is total.
\item $\psi\colon{\rm dom}_+(F)\times{\rm cod}_+(G)\to{\rm cod}_+(F)$ is total.
\item $x\in{\rm dom}(F)$ implies $\nu(x)\in{\rm dom}(G)$.
\item $x\in{\rm dom}(F)$ and $y\in G(\nu(x))$ imply $\psi(x,y)\in F(x)$.
\end{enumerate}

In this case, we write $F\leq_{\sf tW}G$.
We call $\nu$ an {\em inner reduction} and $\psi$ an {\em outer reduction}.
\end{definition}

The difference from the standard Weihrauch reduction \cite{BGP21} is that Definition \ref{def:total-Weihrauch-reducibility} requires totality on the \ppp domains/codomains.

\begin{prop}\label{prop:tW-preorder}
$\leq_{\sf tW}$ is a preorder.
\end{prop}

\begin{proof}
{\em Reflexivity.}
Consider $\nu(x)=x$ and $\psi(x,y)=y$.

{\em Transitivity.}
Assume $F\tW G\tW H$, where $F\tW G$ via $(\nu,\psi)$ and $G\tW H$ via $(\nu',\psi')$.
We need to construct a reduction from $F$ to $H$.

\begin{itemize}
\item {\em Inner reduction.} Consider $\nu''=\nu'\circ\nu$:
\begin{align*}
{\rm dom}_+(F)\xedge{\nu}{\rm dom}_+(G)\xedge{\nu'}{\rm dom}_+(H),
& &
{\rm dom}(F)\xedge{\nu}{\rm dom}(G)\xedge{\nu'}{\rm dom}(H)
\end{align*}
Clearly, $\nu'\circ\nu$ fulfills the conditions (1),(3) for an inner reduction.
\item {\em Outer reduction.}
We define a total function $\psi''\colon \dom_+(F)\times\cod_+(H)\to\cod_+(F)$:
For $(x,y)\in\dom_+(F)\times\cod_+(H)$ we define
\[\psi''(x,y)=\psi(x,\psi'(\nu(x),y)).\]
By our assumption, we have $\nu(x)\in\dom_+(G)$ and $y':=\psi'(\nu(x),y)\in\cod_+(G)$; hence, we get $\psi(x,y')\in\cod_+(F)$.
Therefore. $\psi''$ fulfills the totality condition (2) for an outer reduction.
\item {\em Verification.}
For (4), Let $x\in\dom(F)$ and $y\in H(\nu''(x))$ be given.
By our assumption, we have $\nu(x)\in\dom(G)$.
\begin{itemize}
\item As $G\tW H$ via $(\nu',\psi')$, combining $\nu(x)\in\dom(G)$ and $y\in H(\nu'(\nu(x)))=H(\nu''(x))$, we get $\psi'(\nu(x),y)\in G(\nu(x))$.
\item As $F\tW G$ via $(\nu,\psi)$, combining $x\in\dom(F)$ and $\psi'(\nu(x),y)\in G(\nu(x))$, we get $\psi''(x,y)=\psi(x,\psi'(\nu(x),y))\in F(x)$.
\end{itemize}
Hence, we verified the condition (4), which means that $(\nu'',\psi'')$ witnesses $F\tW H$.
\end{itemize}
\end{proof}

\subsection{Weihrauch Problems}\label{sec:Weihrauch-problem}

We investigate the total Weihrauch strength of various potted problems.
We use the following notations:
\begin{itemize}
\item Let $\varphi_e$ denote the $e$th partial computable function on $\om$.
\item Let $S_e$ be the $\Sigma_1$ set coded by $e$; that is,
$S_e=\{n\in\om:\varphi_e(n)\downarrow\}$.
\item Let $P_e$ be the $\Pi_1$ set coded by $e$; that is,
$P_e=\{n\in\om:\varphi_e(n)\uparrow\}$.
\end{itemize}

\subsubsection{$\Sigma_1$-Choice}
A {\em $\Sigma_1$-choice} problem is a problem in which, given a code for a non-empty $\Sigma_1$ set $S_e$, one is asked to pick up an element $n\in S_e$.
A $\Sigma_1$-choice problem is computably solvable by an {\em unbounded search}, and is therefore trivial in the context of the standard Weihrauch reducibility.
However, since an unbounded search is not total computable, $\Sigma_1$-choice becomes non-trivial in the context of total Weihrauch reducibility.

\begin{definition}
For $0<k< n\leq\om$, the $\Sigma_1$-choice problem $\sii{\sf C}_{k/n}$ is defined as follows:
\begin{itemize}
\item $\dom_+(\sii{\sf C}_{k/n})=\om$, and $\cod_+(\sii{\sf C}_{k/n})=n$.
\item $\dom(\sii{\sf C}_{k/n})=\{e\in\om:|S_e|=k\}$.
\item $\sii{\sf C}_{k/n}(e)=\{k<n:k\in S_e\}$.
\end{itemize}
\end{definition}

\begin{definition}
For $0<k<n\leq\om$, the $\Sigma_1$-choice problem $\sii{\sf C}_{\geq k/n}$ is defined as follows:
\begin{itemize}
\item $\dom_+(\sii{\sf C}_{\geq k/n})=\om$, and $\cod_+(\sii{\sf C}_{k/n})=n$.
\item $\dom(\sii{\sf C}_{\geq k/n})=\{e\in\om:|S_e|\geq k\}$.
\item $\sii{\sf C}_{\geq k/n}(e)=\{k<n:k\in S_e\}$.
\end{itemize}
\end{definition}

In words: There are $n$ options, exactly/at least $k$ of which are correct.
Given a code of an enumeration of correct elements, the problem asks us to pick up one of them.

One may define $\pii{\sf C}_{k/n}$ and $\pii{\sf C}_{\geq k/n}$ in the same manner.
We call ${\sf C}_{\geq 1/n}$ the {\em choice} problem, ${\sf C}_{1/n}$ the {\em unique choice} problem, ${\sf C}_{n-1/n}$ the {\em co-unique choice} problem, and ${\sf C}_{\geq n-1/n}$ the {\em all-or-co-unique} problem.

With a slight modification to the definition, the notion of co-unique choice also makes sense for $n=\omega$.

\begin{definition}
For $0<k<n\leq\om$, the $\Sigma_1$-choice problem $\sii{\sf C}_{n-k/n}$ is defined as follows:
\begin{itemize}
\item $\dom_+(\sii{\sf C}_{n-k/n})=\om$, and $\cod_+(\sii{\sf C}_{n-k/n})=n$.
\item $\dom(\sii{\sf C}_{n-k/n})=\{e\in\om:|n\setminus S_e|=k\}$.
\item $\sii{\sf C}_{n-k/n}(e)=\{k<n:k\in S_e\}$.
\end{itemize}
\end{definition}

\begin{definition}
For $0<k<n\leq\om$, the $\Sigma_1$-choice problem $\sii{\sf C}_{\geq n-k/n}$ is defined as follows:
\begin{itemize}
\item $\dom_+(\sii{\sf C}_{\geq n-k/n})=\om$, and $\cod_+(\sii{\sf C}_{\geq n-k/n})=n$.
\item $\dom(\sii{\sf C}_{\geq n-k/n})=\{e\in\om:|n\setminus S_e|\leq k\}$.
\item $\sii{\sf C}_{\geq n-k/n}(e)=\{k<n:k\in S_e\}$.
\end{itemize}
\end{definition}

For instance, ${\sf C}_{\om-1/\om}$ stands for the co-unique choice on $\om$.

\begin{obs}\label{obs:basic-relations-choice}
Let $\Gamma\in\{\Sigma_1,\Pi_1\}$.
For $\ell\leq k$ and $n\leq m$:
\begin{align*}
\Gamma\text{-}{\sf C}_{k/n}\leq_{\sf tW}\Gamma\text{-}{\sf C}_{\ell/m};
& &
\Gamma\text{-}{\sf C}_{\geq k/n}\leq_{\sf tW}\Gamma\text{-}{\sf C}_{\geq \ell/m};
& &
\Gamma\text{-}{\sf C}_{k/n}\leq_{\sf tW}\Gamma\text{-}{\sf C}_{\geq k/n}.
\end{align*}
\end{obs}

In particular, the following relationship generally holds among the types of choice problems (for a fixed number of options):
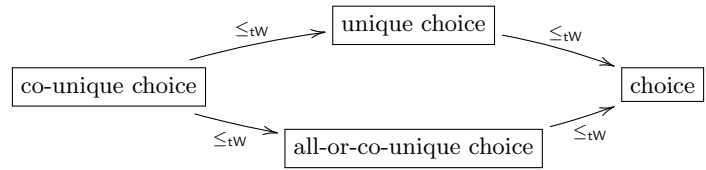
\[
\small
\xymatrix@R=0.1cm{
& \fbox{\text{unique choice}} \ar@/^0.5em/[dr]^-{\leq_{\sf tW}} & \\
\fbox{\text{co-unique choice}} \ar@/^0.5em/[ur]^-{\leq_{\sf tW}}\ar@/_0.5em/[dr]_-{\leq_{\sf tW}}& & \fbox{\text{choice}} \\
& \fbox{\text{all-or-co-unique choice}}\ar@/_0.5em/[ur]_-{\leq_{\sf tW}}&
}
\]

Note here that the direction of the arrows is the reverse of that for the logical entailment relationship; that is, ``choice'' is the strongest, and ``co-unique choice'' is the weakest.

We show that the $\Sigma_1$-choice problem and the $\Sigma_1$-unique-choice problem have the same strength.

\begin{prop}\label{prop:basic-Sigma-1-choice}
$\Sigma_{1}\text{-}{\sf C}_{1/n}\equiv_{\sf tW}\Sigma_1\text{-}{\sf C}_{\geq 1/n}$ for any $n\leq \om$.
\end{prop}

\begin{proof}
$\sii{\sf C}_{1/n}\tW\sii{\sf C}_{\geq 1/n}$: By Observation \ref{obs:basic-relations-choice}.

\medskip
\noindent
$\sii{\sf C}_{\geq 1/n}\tW\sii{\sf C}_{1/n}$:
We only deal with the case $n=\om$.
Define total reductions $\nu,\psi$ as follows:
\begin{itemize}
\item {\em Inner reduction.}
For each $e\in\om$, consider the following algorithm $\nu(e)$:
Wait for $\varphi_e(k)\down$ for some $k<\om$.
Declare $\varphi_{\nu(e)}(k)\down$ only for the first found such $k$.
Note that, even if $\varphi_{\nu(e)}(k)$ is undefined, the algorithm $\nu(e)$ itself is defined, so $\nu$ is total.
\item {\em Outer reduction.}
For each $e,k\in\om$, put $\psi(e,k)=k$.
Clearly, $\psi$ is total.
\end{itemize}

{\em Verification:}
\begin{itemize}
\item  {\em Inner reduction.}
For each instance $e\in\dom(\sii{\sf C}_{\geq 1/\om})$, the $\Sigma_1$ set $S_e=\{k\in\om:\varphi_{e}(k)\downarrow\}\subseteq\om$ is nonempty.
Then, our definition ensures $\varphi_{\nu(e)}(k)\down$ for a unique $k$, so $S_{\nu(e)}\subseteq S_e$ is a singleton.
Therefore, we have $\nu(e)\in\dom(\sii{\sf C}_{1/\om})$.
\item {\em Outer reduction.}
For each solution $k\in \sii{\sf C}_{1/\om}(\nu(e))=S_{\nu(e)}$, since $S_{\nu(e)}\subseteq S_e$, we get $k\in S_e$.
Therefore, $\psi(e,k)=k\in S_e=\sii{\sf C}_\om(e)$.
\end{itemize}
\end{proof}

We show that the $\Pi_1$-co-unique choice problem, the $\Sigma_1$-co-unique choice problem, and the $\Sigma_1$-all-or-co-unique choice problem have the same strength.

\begin{theorem}\label{thm:basic-Sigma-1-counique}
$\pii{\sf C}_{n-1/n}\equiv_{\sf tW}\Sigma_1\text{-}{\sf C}_{n-1/n}\equiv_{\sf tW}\sii{\sf C}_{\geq n-1/n}$ for any $n\leq\om$.
\end{theorem}

\begin{proof}
$\pii{\sf C}_{n-1/n}\tW\sii{\sf C}_{n-1/n}$:
\begin{itemize}
\item {\em Inner reduction.}
Consider the following algorithm $\nu(e)$:
Wait for $\varphi_e(k)\down$ for some $k<n$.
For all $j<n$ except for the first found $k<n$, declare $\varphi_{\nu(e)}(j)\down$.
Note that the algorithm $\nu(e)$ itself is always defined, so $\nu$ is total.
\item {\em Outer reduction.}
Given $e,v\in\om$, put $\psi(e,v)=v$.
Clearly, $\psi$ is total.
\end{itemize}

{\em Verification:}
\begin{itemize}
\item {\em Inner reduction.}
For each instance $e\in\dom(\pii{\sf C}_{n-1/n})$, $P_e$ is a co-singleton.
Then there is $k<n$ such that $P_e=\om\setminus\{k\}$; that is, $\varphi_e(k)\down$.
By the definition of $\nu(e)$, we get $S_{\nu(e)}=\{j<n:\varphi_{\nu(e)}(j)\down\}=\om\setminus\{k\}=P_e$.
In particular, $S_{\nu(e)}$ is a co-singleton, so we have $\nu(e)\in\dom(\sii{\sf C}_{n-1/n})$.
\item {\em Outer reduction.}
For each solution $v\in \sii{\sf C}_{n-1/n}(\nu(e))=S_{\nu(e)}$, we have $\psi(e,v)=v\in S_{\nu(e)}=P_e$.
\end{itemize}

\noindent
$\sii{\sf C}_{n-1/n}\tW\sii{\sf C}_{\geq n-1/n}$:
By Observation \ref{obs:basic-relations-choice}.

\medskip
\noindent
$\sii{\sf C}_{\geq n-1/n}\tW\pii{\sf C}_{n-1/n}$ for $n<\om$:
\begin{itemize}
\item {\em Inner reduction.}
Consider the following algorithm $\nu(e)$:
Wait until seeing $\varphi_e(j)\down$ for all $j<n$ except for one input.
If this happens, for the unique $k$ for which $\varphi_e(k)$ is still undefined, declare $\varphi_{\nu(e)}(k)\down$.
For all other inputs $k<n$, $\varphi_{\nu(e)}(k)$ does not halt.
Again, the algorithm $\nu(e)$ itself is always defined, so $\nu$ is total.
\item {\em Outer reduction.}
Given $e,v\in\om$, put $\psi(e,v)=v$.
Clearly, $\psi$ is total.
\end{itemize}

{\em Verification:}
\begin{itemize}
\item {\em Inner reduction.}
For each instance $e\in\dom(\sii{\sf C}_{\geq n-1/n})$, we have $|n\setminus S_e|\leq 1$.
Hence, once we have $j\in S_e$ for all $j<n$ except for one value of $k<n$, we obtain $P_{\nu(e)}=\{j<n:\varphi_{\nu(e)}(j)\upar\}=\om\setminus\{k\}\subseteq S_e$.
In particular, $P_{\nu(e)}$ is a co-singleton, so we have $\nu(e)\in\dom(\pii{\sf C}_{n-1/n})$.
\item {\em Outer reduction.}
For each solution $v\in \pii{\sf C}_{n-1/n}(\nu(e))=P_{\nu(e)}$, we have $\psi(e,v)=v\in S_{\nu(e)}=P_e$.
\end{itemize}

\medskip
\noindent
$\sii{\sf C}_{\geq\om-1/\om}\tW\pii{\sf C}_{\om-1/\om}$:
In the case of $\om$, a slightly different argument is required.
Consider the following total reductions:
\begin{itemize}
\item {\em Inner reduction.}
Consider the following algorithm $\nu(e)$:
Wait for a stage $t$ such that some $k$ is enumerated into $S_e$.
Let $t$ be the first stage.
Then wait until all elements less than $t$, except for at most one element $u<t$, are enumerated into $S_e$.
If this happens, we remove $u$ from $P_{\nu(e)}$; that is, $P_{\nu(e)}=\om\setminus\{u\}$.
If this never happens, we have $P_{\nu(e)}=\om$.
The algorithm $\nu(e)$ is explicitly defined, so $\nu$ is total.
\item {\em Outer reduction.}
For each $e,v\in\om$, if no elements have been enumerated into $S_e$ by stage $v$, then put $\psi(e,v)=v$.
Otherwise, for the first such element $k\in S_e$, put $\psi(e,v)=k$.
Since the value $\psi(e,v)$ is always determined, $\psi$ is total.
\end{itemize}

{\em Verification:}
\begin{itemize}
\item {\em Inner reduction.}
For each instance $e\in\dom(\sii{\sf C}_{\geq \om-1/\om})$, some element $k$ must be enumerated into $S_e$ at some stage $t$.
Moreover, all elements less than $t$, except for at most one element $u<t$, are enumerated into $S_e$ by some stage (it is possible that $u$ will be enumerated into $S_e$ at a later stage).
Then our construction ensures $P_{\nu(e)}=\om\setminus\{u\}$, so $P_{\nu(e)}$ is a co-singleton.
Hence. we have $\nu(e)\in\dom(\pii{\sf C}_{\om-1/\om})$.
\item {\em Outer reduction.}
For each solution $v\in \pii{\sf C}_{\om-1/\om}(\nu(e))=P_{\nu(e)}$, if $v\geq t$ then we have $\psi(e,v)=k\in S_e$.
For $v<t$, if $v\in P_{\nu(e)}$ then $v\not=u$, and by our choice of $u$, $v$ is enumerated into $S_e$ before $u$.
Hence, $\psi(e,v)=v\in S_e$.
In any case, we get $\psi(e,v)\in S_e=\sii{\sf C}_{\geq\om-1/\om}(e)$.
\end{itemize}
\end{proof}

%
%

All of the $\Sigma_1$-choice problems discussed so far are computably solvable using unbounded search.
Therefore, the complexity of these choice problems is all trivial in terms of Weihrauch degrees.
Interestingly, when considering the ${\sf tW}$-degrees, we will get the following sequence of $\Sigma_1$-choice problems:
\begin{align}\label{equ:hierarchy-choice}
\Sigma_1\text{-}{\sf C}_{\om-1/\om}<_{\sf tW}\cdots&<_{\sf tW}\Sigma_1\text{-}{\sf C}_{n/n+1}<_{\sf tW}\Sigma_1\text{-}{\sf C}_{n-1/n}<_{\sf tW}\cdots<_{\sf tW}\Sigma_1\text{-}{\sf C}_{1/2}<_{\sf tW}\cdots \tag{$\star$} \\
\cdots&<_{\sf tW}\cdots<_{\sf tW}\Sigma_1\text{-}{\sf C}_{1/n}<_{\sf tW}\Sigma_1\text{-}{\sf C}_{1/n+1}<_{\sf tW}\cdots<_{\sf tW}\Sigma_1\text{-}{\sf C}_{1/\om}.\notag
\end{align}

Below, we show the strictness of the hierarchy of the $\Sigma_1$-unique-choice problems $\sii{\sf C}_{1/n}$.
We postpone the separation of the hierarchy of the $\Sigma_1$-co-unique-choice problems $\sii{\sf C}_{n-1/n}$ until Section \ref{sec:computable-separation}, where we provide a separation in a stronger form.

\begin{theorem}\label{thm:UC-hierarchy}
$\Sigma_1\text{-}{\sf C}_{1/n}<_{\sf tW}\Sigma_1\text{-}{\sf C}_{1/n+1}$.
\end{theorem}

\begin{proof}
Assume $\sii{\sf C}_{1/n+1}\tW\sii{\sf C}_{1/n}$ via $(\nu,\psi)$.
By using the recursion theorem, construct the following algorithm $e$ involving a self-reference:
\begin{enumerate}
\item Wait for $\psi(e,i)\down<n+1$ for all $i<n$.
\item If this happens, by the pigeonhole principle, there is $k<n+1$ such that $k\not=\psi(e,i)$ for any $i<n$.
In this case, declare $\varphi_e(k)\down$, where $\varphi_e$ does not terminate for other inputs.
In other words, $S_{e}=\{k\}$.
\end{enumerate}

{\em Verification:}
\begin{itemize}
\item As $e\in\om=\dom_+(\sii{\sf C}_{1/n+1})$ and $n=\cod_+(\sii{\sf C}_{1/n})$, by totality of $\psi\colon\om\times n\to n+1$, the computation of $\psi(e,i)$ terminates for any $i<n$.
\item Hence, the algorithm $e$ must reach at (2), so $S_e$ becomes a singleton, which means $e\in\dom(\sii{\sf C}_{1/n+1})$.
As $\nu$ is an inner reduction, we get $\nu(e)\in\dom(\sii{\sf C}_{1/n})$; that is, $S_{\nu(e)}=\{i\}$ for some $i<n$.
\item As $\psi$ is an outer reduction, for each solution $i\in S_{\nu(e)}$, we must have $\psi(e,i)\in S_e$; however, our construction ensures $\psi(e,i)\not=k$ and $S_e=\{k\}$, which is impossible.
\end{itemize}
\end{proof}

\subsubsection{Markov's principle}

{\em Markov's principle} states that ``if it is impossible for a computation not to halt, then the computation halts,'' and
this principle enables us to perform an unbounded search.
Logically, it corresponds to the {\em double negation elimination for $\Sigma_1$-formulas} (since the termination of a computation is $\Sigma_1$-complete).
We consider the following potted problems related Markov's principle:

\begin{definition}
The problem ${\sf MP}_{\sf PR}$ (Markov's principle for computations) is defined as follows:
\begin{itemize}
\item ${\rm dom}_+({\sf MP}_{\sf PR})=\om$, and ${\rm cod}_+({\sf MP}_{\sf PR})=\om$.
\item ${\rm dom}({\sf MP}_{\sf PR})=\{e\in\om:\varphi_e(0)\downarrow\}$.
\item ${\sf MP}_{\sf PR}(e)=\{s\in\om:\varphi_e(0)[s]\downarrow\}$
\end{itemize}
\end{definition}

\begin{definition}
The problem ${\sf MP}$ (Markov's principle for functions) is defined as follows:
\begin{itemize}
\item ${\rm dom}_+({\sf MP})=\{e\in\om:\varphi_e\mbox{ is total and two-valued}\}$, and ${\rm cod}_+({\sf MP})=\om$.
\item ${\rm dom}({\sf MP})=\{e\in{\rm dom}_+({\sf MP}):\exists n.\ \varphi_e(n)=1\}$.
\item ${\sf MP}(e)=\{n\in\om:\varphi_e(n)=1\}$
\end{itemize}
\end{definition}

\begin{prop}\label{prop:MP-MPPR-equivalent}
${\sf MP}\equiv_{\sf tW}{\sf MP}_{\sf PR}$.
\end{prop}

\begin{proof}
($\geq_{\sf tW}$)
Define total reductions $\nu,\psi$ as follows:
\begin{itemize}
\item {\em Inner reduction.}
Each index $e\in\om$ is transformed into the following primitive recursive function:
\[
\varphi_{\nu(e)}(n)=
\begin{cases}
1&\mbox{ if }\varphi_e(0)[n]\downarrow\\
0&\mbox{ otherwise.}
\end{cases}
\]
As $\varphi_{\nu(e)}$ is total and two-valued, we have $\nu(e)\in{\rm dom}_+({\sf MP})$.
In particular, $\nu\colon\dom_+({\sf MP}_{\sf PR})\to\dom_+({\sf MP})$ is total.
\item {\em Outer reduction.}
For each $e,s\in\om$, put $\psi(e,s)=s$.
Clearly, this is total.
\end{itemize}

{\em Verification:}
\begin{itemize}
\item {\em Inner reduction.}
For each instance $e\in{\rm dom}_+({\sf MP}_{\sf PR})$, we have $\varphi_e(0)\down$, which means $\varphi_e(0)[s]\down$ for some $s$; thus $\varphi_{\nu(e)}(s)=1$.
Therefore, $\nu(e)\in{\rm dom}_+({\sf MP})$.
\item {\em Outer reduction.}
For each solution $s\in{\sf MP}(\nu(e))$, we have $\varphi_{\nu(e)}(s)=1$; in particular, $\varphi_e(0)[s]\down$.
Hence, we get $s\in{\sf MP}_{\sf PR}(e)$.
\end{itemize}

($\leq_{\sf tW}$)
Define total reductions $\nu,\psi$ as follows:
\begin{itemize}
\item {\em Inner reduction.}
For each $e\in\mathrm{dom}(\sf MP)$, it is an index of a total function.
Consider the following algorithm $\nu(e)$:
$\varphi_{\nu(e)}(0)$ searches for $n$ such that $\varphi_e(n)=1$, and then outputs the least such $n$.
This always gives some algorithm $\nu(e)$, so $\nu$ is total.
\item {\em Outer reduction.}
For each $e,s$, simulate the computation of $\varphi_{\nu(e)}(0)$ up to stage $s$; if the computation has halted, let $\psi(e,s)$ be the output value, otherwise put $\psi(e,s)=0$.
Then, $\psi$ is total.
\end{itemize}

{\em Verification.}
\begin{itemize}
\item {\em Inner reduction.} For each instance $e\in{\rm dom}({\sf MP})$, we have $\varphi_e(n)=1$ for some $n$.
Then, by our construction, $\varphi_{\nu(e)}(0)$ eventually terminates the computation,  which means $\nu(e)\in{\rm dom}({\sf MP}_{\sf PR})$.
\item {\em Outer reduction.}
For each solution $s\in{\sf MP}_{\sf PR}(\nu(e))$, $\varphi_{\nu(e)}(0)$ terminates the computation by stage $s$; hence, one can find an $n$ such that $\varphi_e(n)=1$ by that stage, and thus $\psi(e,s)=n$.
Therefore, $\psi(e,s)\in{\sf MP}(e)$.
\end{itemize}
\end{proof}

%

In our context, Markov's principle corresponds to the $\Sigma_1$-choice problem (which is equivalent to the $\Sigma_1$-unique-choice problem by Proposition \ref{prop:basic-Sigma-1-choice}) on $\om$.
Thus, Markov's principle is the hardest problem in the hierarchy of $\Sigma_1$-choice problems (\ref{equ:hierarchy-choice})

\begin{prop}\label{prop:MP-as-Sigma-1-choice}
${\sf MP}\equiv_{\sf tW}\Sigma_{1}\text{-}{\sf C}_{1/\om}$.
\end{prop}

\begin{proof}
By Propositions \ref{prop:basic-Sigma-1-choice} and \ref{prop:MP-MPPR-equivalent}, it suffices to show ${\sf MP}_{\sf PR}\leq_{\sf tW}\Sigma_1\text{-}{\sf C}_{\geq 1/\om}\leq_{\sf tW}{\sf MP}$.

\medskip
\noindent
${\sf MP}_{\sf PR}\leq_{\sf tW}\Sigma_1\text{-}{\sf C}_{\geq 1/\om}$:
\begin{itemize}
\item {\em Inner reduction.}
Given $e\in\om$, let $\nu(e)$ be an index of the $\Sigma_1$ set $\{s\in\om:\varphi_e(0)[s]\downarrow\}$.
Then $\nu$ is total.
\item {\em Outer reduction.}
For each $e,s$, put $\psi(e,s)=s$.
Clearly, this is total.
\end{itemize}

Obviously, $(\nu,\psi)$ gives the actual reduction.

\medskip
\noindent
$\sii{\sf C}_{\geq 1/\om}\leq_{\sf tW}{\sf MP}$:
\begin{itemize}
\item {\em Inner reduction.} Given $e\in\om$, define $\varphi_{\nu(e)}$ as follows:
\[
\varphi_{\nu(e)}(\langle n,s\rangle)=
\begin{cases}
1&\mbox{ if }\varphi_e(n)[s]\down\\
0&\mbox{ otherwise.}
\end{cases}
\]
As $\varphi_{\nu(e)}$ is total and two-valued, we have $\nu(e)\in\dom_+({\sf MP})$.
In particular, $\nu\colon\dom_+(\sii{\sf C}_{\geq 1/\om})\to\dom_+({\sf MP})$ is total.
\item {\em Outer reduction.} Put $\psi(e,\pair{n,s})=n$. Then $\psi$ is total.
\end{itemize}

{\em Verification:}
\begin{itemize}
\item {\em Inner reduction.}
For each instance $e\in\dom(\sii{\sf C}_{\geq 1/\om})$, we have $S_e=\{n\in\om:\varphi_e(n)\down\}\not=\emptyset$, which means $\varphi_e(n)\down$ for some $n$; in particular, $\varphi_e(n)[s]\down$ for some $s$.
Thus, we get $\nu(e)\in\dom({\sf MP})$.
\item {\em Outer reduction.}
For each solution $\pair{n,s}\in{\sf MP}(\nu(e))$, we have $\varphi_{\nu(e)}(\pair{n,s})=1$.
By our construction, this means $\varphi_e(n)\down$; hence, $\psi(e,\pair{n,s})=n\in S_e$.
\end{itemize}
\end{proof}

A weakened version of Markov's principle is known as {\em disjunctive Markov's principle}, which corresponds to the {\em De Morgan's law for $\Pi_1$ formulas}; see e.g.~\cite{FuKo18}.
There is a hierarchy ${\sf MP}_n^\lor$ of disjunctive Markov's principle in constructive mathematics.


\begin{definition}
For each $n\leq\om$, ${\sf MP}_n^\lor$ is defined as follows:
\begin{itemize}
\item ${\rm dom}_+({\sf MP}_n^\lor)=\{e\in\om:\forall k<n\forall t\in\om.\ \varphi_e(k,t)\down<2\}$, and ${\rm cod}_+({\sf MP}_n^\lor)=n$.
\item ${\rm dom}({\sf MP}_n^\lor)=\{e\in{\rm dom}_+({\sf MP}_n^\lor):\exists! k<n\exists t.\ \varphi_e(k,t)=1\}$.
\item ${\sf MP}_n^\lor(e)=\{k<n:\forall t.\ \varphi_e(k,t)=0\}$
\end{itemize}
\end{definition}

In our context, disjunctive Markov's principle corresponds to the $\Sigma_1$-co-unique choice problem (which is equivalent to the $\Pi_1$-co-unique choice problem and the $\Sigma_1$-all-or-co-unique choice problem by Theorem \ref{thm:basic-Sigma-1-counique}).

\begin{theorem}\label{thm:disjunctive-Markov-counique}
${\sf MP}_n^\lor\equiv_{\sf tW}\Sigma_1\text{-}{\sf C}_{n-1/n}$ for any $n\leq\om$.
\end{theorem}

\begin{proof}
By Theorem \ref{thm:basic-Sigma-1-counique}, it suffices to show ${\sf MP}_n^\lor\equiv_{\sf tW}\pii{\sf C}_{n-1/n}$, but this is obvious.
\end{proof}

\subsubsection{Lessor Limited Principle of Omniscience LLPO}

Some of $\Pi_1$-choice problems are computably unsolvable.
In classical computability theory, a vast amount of research has been conducted on this topic.
In the context of Weihrauch reducibility, $\Pi_1\text{-}{\sf C}_{\geq k/n}$ is studied under the name ${\sf ACC}_n$, which stands for ``all-or-counique choice''.
Here, it is linked to Richman's ${\sf LLPO}_n$ hierarchy \cite{Ri90}.

\begin{definition}
For each $n\leq\om$, ${\sf LLPO}_n$ is defined as follows:
\begin{itemize}
\item ${\rm dom}_+({\sf LLPO}_n)=\{e\in\om:\forall i<n\forall k\in\om.\ \varphi_e(i,k)\down<2\}$, and ${\rm cod}_+({\sf LLPO}_{n}^\lor)=n$.
\item ${\rm dom}({\sf LLPO}_n)=\{e\in{\rm dom}_+({\sf LLPO}_n):\exists t.\ \varphi_e(k,t)=1\mbox{ for at most one $k<n$}\}$.
\item ${\sf LLPO}_n(e)=\{k<n:\forall t.\ \varphi_e(t)=0\}$
\end{itemize}

Here, we define ${\sf LLPO}={\sf LLPO}_2$.
\end{definition}

\begin{prop}\label{prop:LLPO-AcoUC}
${\sf LLPO}_n\equiv_{\sf tW}\Pi_1\text{-}{\sf C}_{\geq n-1/n}$ for any $n\leq\om$.
\end{prop}

\begin{proof}
Obvious.
\end{proof}

\begin{cor}\label{obs:MP-reducible-to-LLPO}
${\sf MP}_n^\lor\leq_{\sf tW}{\sf LLPO}_n$.
\end{cor}

%

\begin{cor}\label{cor:LLPO-eq-C-2}
${\sf LLPO}\equiv_{\sf tW}\pii{\sf C}_{\geq 1/2}$.
\end{cor}

An important property of ${\sf LLPO}_n$ is that it is totalizable in the following sense:

\begin{lemma}\label{lem:LLPO-totalizable}
For each $n\leq\om$,
${\sf LLPO}_n$ is total Weihrauch equivalent to a total problem $L_n\colon\om\tto\om$.
\end{lemma}

\begin{proof}
For a given algorithm $e$, we say that $k<n$ is the first terminating input for $\varphi_e$ if there is $s$ such that $\varphi_e(k)[s]\downarrow$ and $\varphi_e(\ell)[t]\uparrow$ for any $\ell<n$ and $\pair{\ell,t}<\pair{k,s}$.
Then define $L_n$ as follows:
\begin{itemize}
\item $\dom(L_n)=\dom_+(L_n)=\cod_+(L_n)=\om$.
\item $L_n(e)=\{k<n:k\mbox{ is not the first terminating input for $\varphi_e$}\}$.
\end{itemize}

\noindent
$\pii{\sf C}_{\geq n-1/n}\tW L_n$:
\begin{itemize}
\item {\em Reductions.}
Put $\nu(e)=e$ and $\psi(e,v)=v$.
Clearly, $\nu$ and $\psi$ are total.
\end{itemize}

{\em Verification:}
\begin{itemize}
\item {\em Inner reduction.}
Regardless of whether $e\in\dom(\pii{\sf C}_{\geq n-1/n})$ or not, $\nu(e)=e\in\dom(L_n)=\om$.
\item {\em Outer reduction.}
Given an instance $e\in\dom(\pii{\sf C}_{\geq n-1/n})$ and a solution $v\in L_n(\nu(e))$, if $P_e=n$ then, as $v<n$, we clearly have $\psi(e,v)=v\in P_e$.
If $P_e\not=n$ then, since $P_e$ is a co-singleton, there is a unique $k<n$ such that $P_e=n\setminus\{k\}$; thus, $\varphi_e(k)\down$.
By uniqueness, such an $k<n$ must be the first terminating input for $\varphi_e$.
By definition., $v\in L_n(e)$ is not the first terminating input; hence, $v\not=k$.
Therefore, we get $\psi(e,v)=v\in n\setminus\{k\}=P_e=\pii{\sf C}_{\geq n-1/n}(e)$.
\end{itemize}

\noindent
$L_n\tW\pii{\sf C}_{\geq n-1/n}$:
\begin{itemize}
\item {\em Inner reduction.}
Let $\nu(e)$ be the following algorithm:
Wait for $\varphi_e(k)\down$ for some $k<n$.
For the least such $k<n$, declare $\varphi_{\nu(e)}(k)\down$.
For other inputs, the computation does not halt.
In other words, declare $\varphi_{\nu(e)}(k)\down$ only for the first terminating input $k<n$
Clearly, $\nu$ is total.
\item {\em Outer reduction.}
For each $e,v\in\om$, put $\psi(e,v)=v$.
Clearly, $\psi$ is total.
\end{itemize}

{\em Verification:}
\begin{itemize}
\item {\em Inner reduction.}
For each $e\in\om$, we have declared $\varphi_{\nu(e)}(k)\down$ for at most one $k<n$.
This means $|n\setminus P_{\nu(e)}|\leq 1$; thus, $n\in\dom(\pii{\sf C}_{\geq n-1/n})$.
\item {\em Outer reduction.}
For each solution $v\in \pii{\sf C}_{\geq n-1/n}(e)=P_{\nu(e)}$, if $P_{\nu(e)}=n$ then there is no terminating input $k>n$ for $\varphi_e$. which means $v\in L_n(e)$.
If $P_{\nu(e)}\not=n$ then we have $P_{\nu(e)}=n\setminus\{k\}$, where $k$ must be the first terminating input for $\varphi_e$.
Since $v\in P_{\nu(e)}$, we have $v\not=k$; hence, $v$ is not the first terminating input for $\varphi_e$.
This means $v\in L_n(e)$.
In any case, we get $\psi(e,v)=v\in L_n(e)$.
\end{itemize}
\end{proof}

The results in Section \ref{sec:Weihrauch-problem} can be summarized as follows:
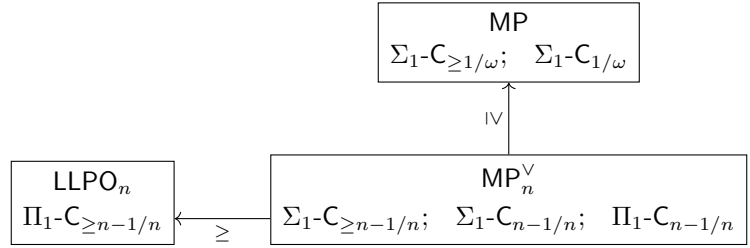
\[
\begin{tikzcd}[row sep=large, column sep=large]
&
|[draw]|
{\shortstack{
${\sf MP}$\\[0.2em]
$\sii{\sf C}_{\geq 1/\om}$;\quad
$\sii{\sf C}_{1/\om}$
}}
\\
|[draw]|
{\shortstack{
${\sf LLPO}_n$\\[0.2em]
$\pii{\sf C}_{\geq n-1/n}$
}}
&
|[draw]|
{\shortstack{
${\sf MP}_n^\lor$\\[0.2em]
$\sii{\sf C}_{\geq n-1/n}$;\quad
$\sii{\sf C}_{n-1/n}$;\quad
$\pii{\sf C}_{n-1/n}$
}}
\arrow[l, "\geq"]
\arrow[u, "\rotatebox{-90}{$\geq$}"]
\end{tikzcd}
\]

\subsection{Join of problems}
%
%

%
%

The Weihrauch degrees form a lattice, and in particular, the join (the least upper bound) of two problems is given by the sum operation \cite{BGP21}.

\begin{definition}
Let $F\pcolon X\tto Y$ and $G\pcolon Z\tto W$ be potted problems.
Then define the sum $F+G\pcolon X+Z\tto Y+W$ as the following potted problem:
\begin{itemize}
\item $\dom_+(F+G)=X+Y=\{\pair{0,x}:x\in X\}\cup\{\pair{1,y}:y\in Y\}$.
\item $\cod_+(F+G)=Z+W$.
\item $\dom(F+G)=\dom(F)+\dom(G)$.
\item $(F+G)(0,x)=\pair{0,F(x)}$ and $(F+G)(1,y)=\pair{1,G(y)}$.
\end{itemize}
\end{definition}

Note that $(F+G)(i,x)$ is of the form $\pair{i,z}$, but since this notation can often be cumbersome, this is sometimes simply written as $z$.

\begin{prop}\label{prop:sum-join}
Let $F,G$ be potted problems.
Then $F+G$ is the least upper bound of $F$ and $G$ with respect to $\leq_{\sf tW}$.
\end{prop}

\begin{proof}
{\em Upper bound:}
Let $F_i\pcolon X_i\tto Y_i$ be given for each $i<2$.
We show $F_i\tW F_0+F_1$.
\begin{itemize}
\item {\em Inner reduction.} Define $\iota_i\colon X_i\to X_0+X_1$ as $\iota_i(x)=(i,x)$.
Clearly, this is total.
\item {\em Outer reduction.} 
Define $\psi_i\colon X_i\times(Y_0+Y_1)\to Y_i$ as $\psi_i(x,\pair{i,y})=y$ and $\psi_i(x,\pair{1-i,y})=0$.
Here, $y\in Y_i$, and recall that $0$ is always contained in the potential domain; in particular, $0\in Y_{1-i}$.
Hence, $\psi_i$ is total.
\end{itemize}

{\em Verification:}
\begin{itemize}
\item {\em Inner reduction.} For each instance $x\in\dom(F_i)$, we have $\iota_i(x)=(i,x)\in\dom(F_0+F_1)$.
\item {\em Outer reduction.} For each solution $\pair{k,y}\in (F_0+F_1)(\iota_i(x))=\pair{i,F_i(x)}$, we clearly have $\psi_i(x,\pair{k,y})=y\in F_i(x)$.
\end{itemize}

{\em Least:}
Let $G\pcolon Z\tto W$ be a potted problem such that $F_0,F_1\tW G$, where $F_i\tW G$ is witnessed by total reductions $(\nu_i,\psi_i)$.
We show $F_0+F_1\tW G$.
\begin{itemize}
\item {\em Inner reduction.}
Define $\nu\colon X_0+X_1\to Z$ as $\nu(i,x)=\nu_i(x)$.
Since each $\nu_i\colon X_i\to Z$ is total, so is $\nu$.
\item {\em Outer reduction.}
Define $\psi\colon (X_0+X_1)\times W\to Y_0+Y_1$ as $\psi(\pair{i,x},y)=\pair{i,\psi_i(x,y)}$.
Since each $\psi_i\colon X_i\times W\to Y_i$ is total, so is $\psi$.
\end{itemize}

{\em Verification:}
\begin{itemize}
\item {\em Inner reduction.}
For each instance $\pair{i,x}\in\dom(F_0+F_1)$, we have $\nu(i,x)=\nu_i(x)\in\dom(G)$.
\item {\em Outer reduction.}
For each solution $y\in G(\nu(i,x))$, we have $\psi(\pair{i,x},y)=\pair{i,\psi_i(x,y)}\in\{i\}\times F_i(x)=(F_0+F_1)(\pair{i,x})$.
\end{itemize}
\end{proof}

One can also introduce other Weihrauch operations such as the meet operation for potted problems; however, we will omit them in this article since we do not use them.

\section{Many-Query Reducibility}\label{sec:main-many-query-reducibility}

\subsection{Oracle-as-polynomial}

In computability theory, there are various notion of reducibility; among them, {\em Turing reducibility}---the most fundamental---allows for {\em multiple queries} to an oracle.
In contrast, Weihrauch reducibility $\leq_{\sf W}$ is a notion of {\em single-query} reducibility.
An extension of this notion to multi-query reducibility $\gw$ is provided via what is known as a {\bf reduction game} \cite{HiJo16}.
Here, we introduce the multi-query version of total Weihrauch reducibility $\tW$.
However, since its definition appears quite complicated at first glance, we will first explain the simple abstract idea behind it; see also Ahman-Bauer \cite{AhBa25,AhBa26} and Pradic-Price \cite{PrPr25,PrPr26}.

\begin{enumerate}
\item A {\em one-query} computation using an oracle capable of solving a problem $F\colon{X}\tto Y$ is an element of the following {\bf polynomial}:
\[\sum_{x\in{X}}V^{F(x)}:=\{(x,\psi)\mid x\in{X}\mbox{ and }\psi\colon {F(x)}\to V\mbox{ is computable}\}.\]

An element $(x,\psi)$ of this polynomial is the following process:
Make a query $x\in X$ to the oracle $F$ once, receive a response $y\in F(x)$, and then output a value $\psi(y)$ of type $V$.
\item A {\em one-query reduction} from a problem $F$ to $G$ is a transformation from ``one-query computations using $F$ as an oracle'' into ``one-query computations using $G$ as an oracle''; that is, a morphism between polynomials $\lambda V.\sum_{x\in X}V^{F(x)}\to\lambda V.\sum_{z\in Z}V^{G(z)}$.
By a simple calculation (abstractly speaking, Yoneda's Lemma), this can be identified with an element of the following set:
\begin{align*}
\prod_{x\in X}\sum_{z\in Z}F(x)^{G(z)}:=\{(\nu,\psi)\mbox{ computable}\mid&\forall x(x\in X\to \nu(x)\in Z\mbox{ and }\\[-1em]
&\forall y( y\in G(\nu(x))\to \psi(x,y)\in F(x)))
\}.
\end{align*}
Each element $(\nu,\psi)$ is nothing other than a Weihrauch reduction.
\item A {\em multi-query oracle computation} using $F$ can be described by composing the polynomial functior $\lambda V.\sum_{x\in X}V^{F(x)}$ the required number of times.
To be precise, since the number of queries during a computation is unknown in advance, a multi-query oracle computation is given as the least fixed point of iterated composition, which is known as the so-called {\bf free monad} over the above polynomial functor:
\[
\mu S.\left(\sum_{x\in X}S^{F(x)}\;+\; V\right).
\]
Here, $\mu$ is the least fixed-point operator.
This is the following process:
\begin{enumerate}
\item 
The {\sc Computer} first selects whether to move to the left or right of the $+$ sign.
That choice is encoded as $i\in\{0, 1\}$.
\item $i=1$ means that the current state has moved to the right.
Then, terminates the computation, and output a value of type $V$.
\item $i=0$ means that the current state has moved to the left.
Then, make a query $x\in X$ to the oracle $F$, receive a response $y\in F(x)$, and then reach at the state $S$.

Here, $S$ is a fixed point, into which the entire expression inside the parentheses is substituted.
This means that the state returns to the initial position (a) again, and go to the next round of the computation.
\end{enumerate}
\end{enumerate}

A multi-query computation (3) can be regarded as an interactive game between {\sc Oracle} $F$ (named {\tt Merlin} below) and {\sc Computer} (named {\tt Arthur} below).
This is the key idea behind a reduction game.
With this in mind, we provide a game description of the free monad $\mu S.\left(\sum_{x\in X}S^{F(x)}\;+\; V\right)$ in the potted setting.

\subsection{Total Weihrauch game}

 \begin{definition}\label{def:total-Weihrauch-game}
 Let $\paset{U}$ be a potted set, and $V$ be a set.
 For a potted problem $F$, consider the following two player perfect information game $\tgame(\paset{U};F;V)$:
        \vspace{-0.2em}
        \[
        \begin{array}{rlllllllll}
           {\tt Merlin}   \colon	& x_0   &		& x_1   &	    & \dots & x_{k}   &\\
           {\tt Arthur}   \colon	&	    & y_0   &		& y_1   & \dots &         &y_{k}
        \end{array}
        \]
        \underline{Rule} (potential):
        The {\em \ppp rule} requires two players to comply with the following:
        \begin{itemize}
        \item {\tt Merlin} plays $x_0\in U_+$ in round $0$, and $x_n\in\cod_+(F)$ in subsequent rounds $n>0$.
       \item {\tt Arthur} plays $y_n=\pair{i,z_n}\in\dom_+(F)+V$ in round $n$.
       \begin{itemize}
       \item $i=0$ is a declaration to proceed to the next round of the game.
       \item $i=1$ is a declaration to terminate the game.
       \end{itemize}
       \item {\tt Arthur} must declare the termination of the game in some round (such a round may depend on a play).
        \end{itemize}
        \underline{Rule} (actual):
        The {\em \aaa rule} requires two players to comply with the following:
        \begin{itemize}
            \item {\tt Merlin} plays $x_0 \in U_-$ in round $0$.
            \item {\tt Arthur} plays $y_n=\pair{i,z_n}\in\dom(F)+V$ in round $n$.
            \begin{itemize}
                \item $i=0$ is a declaration to proceed to the next round of the game.
                \item $i=1$ is a declaration to terminate the game.
                In this case, we call $z_n$ the output of this play.
            \end{itemize}
            \item In round $n+1$, {\tt Merlin} returns a solution $x_{n+1} \in F(z_n)$ in response to {\tt Arthur}'s query.
        \end{itemize}
        \underline{Strategy} (potential):
        \begin{itemize}
            \item Consider a total computable function:
            \[\eta\colon U_+\times\cod_+(F)^{<\om}\to\dom_+(F)+V.\]
	\item Fix such an $\eta$.
	Then for a sequence $(x_0,x_1,\dots,x_k)\in U_+\times\cod_+(F)^{<\om}$, define $y_n=\eta(x_0,x_1,\dots,x_n)$ for each $n\leq k$.
	The sequence $(x_0,y_0,x_1,y_1,\dots,x_n,y_n)$ obtained in this way is called a {\em play} following $\eta$.
	\item Such an $\eta$ is a {\em potential strategy} (\ppp stragety) for {\tt Arthur} if, for any play following $\eta$, one of the following must hold:
	\begin{itemize}
	\item {\tt Merlin} violates the \ppp rule before {\tt Arthur} does.
    \item {\tt Arthur} obeys the \ppp rule and declares the termination of the game in some round.
	\end{itemize}
	\item Such an $\eta$ is a {\em actual strategy} (\aaa strategy) for {\tt Arthur} if, in addition, for any play following $\eta$, one of the following must hold:
	\begin{itemize}
	\item {\tt Merlin} violates the \aaa rule before {\tt Arthur} does.
	\item {\tt Arthur} obeys the \aaa rule and declares the termination of the game in some round.
	\end{itemize}
        \end{itemize}
        \underline{Strategy $\to$ Problem}:
        An \aaa strategy $\eta$ for {\tt Arthur} induces the following problem $\Phi_\eta^F$:
        \begin{itemize}
            \item $\dom_+(\Phi_\eta^F)=U_+$, $\dom(\Phi_\eta^F)=U_-$, and $\cod_+(\Phi_\eta^F)=V$.
            \item For an \aaa instance $x_0\in U_-$, we declare $y\in\Phi_\eta^F(x_0)$ if
           there is a play such that {\tt Merlin} starts with the opening move $x_0$, both {\tt Merlin} and {\tt Arthur} obey the \aaa rule, and the game terminates with the output $y$.
        \end{itemize}
    \end{definition}

\begin{definition}\label{def:total-G-Weihrauch}
Let $F,G$ be potted problems.
Then we say that $G$ is {\bf multi-query total Weihrauch reducible} to $F$ if there is an \aaa strategy $\eta$ for the game $\tgame(D_G;F;E_G)$ for $D_G=(\dom(G),\dom_-(G))$ and $E_G=\cod_+(G)$ such that the following holds:
\[
\forall x\in\dom(G).\ \Phi_\eta^F(x)\subseteq G(x).
\]

In this case, we write $G\tgw F$.
The superscript {\sf G} stands for a game.
\end{definition}

\begin{notation}
If $e$ is an index for $\eta$, then we often write $\Phi_e^F$ to denote $\Phi_\eta^F$.
\end{notation}

\begin{remark}
We often think of {\tt Arthur}'s strategy $\eta$ as a partial computable function on $\om\times\om^{<\om}$ whose domain is maximally extended.
\end{remark}

\begin{obs}\label{obs:tgw-preorder}
$\tgw$ is preorder.
\end{obs}

\begin{proof}
Although the idea behind the proof is trivial, we postpone the detailed exposition until Proposition \ref{prop:transitive}.
Abstractly speaking: A multi-query reduction is a natural transformation between free monads (regarded as polynomial functors), so the transitivity follows from the composition of natural transformations.
\end{proof}

\begin{remark}[Multi-query reducibility for decision problems]
On the one hand, the notion of {\bf Turing reducibility} is recovered as multi-query reducibility $\gw$ restricted to single-valued total functions $\om\to 2$.
On the other hand, the notion of {\bf truth-table reducibility} is recovered as multi-query total reducibility $\tgw$ restricted to single-valued total functions $\om\to 2$.
\end{remark}
\subsection{Example}
In Theorem \ref{thm:UC-hierarchy}, we showed that the $\tW$-strength of $\sii{\sf C}_{1/n}$ varies for different values of $n$.
However, under the more powerful reducibility $\tgw$, which allows many queries, the hierarchy of $\sii{\sf C}_{1/n}$ collapses.

\begin{prop}\label{prop:UC-hierarchy-collapse}
$\Sigma_1\text{-}{\sf C}_{1/2}\eqtgw\Sigma_1\text{-}{\sf C}_{1/n}$ for any $n<\om$.
\end{prop}

\begin{proof}
We show $\Sigma_1\text{-}{\sf C}_{1/n^2}\tgw\Sigma_1\text{-}{\sf C}_{1/n}$.
\begin{itemize}
\item \underline{Round $0$}.
Assume that {\tt Merlin}'s opening move is an index $e$ of the $\Sigma_1$ set $S_e=\{k<n^2:\varphi_e(k)\downarrow\}\subseteq n^2$.
Then wait for seeing $\ell<n$ such that $\varphi_e(k)\downarrow$ for some $k\in[\ell n,(\ell+1)n-1]$.
For the first found such $\ell$, we declare $\varphi_{a(e)}(\ell)\downarrow$.
Then {\tt Arthur} makes the query $a(e)$.

If $S_e\not=\emptyset$ then $S_{a(e)}\subseteq S_e$ is a singleton; otherwise $S_e=S_{a(e)}=\emptyset$.
\item \underline{Round $1$}.
Let $\ell$ be a response from ${\tt Merlin}$.
Then wait for seeing $k<n$ such that $\varphi_e(\ell n+k)\down$.
For the first found such $k$, we declare $\varphi_{b(e,\ell)}(k)\downarrow$.
Then {\tt Arthur}'s next query is $b(e,\ell)$.


Note that $S_{b(e,\ell)}\subseteq n$.
If $S_{e,\ell}:=S_e\cap[\ell n,(\ell+1)n-1]\not=\emptyset$ then $\{\ell n+k:k\in S_{b(e,\ell)}\}\subseteq S_{e,\ell}$ is a singleton; otherwise $S_{e,\ell}=S_{b(e,\ell)}=\emptyset$.
\item \underline{Round $2$}:
Let $k$ be a response from ${\tt Merlin}$.
Then, {\tt Arthur} uses the received information $\ell,k$ to output $\ell n+k$ and declares the termination of the game.
\end{itemize}

{\em Verification:}
\begin{itemize}
\item {\em Totality.}
Regardless of $e,\ell,k$, the queries $a(e)$, $b(e,\ell)$ and the output $\ell n+k$ are always defined.
Hence, the above yields a total strategy.
\item {\em Reduction.}
For each instance $e\in\dom(\sii{\sf C}_{1/n^2})$, $S_e$ is a singleton.
Therefore, there is a unique $\ell$ such that $S_{e,\ell}$ is a singleton, and in this case, we have $S_{a(e)}=\{\ell\}$.
Then, this $\ell$ must be {\tt Merlin}'s response in round $1$.
As $S_{e,\ell}\not=\emptyset$, our construction ensures that $S_{b(e,\ell)}\subseteq n$ is also a singleton, so we get a response $k\in S_{b(e,\ell)}$ from {\tt Merlin} in round $2$.
Then, {\tt Arthur} terminates the game with the output $\ell n+k\in\{\ell n+k:k\in S_{b(e,\ell)}\}\subseteq S_{e,\ell}\subseteq S_e$.
\end{itemize}

Consequently, we get $\Sigma_1\text{-}{\sf C}_{1/n^2}\tgw\Sigma_1\text{-}{\sf C}_{1/n}$.
By repeating this process, we obtain $\Sigma_1\text{-}{\sf C}_{1/2}\tgw\Sigma_1\text{-}{\sf C}_{1/n}$ for any 
$n<\om$.
\end{proof}

\subsection{Computation Tree}\label{sec:computation-tree-multifunction}
Games are intuitively easy to understand, but they are not well-suited for rigorous arguments.
By focusing on trees of strategies (game trees), we can use terminology on trees, making it easier to describe the formal details of the discussion.
It is easier to describe the formal details of a discussion by using terminology on trees, via a game tree (strategy trees).
The translation between {\em terminating games} and {\em labeled well-founded trees} is well-known in various contexts; e.g.~a translation between the least fixed point (the term algebra) and its syntax tree.
Those familiar with type theory might find it helpful to think of a W-type (a type of well-founded trees / the initial algebra of a polynomial functor).
We describe game trees in the potted setting.

\medskip

\begin{terminology}[on trees]~
\begin{itemize}
\item A {\bf tree} is a set of finite strings closed under initial segments.
Then the empty string is also called a {\bf root}.
A maximal string is called a {\bf leaf}.
\item A tree is {\bf well-founded} if it has no infinite path.
\item We use the symbol $\leaf T$ to denote the set of all leaves of a tree $T$.
We also put $\internal T=T\setminus\leaf T$.
Each $\sigma\in\internal T$ is called an {\bf internal node}.
\item We say that a tree $T\subseteq Y^{<\om}$ is {\bf full-branching} if for any internal node $\sigma\in\internal T$ and $i\in Y$, we have $\sigma\fr i\in T$.
\end{itemize}
\end{terminology}

\begin{remark}
If $T\subseteq Y^{<\om}$ is a full-branching well-founded tree, then for any $\alpha\in Y^\om$ there is $n$ such that $\alpha\upto n\in\leaf T$.
\end{remark}

\begin{construction}
Let $F\pcolon X\tto Y$ be a potted problem.
Analyze the game $\tgame(\paset{U};F;V)$ in Definition \ref{def:total-Weihrauch-game}.

\underline{\em Strategy $\Rightarrow$ Computation Tree:}
{\tt Arthur}'s strategy induces a labeled well-founded tree.
\begin{itemize}
\item The set of {\tt Merlin}'s possible plays form a well-founded tree.
\item {\tt Arthur}'s moves are recorded as labels on the internal nodes.
\end{itemize}

Recall that {\tt Arthur}'s strategy is a total computable function
\[\eta\colon U_+\times Y^{<\om}\to X+V.\]

\noindent
\underline{Potential}:
For each \ppp input $n\in U_+$, the {\bf \ppp computation tree} induced by $\eta$ is the following labeled well-founded tree $(\tpot_n,\nu_n,\psi_n)$:
\begin{itemize}
\item $\tpot_n\subseteq Y^{<\om}$ is a computable full-branching well-founded tree.
\begin{itemize}
\item $\sigma\in\tpot_n$ is a leaf iff $\sigma$ is a minimal string satisfying $\eta(n,\sigma)=(1,y)$ for some $y$.
\item Hence, $\tpot_n$ is the tree of all plays of {\tt Merlin} obeying the \ppp rule.
\end{itemize}
\item At each internal node of $\tpot_n$, the strategy $\eta$ makes a query in $X$.
Formally, define the query-labeling function $\nu_n\colon \internal \tpot_n\to X$ as follows:
\begin{itemize}
\item The above item on $\tpot_n$ ensures that, for each $\sigma\in \internal \tpot_n$, we have $\eta(n,\sigma)=(0,x)$ for some $x\in X$; hence, we put $\nu_n(\sigma)=x$.
\end{itemize}
\item At each leaf of $\tpot_n$, the strategy outputs a value of type $V$.
Formally, define the output-labeling function $\psi_n\colon \leaf \tpot_n\to V$ as follows:
\begin{itemize}
\item The above item on $\tpot_n$ ensures that, for each $\sigma\in \leaf \tpot_n$, we have $\eta(n,\sigma)=(1,y)$ for some $y\in V$; hence, we put $\psi_n(\sigma)=y$.
\end{itemize}
\end{itemize}

\noindent
\underline{Actual}:
For each \aaa input $n\in U_-$, the {\bf \aaa computation tree} for $\Phi^F_\eta(n)$ induced by $\eta$ is the following ``{\em $F$-branching}'' subtree $\tactual_n\subseteq \tpot_n$:
\begin{itemize}
\item $\tactual_n\subseteq \tpot_n$ has a root.
For each internal node $\sigma\in \internal \tpot_n$, if we have already put $\sigma\in\tactual_n$ and $\nu_n(\sigma)\in \dom(F)$ then we put $\sigma\fr i\in\tactual_n$ for any $i\in F(\nu_n(\sigma))$.
\item In other words, $\tactual_n$ is the tree of all plays by {\tt Merlin} obeying the \aaa rule.
\item The label function is given by restricting the above $\psi_n$ and $\nu_n$ to $\tactual_n$.
\end{itemize}
\end{construction}

\begin{obs}\label{obs:Arthur-strategy}
There is an effective one-to-one correspondence between \ppp strategies for $\tgame(\paset{U};F;V)$ and computable functions $n\in U_+\mapsto (\tpot_n,\nu_n,\psi_n)$ satisfying the following:
\begin{itemize}
\item $\tpot_n\subseteq Y^{<\om}$ is a computable full-branching well-founded tree.
\item $\nu_n\colon \internal \tpot_n\to X$ is a computable function.
\item $\psi_n\colon \leaf \tpot_n\to V$ is a computable function.
\end{itemize}

Furthermore, such a triple $(\tpot_n,\nu_n,\psi_n)$ uniquely determines the \aaa tree $\tactual_n\subseteq\tpot_n$ by the above construction, which corresponds to an \aaa strategy.
\end{obs}

\begin{notation}
For a triple $\Phi=(\tpot_n,\nu_n,\psi_n:n\in U_+)$ satisfying the condition in Observation \ref{obs:Arthur-strategy}, define $\Phi^F$ as follows:
\begin{itemize}
\item $\dom_+(\Phi^F)=U_+$.
\item $\dom(\Phi^F)=\{n\in U_+:\forall \sigma\in \tactual_n\cap\internal \tpot_n.\ \nu_n(\sigma)\in \dom(F)\}$.
\item $\Phi^F(n)=\{\psi_n(\sigma):\sigma\in\leaf\tactual_n\}$.
\end{itemize}

\end{notation}

\begin{obs}
$G\tgw F$ if and only if there is $\Phi$ as above such that:
\begin{itemize}
\item {\it Potential.} $\dom_+(G)\subseteq\dom_+(\Phi^F)$.
\item {\it Actual.} If $n\in\dom(G)$ then $\Phi^F(n)\subseteq G(n)$.
\end{itemize}
\end{obs}


%

%
%
%
%

\begin{remark}
To align with the notation in traditional computability theory, one may write $\psi^\sigma(n)$ or $\psi(\sigma;n)$ instead of $\psi_n(\sigma)$.
\end{remark}

\begin{remark}
A tree $\tpot$ can be recovered from $(\nu,\psi)$:
\begin{itemize}
\item That is, a leaf of $\tpot$ is a minimal string $\sigma$ such that $\psi(\sigma)\down$.
\item Then $\tpot$ is recovered as the downward closure of the leaves.
\end{itemize}

Thus, the pair $(\nu,\psi)$ control the whole computation. 
A code of the computation $(\nu,\psi)$ is a pair of indices of partial computable functions $\pcolon\om^{<\om}\to\om$ maximally extending $\psi,\nu$.
A computation of code $e$ is written as $(\nu^e,\psi^e)$.
\end{remark}

\subsection{Computability}

In the context of total reducibility, a (partial) computable oracle also play an important role.
First, we analyze the conditions under which an oracle generates only computable functions.

\begin{definition}
Let $F\pcolon X\tto Y$ be a problem.
\begin{itemize}
\item We say that $G\pcolon X\tto Y$ {\bf refines} $F\pcolon X\tto Y$ if $\dom(F)\subseteq\dom(G)$ and $G(x)\subseteq F(x)$ for any $x\in\dom(F)$.
\item We say that $f\pcolon X\to Y$ is a {\bf choice function} for $F\pcolon X\tto Y$ if $\dom(F)\subseteq\dom(f)$ and $f(x)\in F(x)$ for any $x\in\dom(F)$.
\item We say that $F$ is {\bf (total) computable} if it has a (total) computable choice function.
\end{itemize}
\end{definition}

Below, ${\rm id}_\om\colon\om\to\om$ is the total identity function on $\om$.

\begin{prop}\label{prop:below-id-total-computable-choice}
For a potted problem $F\pcolon X\tto Y$:
\[F\tgw{\rm id}_\om\iff\mbox{$F$ is total computable.}\]
\end{prop}

\begin{proof}
($\Rightarrow$)
{\em Sketch:}
\begin{itemize}
\item Since the solution to a query $q$ to ${\rm id}_\om$ is $q$ itself, ${\tt Merlin}$'s play can be simulated in a computable manner.
Therefore, by feeding ${\tt Merlin}$'s computable play into ${\tt Arthur}$'s total computable strategy, we can compute the output $f(n)$ for each input $n \in \dom_+(F)$.
This $f$ is a total computable choice function for $F$.
\end{itemize}

{\em Details.}
The precise construction of $f$ is given as follows:
\begin{itemize}
\item For each \ppp input $n\in\dom_+(F)$, one can compute a \ppp tree $(\tpot_n,\nu_n,\psi_n)$.
\item Given $n\in\dom_+(F)$, we inductively compute a leaf $\rho(n)\in\leaf \tpot_n$ as follows:
\begin{itemize}
\item Let $\rho_0$ be the root of $\tpot_n$, i.e., the empty string.
We inductively assume that we have already computed a node $\rho_s\in \tpot_n$ of length $s$
\item If $\rho_s$ is an internal node of $\tpot_n$, the total computable function $\nu_n$ makes a query $\nu_n(\rho_s)$ to the oracle ${\rm id}_\om$.
Since its solution is the query itself, put $\rho_{s+1}=\rho_s\fr \nu_n(\rho_s)$.
\item If $\rho_s$ is a leaf of $\tpot_n$, put $\rho(n)=\rho_s$.
\end{itemize}
\item For $n\in\dom_+(F)$, define $f(n)=\psi_n(\rho(n))$.
Then $f$ is a total computable function.
\end{itemize}

{\em Verification:}
\begin{itemize}
\item For each \aaa input $n\in\dom(F)$, this construction give a unique leaf $\rho(n)\in\leaf\tactual_n$ in the \aaa computation tree $\tactual_n\subseteq \tpot_n$.
Since $(\tpot_n,\nu_n,\psi_n)$ solves $F(n)$, we get $f(n)=\psi_n(\rho(n))\in F(n)$.
\end{itemize}

Consequently, $f$ is a total computable choice for $F$.
This shows that $F$ is total computable.

\medskip

($\Leftarrow$)
Assume that $F$ is total computable.
Then, there is a total computable choice $f\colon X\to Y$ for $F$.
Note that $f$ refines $F$; in particular, $F\tW f$ holds.
We show $f\tW{\rm id}_\om$.
\begin{itemize}
\item {\em Inner reduction.} For each $n\in X$, put $\nu(n)=n$; then $\nu$ is clearly total.
\item {\em Outer reduction.} For each $n\in X$ and $m\in\om$, put $\psi(n,m)=f(n)$.
As $f$ is total computable on $X$, so is $\psi$.
\end{itemize}

This clearly gives a reduction $f\tW{\rm id}_\om$.
Thus, we get $F\tW f\leq_{\sf tW}{\rm id}_\om$.
\end{proof}

Recall from Proposition \ref{prop:MP-as-Sigma-1-choice} that Markov's principle ${\sf MP}$ corresponds to $\Sigma_1$-unique choice $\sii{\sf C}_{1/\om}$.

\begin{theorem}\label{MP-equial-partial-computable}
For a potted problem $F\pcolon X\tto Y$:
\[F\tgw\sii{\sf C}_{1/\om}\iff\mbox{$F$ is computable,}\]
\end{theorem}

\begin{proof}
($\Rightarrow$)
{\em Sketch.}
\begin{itemize}
\item Note that ${\sf MP}$ has a (partial) computable choice $h$.
If $n\in\dom(F)$, any query made by ${\tt Arthur}$ must be in $\dom({\sf MP})$, so the computation of $h$ to find its solution always terminates.
This means that ${\tt Merlin}$'s play can be simulated in a partial computable manner.
Therefore, by feeding ${\tt Merlin}$'s partial computable play into ${\tt Arthur}$'s total computable strategy, we can compute an output $f(n)$ for each \aaa input $n\in\dom(F)$.
This $f$ is a computable choice function for $F$.
\end{itemize}

%
%
%
%

{\em Details.}
Put $M={\sf MP}_{\sf PR}$, and assume $F\tgw M$ via $e$.
We construct a choice $\Gamma_e(n)$ for $F(n)$.
\begin{itemize}
\item For each \ppp input $n\in\dom_+(F)$, the index $e$ yields a \ppp computation tree $(\tpot^e_n,\nu^e_n,\psi^e_n)$.
Hereafter, the superscript $e$ will be omitted.
\item Given \aaa input $n\in\dom(F)$, one can compute the \aaa computation tree $\tactual_n\subseteq \tpot_n$ and its leftmost leaf $\rho(n)\in\leaf\tactual_n$.
In details, we inductively compute the leaf $\rho(n)\in\leaf \tpot_n$ as follows:
\begin{itemize}
\item Let $\rho_0$ be the root of $\tactual_n$.
We inductively assume that we have already computed a node $\rho_s\in \tactual_n$ of length $s$.
\item If $\rho_s$ is an internal node of $\tactual_n$, the total computable function $\nu_n$ makes a query $\nu_n(\rho_s)=q$ to $M$.
Since we work on the \aaa computation tree, we must have $q\in{\rm dom}(M)$.
Then, wait for a stage $t$ by which the computation $\varphi_q(0)$ halts.
As $q\in{\rm dom}(M)$, we have $\varphi_q(0)\down$, so we eventually see such a stage $t$, and $t\in M(q)$.
Then, by putting $\rho_{s+1}=\rho_s\fr t$, we get $\rho_{s+1}\in\tactual_n$.
\item If $\rho_s$ is a leaf of $\tactual_n$, put $\rho(n)=\rho_s$.
\end{itemize}
\item This construction yields a leaf of the \aaa computation tree, $\rho(n)\in\leaf\tactual_n$.
\item Define $\Gamma_e(n)=\psi_n(\rho(n))$.
Then $\Gamma_e$ is a computable function.
As $(\tpot_n,\nu_n,\psi_n)$ solves $F(n)$, we get $\Gamma_e(n)=\psi_n(\rho(n))\in F(n)$.
\end{itemize}

\medskip

($\Leftarrow$)
Assume that $F$ is computable; that is, $F$ has a computable choice function $\varphi_e$, so we have $F\tW \varphi_e$.
We show $\varphi_e\tW\Sigma_1\text{-}{\sf C}_{1/\om}$.
\begin{itemize}
\item {\em Inner reduction.} For each instance $n$, let $\nu(n)$ be an index of the $\Sigma_1$ set $S_{\nu(n)}=\{\varphi_e(n)\}$.
Then $\nu$ is clearly total.
\item {\em Outer reduction.} For each solution $m\in S_{\nu(n)}$, put $\psi(n,m)=m$.
Clearly, $\psi$ is total.
\end{itemize}

{\em Verification:}
\begin{itemize}
\item {\em Inner reduction.}
If $\varphi_e(n)\downarrow$ then $S_{\nu(n)}$ is a singleton, which means $\nu(n)\in{\rm dom}(\Sigma_1\text{-}{\sf C}_{1/\om})$.
\item {\em Outer reduction.}
Given solution $m\in \sii{\sf C}_{1/\om}(\nu(n))=S_{\nu(n)}$, note that we have $S_{\nu(n)}=\{\varphi_e(n)\}$ by definition; thus, the reduction computes $\psi(n,m)=m=\varphi_e(n)$.
\end{itemize}

Therefore, we get $F\tW\varphi_e\leq_{\sf tW}\Sigma_1\text{-}{\sf C}_{1/\om}$.
\end{proof}

\begin{remark}
The proof of $\Rightarrow$ in Theorem \ref{MP-equial-partial-computable} indeed yields a total computable function $\delta\colon\om\to\om$ satisfying the following:
\[\mbox{$F\tgw\sii{\sf C}_{1/\om}$ via $e$ $\implies$ $\varphi_{\delta(e)}$ is a choice function for $F$.}\]

This is because the construction of $\Gamma_e$ in Theorem \ref{MP-equial-partial-computable} only depends on a code $e$ of the strategy $n\mapsto(\nu^e_n,\psi^e_n)$.
Here, although $\delta$ is total, $\Gamma_e=\varphi_{\delta(e)}$ itself is a partial function.
\end{remark}

Recall from Corollary \ref{cor:LLPO-eq-C-2} that ${\sf LLPO}$ corresponds to $\pii{\sf C}_{\geq 1/2}$.
The following is a reformulation of a basic observation in classical computability theory.
In realizability theory, this forms the core of Lifschitz realizability.

\begin{theorem}\label{LLPO-equial-partial-computable}
For a single-valued problem $f\pcolon\om\to\om$:
\[f\tgw\sii{\sf C}_{1/\om}+\pii{\sf C}_{\geq 1/2}\iff\mbox{$f$ is computable.}\]
\end{theorem}

\begin{proof}
($\Leftarrow$)
By Theorem \ref{MP-equial-partial-computable}, if $f$ is computable, then we have $f\tgw\sii{\sf C}_{1/\om}$; in particular, we get $f\tgw\sii{\sf C}_{1/\om}+\pii{\sf C}_{\geq 1/2}$.

($\Rightarrow$)
Assume $f\tgw G$ via $e$.
For each \ppp input $n\in\dom_+(f)$, the index $e$ yields a \ppp computation tree $(\tpot^e_n,\nu^e_n,\psi^e_n)$.
Hereafter, the superscript $e$ will be omitted. 

{\em Idea.}
\begin{itemize}
\item
For any $n\in{\rm dom}(f)$, the \aaa computation tree $\tactual_n$ is a well-founded binary tree; hence, it is a finite tree by (the contrapositive of) weak K\"onig's lemma.
\item Moreover, this finite tree $\tactual_n$ is a $\Pi_1$ tree (where the $\Pi_1$-ness is caused by $\pii{\sf C}_{\geq 1/2}$), so its computable upper approximation $\tactual_n[t]$ converges to $\tactual_n$ at some finite stage.
\item By single-valuedness of $f(n)$, $\psi_n$ is constant on the leaves of $\tactual_n$.
By finiteness of $\tactual_n$, we eventually observe that $\psi_n$ is constant on the leaves of $\tactual_n[t]$ at some stage $t$.
After observing this, we output the unique value of $\psi_n$, which must be $f(n)$.
\end{itemize}

{\em Details.}
We construct an algorithm $\Gamma_e(n)$ to compute $f(n)$ as follows:
\begin{enumerate}
\item At each stage $t$, proceed as follows:
Declare that the root of the tree $\tpot_n$ is active.
Let $\sigma\in\internal \tpot_n$ be an internal node.
Inductively assume that $\sigma$ is active.
\item 
If $\nu_n(\sigma)$ makes a query $q$ to $\sii{\sf C}_{1/\om}$,
check whether some element is enumerated into the $q$th $\Sigma_1$ set $S_q \subseteq \om$ by stage $t$.
\begin{itemize}
\item If some element is enumerated, for the first such $k\in S_q$, we declare that $\sigma\fr k$ is active, while the node $\sigma$ is switched to inactive.
\item If no element is enumerated, keep $\sigma$ active but in the wait state.
\end{itemize}
\item If $\nu_n(\sigma)$ makes a query $q$ to $\pii{\sf C}_2$, look at the approximation $P_q[t]$ of the $q$th $\Pi_1$ set $P_q\subseteq\{0,1\}$ at stage $t$.
\begin{itemize}
\item If we currently have $i\in P_q[t]$, declare that $\sigma\fr i$ is active.
\item If we recognize $i\not\in P_q[t]$, all nodes extending $\sigma\fr i$ are switched to inactive.
\item If we recognize $P_q[t]=\emptyset$, keep $\sigma$ active but in the wait state.
\end{itemize}
\item If there is an active node in the wait state, go back to (1) at the next stage $t+1$.
If there is no such a node, go to (5).
\item If $\psi_n(\rho)$ takes a common value on all active leaves in $\tpot_n$, then $\Gamma_e(n)$ output such a common value.
\end{enumerate}

{\em Verification:}
\begin{itemize}
\item For each \aaa input $n\in\dom(f)$, the \aaa computation tree $\tactual_n\subseteq \tpot_n$ has exactly one branching at a $\sii{\sf C}_{1/\om}$-querying node, and at most two branching at  $\pii{\sf C}_{\geq 1/2}$-querying node; in particular, finitely branching.
\item 
Before all leaves of $\tactual_n$ become active, there is always an active node in the wait state.
Furthermore, every leaf eventually becomes active after some stage.
This is because:
\begin{itemize}
\item 
If an internal node $\sigma\in\internal\tactual_n$ makes a query $q$ to $\sii{\sf C}_{1/\om}$, then at stage $t$, either $\sigma$ in the wait state for something to be enumerated into $S_q$, or else the unique successor of $\sigma$ in $\tactual_n$ is set to active.
Since this is a node in the \aaa computation tree, we have $q\in\dom(\sii{\sf U}_{1/\om})$.
Therefore, exactly one element must be enumerated into $S_q$ at some stage $s_\sigma$, and then $\sigma\fr q$ is activated.
\item If an internal node $\sigma\in\internal\tactual_n$ makes a query $q$ to $\pii{\sf C}_{\geq 1/2}$, then $\sigma\fr i\in\tactual_n$ implies $i\in P_q$, so $\sigma\fr i$ is always active.
As $P_q\subseteq\{0,1\}$, this set stabilizes after some stage $s_\sigma$; that is, we have $P_q=P_q[t]$ for any $t\geq s_\sigma$.
\end{itemize}
\item We see that all active nodes in $\tpot_n$ belong to $\tactual_n$ after some stage:
\begin{itemize}
\item Since $\tactual_n$ is finite, considering $s=\max\{s_\sigma:\sigma\in\internal\tactual_n\}$, one can inductively see that all active nodes belong to $\tactual_n$ after stage $s$.
\end{itemize}
\item when the algorithm $\Gamma_e(n)$ reaches at step (5), by the condition at step (4) and the second item above, all leaves of $\tactual_n$ must be active.
The output $\psi_n(\rho)$ at such a leaf $\rho$ must be the correct $f(n)$.
In particular, if $\Gamma_e(n)$ returns some value, it must be $f(n)$.
\item By the third item above, all active leaves in $\tpot_n$ belong to $\tactual_n$; hence, by single-valuedness of $f(n)$, $\psi_n(\rho)$ takes a common value on all active leaves.
Therefore, $\Gamma_e(n)$ eventually outputs some value.
\end{itemize}

Consequently, the algorithm $\Gamma_e$ computes $f$.
Since an index of $\Gamma_e$ is computable in $e$, hereafter, we refer to $\Gamma_e$ as $\varphi_{\ep(e)}$.
\end{proof}

\section{Extended Oracle}\label{sec:bilayer-computability}

\subsection{Total Extended Weihrauch reducibility}
When dealing with pure logical principles (with no complexity bounds), some previous studies have employed a notion called ``extended predicates'' or ``bilayer functions'' (see Bauer \cite{Bau22}, Kihara \cite{Kih23}, and Kihara-Ng \cite{KiNg26}).
In Kihara's context, this is a function involving {\em public} and {\em secret} inputs, and computability is involved only on the public inputs.

\begin{definition}
A {\bf potted bilayer problem}  is a tuple $(X,\Lambda,Y;F)$ consisting of sets $X,Y\subseteq\om$ containing $0$, an (arbitrary) set $\Lambda$ and a subset $F\subseteq X\times\Lambda\times Y$.
We use the following notations and terminologies:
\begin{itemize}
\item We write $F\pcolon X\times\Lambda\tto Y$ for such a tuple.
\item The product $X\times\Lambda$ is denoted by ${\rm dom}_+(F)$, which s called the {\em potential domain} (or the \ppp domain).
\item $Y$ is denoted by ${\rm cod}_+(F)$, which is called the {\em potential codomain} (or the \ppp codomain).
\item For each $(x,\alpha)\in X\times\Lambda$, $F(x|\alpha)$ is defined as $\{y\in Y:(x,\alpha,y)\in F\}$.
\item $\dom(F)$ is defined as $\{(x,\alpha)\in X\times\Lambda:F(x|\alpha)\not=\emptyset\}$, which is called the {\em actual domain} (or the \aaa domain).
\item If $F(x|\alpha)$ is a singleton $\{y\}$, then we write $F(x|\alpha)=y$.
\item If $F(x|\alpha)$ is a singleton for each $(x|\alpha)\in\dom(F)$, then we say that $F$ is {\em single-valued} and write $F\pcolon X\times\Lambda\to Y$.
\end{itemize}
\end{definition}

For each $(x|\alpha)\in X\times\Lambda$, $x$ is called a {\bf public} instance, and $\alpha$ is called a {\bf secret} instance.
A usual problem can be thought of as a function that has only public instances.

\begin{remark}
Each potted problem $F\pcolon X\tto Y$ is always identified with the following potted bilayer problem $\tilde{F}\pcolon X\times\{\ast\}\tto Y$:
\[\tilde{F}(x|\ast)=F(x).\]

Hereinafter, $\tilde{F}$ will also be written simply as $F$.
\end{remark}

The notion of {\em extended Weihrauch reducibility} was first introduced by Bauer \cite{Bau22}.
Then Kihara \cite{Kih23} described it as a reduction between bilayer problems involving secret data:
A reduction between bilayer problems consists of machines $\psi$ and $\nu$, which can only see the public input, and an advisor $\zeta$, which sees both the public and secret inputs and makes a secret query.

\begin{definition}\label{def:redu-potted-bilayer}
Let $F\pcolon X\times\Lambda\tto Y$ and $G\pcolon Z\times\Delta\tto W$ be potted bilayer problems.
We say that $F$ is {\bf total extended Weihrauch reducible to} $G$ if there are partial computable functions $\nu,\psi$ and a function $\zeta$ (not necessarily computable) such that the following holds:
\begin{enumerate}
\item $\nu\colon X\to Z$ is total.
\item $\zeta\colon X\times\Lambda\to\Delta$ is total.
\item $\psi\colon X\times W\to Y$ is total.
\item $(x|\alpha)\in{\rm dom}(F)$ implies $(\tilde{x}|\tilde{\alpha}):=(\nu(x)|\zeta(x|\alpha))\in{\rm dom}(G)$.
\item $(x|\alpha)\in{\rm dom}(F)$ and $y\in G(\tilde{x}|\tilde{\alpha})$ imply $\psi(x,y)\in F(x|\alpha)$.
\end{enumerate}

In this case, we write $F\leq_{\sf tW}G$.
We call $\nu$ a {\em public inner reduction}, $\zeta$ a {\em secret inner reduction} and $\psi$ an {\em outer reduction}.
\end{definition}

The reader might think it would be more appropriate to use {\sf teW} rather than {\sf tW} for the subscript, but in this article we have avoided making the notation complicated.
In fact, the reduction itself remains the same; only the type of problems under consideration changes.
That is, a reduction is still a morphism of polynomial functors; the only change is that the objects have shifted from ``monosemically-coded'' sets (modest sets) to ``polysemically-coded'' sets (assemblies).
Henceforth, we simply refer to the notion of reducibility described above as total Weihrauch reducibility.

\subsection{Examples (One-Query Reducibility)}

Since the definition looks quite complicated, let us look at a few simple concrete examples before analyzing this relation $\tW$.
\begin{definition}
For $0<k<n\leq\om$, the choice principle ${\sf C}_{k/n}$ is defined as follows:
\begin{itemize}
\item $\dom_+({\sf C}_{k/n})=\{\ast\}\times\mathcal{P}(n)$.
\item $\dom({\sf C}_{k/n})=\{\ast\}\times\{A\subseteq n:|A|=k\}$.
\item $\cod_+({\sf C}_{k/n})=n$.
\item ${\sf C}_{k/n}(\ast|A)=A$; that is, any $a\in A$ is a solution to ${\sf C}_{k/n}(\ast|A)$.
\end{itemize}
\end{definition}

\begin{definition}
For $0<k<n\leq\om$, the choice principle ${\sf C}_{\geq k/n}$ is defined as follows:
\begin{itemize}
\item $\dom_+({\sf C}_{\geq k/n})=\{\ast\}\times\mathcal{P}(n)$.
\item $\dom({\sf C}_{\geq k/n})=\{\ast\}\times\{A\subseteq n:|A|\geq k\}$.
\item $\cod_+({\sf C}_{\geq k/n})=n$.
\item ${\sf C}_{\geq k/n}(\ast|A)=A$; that is, any $a\in A$ is a solution to ${\sf C}_{\geq k/n}(\ast|A)$.
\end{itemize}
The unique choice problem ${\sf C}_{1/n}$ can be identified with the following ``secret'' identity function:
\begin{itemize}
\item $\dom_+({\sf C}_{1/n})=\dom({\sf C}_{1/n})=\{\ast\}\times n$.
\item $\cod_+({\sf C}_{1/n})=n$.
\item ${\sf C}_{1/n}(\ast|a)=a$; that is $a$ is a solution to ${\sf C}_{1/n}(\ast|a)$.
\end{itemize}
\end{definition}

\begin{remark}
The choice problem ${\sf C}_{1/2}$ is associated with the {\em weak law of excluded middle} $\neg A\lor\neg\neg A$, which is equivalent to ${\sf C}_{\geq 1/2}$ realizing {\em de Morgan's law} $\neg(A\land B)\to \neg A\lor\neg B$.
In general, the all-or-co-unique choice problem ${\sf C}_{\geq n-1/n}$ can be associated with a variant of the weak law of excluded middle $\bflogic{WLEM}_k$ (see Section \ref{sec:WLEM}).
Similarly, the choice problem ${\sf C}_{\geq 1/\om}$ may be linked to de Morgan's law $\bflogic{DML}_\N$ on the natural numbers $\neg\forall x^\N P(x)\to \exists x^\N\neg P(x)$ (see Section \ref{sec:WLEM}).
\end{remark}

\begin{prop}\label{prop:UC-equal-DML}
${\sf C}_{1/n}\equiv_{\sf tW}{\sf C}_{\geq 1/n}$ for any $n\leq\om$.
\end{prop}

\begin{proof}
($\leq_{\sf tW}$)
This is obvious since the left-hand side is the restriction of the right-hand side.

($\geq_{\sf tW}$)
Total public reductions $\nu,\psi$ and a total secret reduction $\zeta$ are given as follows:
\begin{itemize}
\item {\em Inner reduction.} 
For public put $\nu(\ast)=\ast$.
This is clearly total computable.
For secret, put $\zeta(\ast|\emptyset)=\emptyset$ and $\zeta(\ast|A)=\min A$ for $A\not=\emptyset$.
This is clearly total.
\item {\em Outer reduction.}
Put $\psi(\ast,k)=k$.
This is clearly total computable.
\end{itemize}

{\em Verification:}
\begin{itemize}
\item {\em Inner reduction.}
For each instance $(\ast|A)\in\dom({\sf C}_{\geq 1/n})$, we have $A\not=\emptyset$, which is mapped to $(\ast|\min A)\in\dom({\sf UC}_k)$.
\item {\em Outer reduction.} For each solution $k={\sf C}_{1/n}(\ast|\min A)$, we get $\psi(\ast,k)=k=\min A\in A={\sf C}_{\geq 1/n}(\ast|A)$.
\end{itemize}
\end{proof}

\begin{prop}\label{prop:coUC-equal-WLEM}
${\sf C}_{n-1/n}\equiv_{\sf tW}{\sf C}_{\geq n-1/n}$ for any $n\geq \om$.
\end{prop}

\begin{proof}
($\leq_{\sf tW}$)
This is obvious since the left-hand side is the restriction of the right-hand side.

($\geq_{\sf tW}$)
Total public reductions $\nu,\psi$ and a secret reduction $\zeta$ are given as follows:
\begin{itemize}
\item {\em Inner reduction.} For public, put $\nu(\ast)=\ast$.
This is clearly total computable.
For secret, given $A\subseteq k$, the following gives a total function:
\[
\zeta(\ast|A)=
\begin{cases}
A\setminus\{0\}&\mbox{ if }A=n\\
A&\mbox{ if }A\not=n
\end{cases}
\]
\item {\em Outer reduction.}
Put $\psi(\ast,k)=k$.
This is clearly total computable.
\end{itemize}

{\em Verification:}
\begin{itemize}
\item {\em Inner reduction.}
For each instance $(\ast|A)\in\dom({\sf C}_{\geq n-1/n})$, we have $|k\setminus A|\leq 1$.
If $A=k$ then $\eta(\ast|A)=A\setminus\{0\}$ is a co-singleton.
If $A\not=k$ then $\eta(\ast|A)=A$ is also a co-singleton.
In any case, we have $(\ast|\zeta(\ast|A))\in\dom({\sf C}_{n-1/n})$.
\item {\em Outer reduction.}
For each solution $k\in{\sf C}_{n-1/n}(\ast|\zeta(\ast|A))$, we get $\psi(\ast,k)=k\in\zeta(\ast|A)\subseteq A={\sf C}_{\geq n-1/n}(\ast|A)$.
\end{itemize}
\end{proof}


We show that ${\sf C}_{1/\om}$ is harder than any potted problem.

\begin{prop}\label{prop:DML-upper-bound}
$F\tW{\sf C}_{1/\om}$ for any potted problem $F\pcolon X\tto Y$.
\end{prop}

\begin{proof}
By Proposition \ref{prop:UC-equal-DML}, it suffices to show $F\tW{\sf C}_{\geq 1/\om}$.
Recall that an instance $x\in X$ for $F$ is identified with $(x|\ast)$.
Total public reductions $\nu,\psi$ and a secrete reduction $\zeta$ are given as follows:
\begin{itemize}
\item {\em Inner reduction.} 
For public, put $\nu(x)=\ast$ for any $x\in X$.
This is clearly total computable.
For secret, if $x\not\in\dom(F)$ then put $\zeta(x|\ast)=\emptyset$, and if $x\in\dom(F)$ then put $\zeta(x|\ast)=F(x)$.
This is total.
\item {\em Outer reduction.} 
Put $\psi(x,y)=y$.
This is clearly total computable.
\end{itemize}

{\em Verification:}
\begin{itemize}
\item {\em Inner reduction.}
For each instance $(x|\ast)\in\dom(F)$, we have $\emptyset\not=F(x)\subseteq\om$, which is mapped to $(\ast|F(x))\in\dom({\sf C}_{\geq 1/\om})$.
\item {\em Outer reduction.}
For each solution $y\in{\sf C}_{\geq 1/\om}(\ast|F(x))$, we get $\psi(x,y)=y\in F(x)$.
\end{itemize}
\end{proof}

As a variation of Proposition \ref{prop:DML-upper-bound}, we see that ${\sf C}_{1/2}$ provides an upper bound for two-valued problems.

\begin{prop}\label{prop:WLEM-upper-bound}
For any potted problem $F\pcolon X\tto Y$, if $|Y|\leq n$ then $F\tW{\sf C}_{1/n}$.
\end{prop}

\begin{proof}
By Proposition \ref{prop:UC-equal-DML}, it suffices to show $F\tW{\sf C}_{\geq 1/n}$.
Recall that an instance $x\in X$ for $F$ is identified with $(x|\ast)$.
As $|Y|\leq n$, there is a surjection $h\colon n\to Y$.
Since $h$ is a finite function, one can treat this as a computable function.
Total public reductions $\nu,\psi$ and a secrete reduction $\zeta$ are given as follows:
\begin{itemize}
\item {\em Inner reduction.}
For public, put $\nu(x)=\ast$ for any $x\in X$.
This is clearly total computable.
For secret, if $x\not\in\dom(F)$ then put $\zeta(x|\ast)=\emptyset$, and if $x\in\dom(F)$ then put $\zeta(x|\ast)=\{i\leq n:h(i)\in F(x)\}$.
This is total.
\item {\em Outer reduction.}
Put $\psi(x,y)=h(y)$.
This is total computable.
\end{itemize}

{\em Verification:}
\begin{itemize}
\item {\em Inner reduction.}
For each instance $(x|\ast)\in\dom(F)$, we have $\emptyset\not=F(x)\subseteq Y$, which is mapped to $\emptyset\not=\zeta(x|\ast)\subseteq n$.
Hence, $(\ast|\zeta(x|\ast))\in\dom({\sf C}_{\geq 1/n})$.
\item {\em Outer reduction.}
For each solution $y\in{\sf C}_{\geq 1/n}(\ast|\zeta(x|\ast))=\zeta(x|\ast)$, by our definition of $\zeta$, we have $h(y)\in F(x)$.
Hence. we get $\psi(x,y)=h(y)\in F(x)$.
\end{itemize}

\end{proof}

To summarize the discussion so far, we have replaced logical principles with choice problems, as follows:
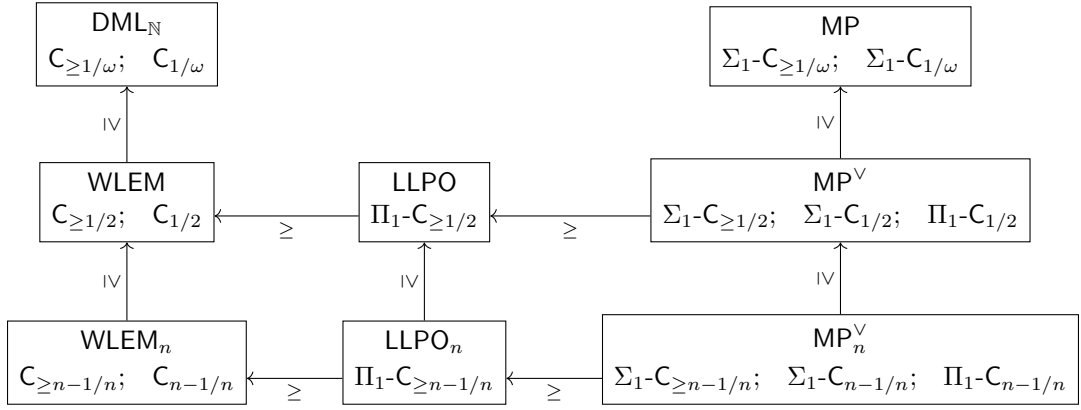
\[
\begin{tikzcd}[row sep=large, column sep=large]
|[draw]|
{\shortstack{
${\sf DML}_\N$\\[0.2em]
${\sf C}_{\geq 1/\om}$;\quad${\sf C}_{1/\om}$
}}
&
&
|[draw]|
{\shortstack{
${\sf MP}$\\[0.2em]
$\sii{\sf C}_{\geq 1/\om}$;\quad
$\sii{\sf C}_{1/\om}$
}}
\\
|[draw]|
{\shortstack{
${\sf WLEM}$\\[0.2em]
${\sf C}_{\geq 1/2}$;\quad${\sf C}_{1/2}$
}}
\arrow[u, "\rotatebox{-90}{$\geq$}"]
&
|[draw]|
{\shortstack{
${\sf LLPO}$\\[0.2em]
$\pii{\sf C}_{\geq 1/2}$
}}
\arrow[l, "\geq"]
&
|[draw]|
{\shortstack{
${\sf MP}^\lor$\\[0.2em]
$\sii{\sf C}_{\geq 1/2}$;\quad
$\sii{\sf C}_{1/2}$;\quad
$\pii{\sf C}_{1/2}$
}}
\arrow[l, "\geq"]
\arrow[u, "\rotatebox{-90}{$\geq$}"]
\\
|[draw]|
{\shortstack{
${\sf WLEM}_n$\\[0.2em]
${\sf C}_{\geq n-1/n}$;\quad${\sf C}_{n-1/n}$
}}
\arrow[u, "\rotatebox{-90}{$\geq$}"]
&
|[draw]|
{\shortstack{
${\sf LLPO}_n$\\[0.2em]
$\pii{\sf C}_{\geq n-1/n}$
}}
\arrow[l, "\geq"]
\arrow[u, "\rotatebox{-90}{$\geq$}"]
&
|[draw]|
{\shortstack{
${\sf MP}_n^\lor$\\[0.2em]
$\sii{\sf C}_{\geq n-1/n}$;\quad
$\sii{\sf C}_{n-1/n}$;\quad
$\pii{\sf C}_{n-1/n}$
}}
\arrow[l, "\geq"]
\arrow[u, "\rotatebox{-90}{$\geq$}"]
\end{tikzcd}
\]

\subsection{Basic Properties}

First, we must show that $\tW$ forms a preorder.
The proof may seem somewhat cumbersome, but this is due to the complexity of the definition; the argument itself is self-evident.
In fact, abstractly speaking, this follows from the fact that a (total) Weihrauch reduction corresponds to a natural transformation between polynomial functors \cite{PrPr25} (and of course, the composition of natural transformations is a natural transformation).

\begin{prop}\label{prop:tW-preorder2}
$\leq_{\sf tW}$ is a preorder.
\end{prop}

\begin{proof}
{\em Reflexivity:}
Consider $\nu(x)=x$, $\zeta(x|\alpha)=\alpha$ and $\psi(x,y)=y$.

{\em Transitivity:}
Let $F\pcolon X_0\times\Lambda_0\tto Y_0$, $G\pcolon X_1\times\Lambda_1\tto Y_1$, $H\pcolon X_2\times\Lambda_2\tto Y_2$ be potted bilayer problems such that $F\tW G\tW H$.
Assume $F\tW G$ via $(\nu,\psi|\zeta)$ and $G\tW H$ via $(\nu',\psi'|\zeta')$; that is, we have the following functions:
\begin{align*}
&\nu\colon X_0\to X_1, & & \nu'\colon X_1\to X_2,\\
&\zeta\colon X_0\times\Lambda_0\to\Lambda_1, & & \zeta'\colon X_1\times\Lambda_1\to \Lambda_2,\\
&\psi\colon X_0\times X_1\to Y_0, & & \psi'\colon X_1\times Y_2\to Y_1.
\end{align*}

We want to construct a reduction $F$ from $H$.
\begin{itemize}
\item {\em Public reduction.}
Since $\nu''=\nu'\circ\nu\colon X_0\to X_2$ is a total function, this satisfies the condition (1) in Definition \ref{def:redu-potted-bilayer}.
\item {\em Secret reduction.}
Define $\zeta''$ as follows:
\[
\xymatrix@R=6pt{
X_0\times \Lambda_0\ar[r]^-{\pair{\nu\pi_0,\zeta}} & X_1\times\Lambda_1\ar[r]^-{\zeta'}& \Lambda_2\\
(x|\alpha)\ar@{|->}[r] & (\nu(x)|\zeta(x|\alpha)) &
}
\]
That is, put $\zeta''(x|\alpha)=\zeta'(\nu(x)|\zeta(x|\alpha))$, which satisfies the condition (2) in Definition \ref{def:redu-potted-bilayer}.

\item {\em Outer reduction.}
Define $\psi''\colon X_0\times Y_2\to Y_0$ as follows:
\[\psi''(x,y)=\psi(x,\psi'(\nu(x),y))\]
for any $(x,y)\in X_0\times Y_2$.
Then, we have $\nu(x)\in X_1$, so $y':=\psi'(\nu(x),y)\in Y_1$; thus, $\psi(x,y')\in Y_0$.
Hence, $\psi''$ fulfills the totality condition (3) in Definition \ref{def:redu-potted-bilayer}.
\end{itemize}

{\em Verification:}
\begin{itemize}
\item {\em Inner reduction.}
Now, we see
\[(x|\alpha)\in\dom(F)\implies(\tilde{x}|\tilde{\alpha}):=(\nu(x)|\zeta(x|\alpha))\in\dom(G)\implies(\nu'(\tilde{x})|\zeta'(\tilde{x}|\tilde{\alpha}))\in\dom(H).\]
By definition, we have $\nu''(x)=\nu'(\nu(x))=\nu'(\tilde{x})$; thus, $\zeta''(x|\alpha)=\zeta'(\nu(x)|\zeta(x|\alpha))=\zeta'(\tilde{x}|\tilde{\alpha})$.
Therefore, we get $(\nu''(x)|\zeta''(x|\alpha))\in\dom(H)$, which satisfies the condition (4) in Definition \ref{def:redu-potted-bilayer}.
\item {\em Outer reduction.}
Let $(x|\alpha)\in\dom(F)$ and $y\in H(\nu''(x)|\zeta''(x|\alpha))$ be given.
In particular, we have $(\tilde{x}|\tilde{\alpha}):=(\nu(x)|\zeta(x|\alpha))\in\dom(G)$.
\begin{itemize}
\item Since $G\tW H$ via $(\nu',\psi')$, as we know $(\tilde{x}|\tilde{\alpha})\in\dom(G)$, each solution $y\in H(\nu'(\tilde{x})|\zeta'(\tilde{x}|\tilde{\alpha}))=H(\nu''(x)|\zeta''(x|\alpha))$ yields $\psi'(\tilde{x},y)\in G(\tilde{x}|\tilde{\alpha})$.
\item Since $F\tW G$ via $(\nu,\psi)$, as we know $(x|\alpha)\in\dom(F)$ and $\psi'(\tilde{x},y)\in G(\tilde{x}|\tilde{\alpha})=G(\nu(x)|\zeta(x|\alpha))$, we get $\psi''(x,y)=\psi(x,\psi'(\tilde{x},y))\in F(x|\alpha)$.
\end{itemize}

Therefore, the condition (5) in Definition \ref{def:redu-potted-bilayer} is fulfilled.
This implies that $(\nu'',\psi'')$ witnesses $F\tW H$.
\end{itemize}
\end{proof}

The domain of a potted bilayer problem can be simplified as follows:
\begin{prop}\label{prop:bilayer-extension}
Let $F\pcolon X\times\Lambda\tto Y$ be a potted bilayer problem.
Then the domain of $F$ can be extended to a potted set $\paset{A}=(A_-;A_+)$ in the following sense:
\begin{enumerate}
\item $A_+=X$.
\item There is a potted bilayer function $\hat{F}$ satisfying the following conditions:
\begin{align*}
\dom_+(\hat{F})=\dom_+(F)=A_+\times\Lambda,
& &
\dom(F)\subseteq\dom(\hat{F})=A_-\times\Lambda.
\end{align*}
\item $\hat{F}(x|\alpha)=F(x|\alpha)$ for any $(x|\alpha)\in\dom(F)$.
\item $\hat{F}\equiv_{\sf tW}F$.
\end{enumerate}
\end{prop}

\begin{proof}
We first define a potted set $\paset{A}=(A_-;A_+)$.
Put $A_+=X$, and define $A_-$ as:
\[A_-=\{x\in X:\exists\alpha\in\Lambda.\ (x|\alpha)\in\dom(F)\}.\]

Definition of $\hat{F}$:
\begin{itemize}
\item Given $(x|\alpha)\in A_-\times\Lambda$, if $(x|\alpha)\in\dom(F)$ then put $\hat{F}(x|\alpha)=F(x|\alpha)$; otherwise, $\hat{F}(x|\alpha)=Y$.
\end{itemize}

{\em Verification:}
The conditions (1),(2) and (3) in Definition \ref{def:redu-potted-bilayer} are clearly fulfilled.
We only show (4),
\begin{itemize}
\item $\geq_{\sf tW}$: Since $\hat{F}$ is an extension of $F$, just consider $\nu(x)=x$, $\zeta(x|\alpha)=\alpha$ and $\psi(x,y)=y$.
\item $\tW$:
Put $\nu(x)=x$ and $\psi(x,y)=y$.
Define the total function $\zeta\colon A_-\times\Lambda\to\Lambda$ as follows:
\begin{itemize}
\item Given $(x|\alpha)\in A_-\times\Lambda$, if $(x|\alpha)\in\dom(F)$ then put $\zeta(x|\alpha)=\alpha$; otherwise, since $x\in A_-$, there is $\beta\in\Lambda$ such that $(x|\beta)\in\dom(F)$.
Choose such a $\beta$, and put $\zeta(x|\alpha)=\beta$.
\end{itemize}
Let $(x|\alpha)\in\dom(\hat{F})$ be given.
\begin{itemize}
\item If $(x|\alpha)\in\dom(F)$ then $(\nu(x)|\zeta(x|\alpha))=(x|\alpha)\in\dom(F)$; thus, if $y\in F(x|\alpha)$ then $\psi(x,y)=y\in F(x|\alpha)=\hat{F}(x|\alpha)$.
\item If $(x|\alpha)\not\in\dom(F)$ then $(\nu(x)|\zeta(x|\alpha))=(x|\beta)\in\dom(F)$; thus, if $y\in F(x|\beta)$ then we get $\psi(x,y)=y\in F(x|\beta)\subseteq Y=\hat{F}(x|\alpha)$.
\end{itemize}

This shows $\hat{F}\tW F$.
\end{itemize}

Consequently, we obtain $\hat{F}\equiv_{\sf tW} F$.
\end{proof}

\begin{notation}
An extension such as Proposition \ref{prop:bilayer-extension} is written as $\hat{F}\colon\paset{A}\times\Lambda\to Y$.
Indeed, by Proposition \ref{prop:bilayer-extension}, we may always assume that a potted bilayer problem is of the form $F\colon\paset{X}\times\Lambda\tto Y$.
\end{notation}

\subsection{{\tt Merlin}-{\tt Arthur}-{\tt Nimue} game}\label{sec:MAN-game}
As a multi-query reduction game involving a bilayer problem, Kihara \cite{Kih23} introduced an incomplete information three-player game involving secrete moves.
In abstract terms, this is also nothing more than a free monad $\lambda V.\mu S.(\sum_{x\in X}S^F(x)+V)$, but because the underlying objects have changed from monosemically-coded sets (modest sets) to polysemically-coded sets (assemblies), the exchange of secret data is inserted.
We consider such a game in the potted context.

 \begin{definition}\label{def:total-Weihrauch-game}
 Let $\paset{U}$ be a potted set, $V\subseteq\om$ be a set, and $\Xi$ be an (arbitrary) set.
 For a potted bilayer problem $F\colon X\times\Lambda\tto Y$, consider the following imperfect information three-player game $\tgame(\paset{U}|\Xi;F;V)$:
        \vspace{-0.2em}
        \[
        \begin{array}{rlllllllll}
           {\tt Merlin}   \colon	& (x_0|\alpha)   &		& x_1   &	    & \dots & x_{k}   &\\
           {\tt Arthur}   \colon	&	    & y_0   &		& y_1   & \dots &         &y_{k} \\
           {\tt Nimue}   \colon	&	    & \beta_0   &		& \beta_1   & \dots &         &\beta_{k}
        \end{array}
        \]
        \underline{Rule} (potential): The {\em \ppp rule} requires three players to comply with the following:
        \begin{itemize}
        \item {\tt Merlin}'s opening move in round $0$ is the pair of a public move $x_0\in U_+$ and a secret move $\alpha\in\Xi$.
        In subsequent rounds $n>0$, {\tt Merlin} plays a public move $x_n\in\cod_+(F)$.
       \item {\tt Arthur} plays a public move $y_n=\pair{i,z_n}\in X_++V$ in round $n$.
       \begin{itemize}
       \item $i=0$ is a declaration to proceed to the next round of the game.
       \item $i=1$ is a declaration to terminate the game.
       \end{itemize}
       \item {\tt Nimue} plays a secrete move $\beta_n\in\Lambda$ in round $n$.
       \item {\tt Arthur} must declare the termination of the game in some round (such a round may depend on a play).
        \end{itemize}
        \underline{Rule} (actual): The {\em \aaa rule} requires two players to comply with the following:
        \begin{itemize}
            \item {\tt Merlin}'s opening move in round $0$ is the pair of a public move $x_0\in U_-$ and a secret move $\alpha\in\Xi$.
            \item {\tt Arthur} plays a public move $y_n=\pair{i,z_n}\in X_-+V$ in round $n$.
            \begin{itemize}
                \item $i=0$ is a declaration to proceed to the next round of the game.
                \item $i=1$ is a declaration to terminate the game.
                In this case, we call $z_n$ the output of this play.
            \end{itemize}
       \item {\tt Nimue} plays a secret move $\beta_n\in\Lambda$ in round $n$.
            \item In round $n+1$, {\tt Merlin} returns a public solution $x_{n+1} \in F(z_n|\beta_n)$ in response to {\tt Arthur}-{\tt Nimue}'s query.
        \end{itemize}
        \underline{Strategy} (potential):
        \begin{itemize}
            \item Consider total functions
            \begin{align*}
            &\eta\colon U_+\times \cod_+(F)^{<\om}\to X_++V\\
            &\zeta\colon U_+\times \cod_+(F)^{<\om}\times\Xi\to \Lambda
            \end{align*}
           We think of $\eta$ as {\tt Arthur}'s strategy and $\zeta$ as {\tt Nimue}'s strategy.
           Here, we require computability only for $\eta$.
	\item Fix such $\eta$ and $\zeta$.
	Then for a sequence $(x_0,x_1,\dots,x_k)\in A_+\times Y^{<\om}$ and a secret instance $\alpha\in\Xi$, define
	\begin{align*}
	y_n&=\eta(x_0,x_1,\dots,x_n)\\
	\beta_n&=\zeta(x_0,x_1,\dots,x_n|\alpha)
	\end{align*}
	for each $n\leq k$.
	The sequence $(x_0|\alpha,y_0|\beta_0,x_1,y_1|\beta_1,\dots,x_k,y_k|\beta_k)$ obtained in this way is called a {\em play} following $(\eta|\zeta)$.
	\item Such an $(\eta|\zeta)$ is a {\em potential strategy} (\ppp strategy) for {\tt Arthur}-{\tt Nimue} if, for any play following $(\eta|\zeta)$, one of the following must hold:
	\begin{itemize}
	\item {\tt Merlin} violates the \ppp rule before {\tt Arthur}-{\tt Nimue} do.
	\item {\tt Arthur}-{\tt Nimue} obey the \ppp rule and declare the termination of the game in some round.
	\end{itemize}
	\item Such an $(\eta|\zeta)$ is an {\em actual strategy} (\aaa strategy) for {\tt Arthur}-{\tt Nimue} if, in addition, for any play following $(\eta|\zeta)$, one of the following must hold:
	\begin{itemize}
	\item {\tt Merlin} violates the \aaa rule before {\tt Arthur}-{\tt Nimue} do.
	\item {\tt Arthur}-{\tt Nimue} obey the \aaa rule and declare the termination of the game in some round.
	\end{itemize}
        \end{itemize}
        \underline{Strategy $\to$ Problem}:
        An \aaa strategy $\eta$ for {\tt Arthur}-{\tt Nimue} induces the following potted bilayer problem $\Phi_{\eta|\zeta}^F$:
        \begin{itemize}
            \item $\dom_+(\Phi_{\eta|\zeta}^F)=U_+\times\Xi$, $\dom(\Phi_{\eta|\zeta}^F)=U_-\times\Xi$, and $\cod_+(\Phi_{\eta|\zeta}^F)=V$.
            \item For a public \aaa instance $x_0\in U_-$ and a secret instance $\alpha\in\Xi$, we declare $y\in\Phi_{\eta|\zeta}^F(x_0|\alpha)$ if there is a play such that {\tt Merlin} starts with the opening move $(x_0|\alpha)$, all of {\tt Merlin}, {\tt Arthur} and {\tt Nimue} obey the \aaa rule, and the game terminates with the output $y$.
        \end{itemize}
    \end{definition}

\begin{definition}
Let $F,G$ be potted bilayer problems.
Let $D_G,\Xi_G,E_G$ be such that $\dom_+(G)=(D_G)_+\times\Xi_G$, $\dom(G)\subseteq(D_G)_-\times\Xi_G$, and $E_G=\cod_+(G)$.
Then we say that $G$ is {\bf multi-query total Weihrauch reducible} to $F$ if there is an \aaa strategy $(\eta|\zeta)$ for the game $\tgame(D_G|\Xi_G;F;E_G)$ such that the following holds:
\[
\forall x\in\dom(G).\ \Phi_{\eta|\zeta}^F(x)\subseteq G(x).
\]

In this case, we write $F\tgw G$.
\end{definition}

\begin{remark}
As before, we often think of {\tt Arthur}'s strategy as a partial computable function on $\om\times\om^{<\om}$ whose domain is maximally extended.
\end{remark}

\subsection{Example (Multi-Query Reducibility)}

${\sf C}_{n/n+1}$ is the co-unique choice on $n+1$, while $\pii{\sf C}_{\geq n-1/n}$ is the $\Pi_1$ all-or-co-unique choice on $n$.
At first glance, ${\sf C}_{n/n+1}$ does not appear to imply $\pii{\sf C}_{\geq n-1/n}$, but as Kihara \cite[Proposition 4.3]{Kih23} pointed out, if we allow for partially computable reductions (i.e., using Markov's principle $\sii{\sf C}_{1/\om}$), then ${\sf C}_{n/n+1}$ does imply $\pii{\sf C}_{\geq n-1/n}$.
In logical terms, ${\sf WLEM}_{n+1}+{\sf MP}$ implies ${\sf LLPO}_n$.
Here, we refine this observation.

\begin{prop}\label{prop:LLPOn-reducible-WLEMnplusone}
$\pii{\sf C}_{\geq n-1/n}\tgw{\sf C}_{n/n+1}+\sii{\sf C}_{n-1/n}$.
\end{prop}

\begin{proof}
By Theorem \ref{thm:basic-Sigma-1-counique} and Proposition \ref{prop:coUC-equal-WLEM}, it suffices to show $\pii{\sf C}_{\geq n-1/n}\tgw{\sf C}_{\geq n/n+1}+\pii{\sf C}_{n-1/n}$.
For each instance $e\in{\rm dom}(\pii{\sf C}_{\geq n-1/n})$, there are $n+1$ possibilities.
\begin{itemize}
\item Case $i<n$: $i\not\in P_e$.
\item Case $n$: $P_e=n$.
\item Two or more of these cases cannot occur simultaneously.
\end{itemize}

\underline{Round $0$}: ${\tt Arthur}$ asks {\tt Nimue} for help, so {\tt Nimue} asks ${\sf C}_{\geq n/n+1}$ which case does not occur.
\begin{itemize}
\item If {\tt Merlin}'s response is $i<n$, then ${\tt Arthur}$ thinks that Case $i$ never happens; that is, $i\in P_e$, so outputs $i$.
\item If {\tt Merlin}'s response is $n$, then ${\tt Arthur}$ thinks that Case $n$ does not occur; that is, $P_e\not=n$.
Then ${\tt Arthur}$ declares to proceed to the next round.
\end{itemize}

\underline{Round $1$}: ${\tt Arthur}$ makes a query $e$ to $\pii{\sf C}_{n-1/n}$.
That is, require {\tt Merlin} to choose an element of $P_e$.
\begin{itemize}
\item If {\tt Merlin}'s response is $k$, then ${\tt Arthur}$ thinks $k\in P_e$, so returns $k$.
\end{itemize}

{\em Verification:}
\begin{itemize}
\item Regardless of the value of $e$, {\tt Arthur} outputs a value, so it is a total strategy.
\item For each instance $e\in{\rm dom}(\pii{\sf C}_{\geq n-1/n})$, the query in round $1$ belongs to ${\rm dom}({\sf C}_{\geq n/n+1})$, so {\tt Merlin} returns a correct solution.
\item If the response is $i<n$, then $i$ belongs to $P_e$.
\item If the response is $n$, then we must have $P_e\not=n$, so $P_e$ is a co-singleton.
Therefore, the query $e$ in round$1$ is a correct query to $\pii{\sf C}_{n-1/n}$, so the response $k$ must belong to $P_e$.
\end{itemize}

Consequently, we obtain $\pii{\sf C}_{\geq n-1/n}\tgw{\sf C}_{\geq n/n+1}+\pii{\sf C}_{n-1/n}$.
\end{proof}

In logical terms, the above shows ${\sf WLEM}_{n+1}+{\sf MP}_n^\lor$ implies ${\sf LLPO}_n$.
Later, we will see that the use of (a weak variant of) Markov's principle is essential (Corollary \ref{cor:LLPOk-not-reduce-WLEMkp1}).

As in Proposition \ref{prop:UC-hierarchy-collapse}, the hierarchy of unique choice problems collapses.
\begin{prop}\label{prop:UC-hierarchy-collapse-bilayer}
${\sf C}_{1/2}\eqtgw{\sf C}_{1/n}$ for any $n<\om$.
\end{prop}

\begin{proof}
We show ${\sf C}_{n^2}\tgw{\sf C}_{n}$.

\begin{itemize}
\item \underline{Round $0$}.
Assume that {\tt Merlin} first plays a secret move $A\subseteq n^2$.
Then {\tt Arthur} asks {\tt Nimue} for help, so {\tt Nimue} makes the query $B=\{\ell<n:A\cap [\ell n,(\ell+1)n-1]\not=\emptyset\}$.

If $A$ is a singleton, so is $B$.
\item \underline{Round $1$}.
Let $\ell<n$ be a response from {\tt Merlin}.
Then {\tt Arthur} asks {\tt Nimue} for help, so {\tt Nimue} makes the query $C_\ell=\{k<n:\ell n+k\in A\cap [\ell n,(\ell+1)n-1]\}$.

If $\ell\in B$ then $C_\ell$ is a singleton.
\item \underline{Round $2$}.
Let $k$ be a response from ${\tt Merlin}$.
Then, {\tt Arthur} uses the received information $\ell,k$ to output $\ell n+k$ and declare the termination of the game.
\end{itemize}

{\em Verification:}
\begin{itemize}
\item {\em Totality.}
Regardless of $A$, the queries $B$, $C_\ell$ and the output $\ell n+k$ are always defined. Hence, the above yields a total strategy.
\item {\em Reduction.}
For each instance $(\ast|A)\in\dom({\sf C}_{1/n^2})$, $A$ is a singleton.
Then $B\not=\emptyset$ is also a singleton, so {\tt Merlin}'s response in round $1$ is some $\ell\in B$; hence, $C_\ell$ is a singleton.
Thus, {\tt Merlin}'s response in round $2$ is some $k\in C_\ell$.
By definition, we get $\ell n+k\in A$.
Therefore, {\tt Arthur}'s output $\ell n+k$ belongs to $A$, so this is a solution to ${\sf C}_{1/n^2}(A)$.
\end{itemize}

Consequently, we get ${\sf C}_{n^2}\tgw{\sf C}_{n}$.
By repeating this process, we obtain ${\sf C}_{1/2}\tgw{\sf C}_{1/n}$ for any $n<\om$.
\end{proof}

\subsection{Computation Tree}
Once again, computation trees are better suited for precise descriptions than games.
We extend the notion of computation tree, introduced in Section \ref{sec:computation-tree-multifunction}, to the bilayer context.
We again use the standard translation between terminating games and labeled well-founded trees.

\begin{construction}
Let $F\pcolon X\times\Lambda\tto Y$ be a potted bilayer problem.
Describe a game tree for $\tgame(A|\Xi;F;B)$ by a computation tree.

\underline{{\em Strategy $\Rightarrow$ Computation Tree:}}
{\tt Arthur}-{\tt Nimue}'s strategy induces a labeled well-founded tree.
\begin{itemize}
\item The set of {\tt Merlin}'s possible plays form a well-founded tree.
\item {\tt Arthur}'s moves are recorded as labels on the internal nodes.
\item {\tt Nimue}'s move imposes a restriction on the successors of each node.
\end{itemize}

Recall that {\tt Arthur}-{\tt Nimue}'s strategy consists of total functions:
\begin{align*}
\eta\colon U_+\times Y^{<\om}\to X_++V,
& &
\zeta\colon U_+\times Y^{<\om}\times\Xi\to \Lambda.
\end{align*}

\noindent
\underline{Potential}:
For each public \ppp input $n\in U_+$ and secrete input $\alpha\in\Xi$, the {\bf \ppp computation tree} for induced by $(\eta|\zeta)$ is the following labeled well-founded tree $(\tpot_n,\nu_n,\psi_n)$:
\begin{itemize}
\item $\tpot_n\subseteq Y^{<\om}$ is a computable full-branching well-founded tree.
\begin{itemize}
\item $\sigma\in\tpot_n$ is a leaf iff $\sigma$ is a minimal string satisfying $\eta(n,\sigma)=(1,y)$ for some $y$.
\item Hence, $\tpot_n$ is the tree of all plays of {\tt Merlin} obeying the \ppp rule.
\end{itemize}
\item At each internal node of $\tpot_n$, {\tt Arthur}'s strategy $\eta$ makes a query in $X$.
Formally, define the query-labeling function $\nu_n\colon \internal \tpot_n\to X$ as follows:
\begin{itemize}
\item The above item on $\tpot_n$ ensures that, for each $\sigma\in \internal \tpot_n$, we have $\eta(n,\sigma)=(0,x)$ for some $x\in X$; hence, we put $\nu_n(\sigma)=x$.
\end{itemize}
\item At each leaf of $\tpot_n$, the strategy outputs a value of type $V$.
Formally, define the output-labeling function $\psi_n\colon \leaf \tpot_n\to V$ as follows:
\begin{itemize}
\item The above item on $\tpot_n$ ensures that, for each $\sigma\in \leaf \tpot_n$, we have $\eta(n,\sigma)=(1,y)$ for some $y\in V$; hence, we put $\psi_n(\sigma)=y$.
\end{itemize}
\end{itemize}

Since {\tt Nimue}'s strategy is not controlled by computation, we separately it from the computation tree.
\begin{itemize}
\item At each internal node of $\tpot_n$, {\tt Nimue}'s strategy $\zeta$ makes a advice in $\Lambda$.
Formally, define the advice-labeling function $\zeta_{n|\alpha}\colon \internal \tpot_n\to \Lambda$ as follows:
\begin{itemize}
\item Put $\zeta_{n|\alpha}(\sigma)=\zeta(n,\sigma|\alpha)$.
\end{itemize}
\end{itemize}

\noindent
\underline{Actual}:
For each public \aaa input $n\in U_-$ and a secret input $\alpha\in\Xi$, the {\bf \aaa computation tree} for $\Phi^F_{\eta|\zeta}(n|\alpha)$ induced by $(\eta|\zeta)$ is the following ``{\em $F$-branching}'' subtree $\tactual_{n|\alpha}\subseteq \tpot_n$:
\begin{itemize}
\item $\tactual_{n|\alpha}\subseteq \tpot_n$ has a root.
For each internal node $\sigma\in\internal \tpot_n$, if we have already put $\sigma\in\tactual_{n|\alpha}$, then for any $i\in F(\nu_n(\sigma)|\zeta_{n|\alpha}(\sigma))$ (if defined), we put $\sigma\fr i\in\tactual_{n|\alpha}$.
\item In other words, $\tactual_{n|\alpha}$ is the tree of all plays by {\tt Merlin} obeying the \aaa rule.
\item The label function is given by restricting the above $\psi_n$, $\nu_n$ and $\zeta_{n|\alpha}$ to $\tactual_{n|\alpha}$.
\end{itemize}
\end{construction}

\begin{obs}\label{obs:Arthur-Nimue-strategy}
There is an effective one-to-one correspondence between \ppp strategies for $\tgame(A|\Xi;F;B)$ and $n\in U_+\mapsto (\tpot_n,\nu_n,\psi_n|\zeta_n)$ satisfying the following:
\begin{itemize}
\item $\tpot_n\subseteq Y^{<\om}$ is a computable full-branching well-founded tree.
\item $\nu_n\colon \internal \tpot_n\to X_+$ is a computable function.
\item $\psi_n\colon \leaf \tpot_n\to V$ is a computable function.
\item $\zeta_n\colon \internal \tpot_n\times\Xi\to\Lambda$ is a function.
\end{itemize}

Furthermore, for each $\alpha\in\Xi$, such a tuple $(\tpot_n,\nu_n,\psi_n|\zeta_n)$ uniquely determines the \aaa tree $\tactual_{n|\alpha}\subseteq T_n$ by the above construction, which corresponds to an \aaa strategy.
\end{obs}

\begin{notation}
For a tuple $\Phi=(\tpot_n,\nu_n,\psi_n|\zeta_n:n\in A_+)$ satisfying the condition in Observation \ref{obs:Arthur-Nimue-strategy}, define $\Phi^F$ as follows:
\begin{itemize}
\item $\dom_+(\Phi^F)=U_+$.
\item $\dom(\Phi^F)=\{(n|\alpha)\in U_+\times\Xi:\forall \sigma\in \tactual_{n|\alpha}\cap\internal \tpot_n.\ (\nu_n(\sigma)|\zeta_n(\sigma|\alpha))\in \dom(F)\}$.
\item $\Phi^F(n|\alpha)=\{\psi(\sigma):\sigma\in\leaf\tactual_{n|\alpha}$.
\end{itemize}
\end{notation}

\noindent
{\bf Alternative Definition:}
{\tt Nimue}'s strategy is not actually controlled by the {\em computation}; therefore, strictly speaking, it is appropriate to consider this separately from the \aaa {\em computation} tree.
Rather, we regard this as the action of extracting a bushy subtree from the genuine \aaa computation tree induced only by {\tt Arthur}'s strategy.
This idea is formally described as follows:
\begin{definition}~
\begin{enumerate}
\item For sets $D,E$, we say that $\U$ is a {\em $D$-indexed subset family over $E$} if it is of the form $\U=(U_i)_{i\in D}$, and each $A\in U_i$ is a nonempty subset of $E$.
\item Let $\nu\pcolon E^{<\om}\to D$ be given.
We say that $T\subseteq E^{<\om}$ is {\bf $(\U;\nu)$-branching} if:
\begin{itemize}
\item For each internal node $\sigma\in T$, $\nu(\sigma)\in D$ is defined and the following holds:
\[\{n\in E:\sigma\fr n\in T\}\in U_{\nu(\sigma)}.\]
\end{itemize}
%
\end{enumerate}
\end{definition}

\begin{example}
A potted bilayer problem $F\pcolon X\times\Lambda\tto Y$ is identified with the following $X$-indexed subfamily over $Y$:
\[
\overline{F}=\{F(x|\alpha):\alpha\in\Lambda\}_{x\in X}
\]

We say that $T\subseteq Y^{<\om}$ is {\em $(F;\nu)$-branching} if it is $(\overline{F};\nu)$-branching.
\end{example}

Hereafter, we will refer to an indexed tuple $\Phi=(\tpot_n,\nu_n,\psi_n:n\in U_+)$ that satisfies the conditions in Observation \ref{obs:Arthur-Nimue-strategy} as a {\em total tree of type $\paset{U}\to V$}.

\begin{lemma}
Let $F\colon\paset{X}\times\Lambda\tto Y$ and $G\colon\paset{D}\times\Xi\tto E$ be potted bilayer problems.
Then the following two conditions are equivalent:
\begin{enumerate}
\item $G$ is multi-query total Weihrauch reducible to $F$.
\item There is a total tree $\Phi=(\tpot_n,\nu_n,\psi_n:n\in D_+)$ of type $\paset{D}\to E$ satisfying the following:
\begin{itemize}
\item For any \ppp instance $(n|\alpha)\in\dom(G)$, there is an $(F;\nu_n)$-branching subtree $\tactual_{n|\alpha}\subseteq\tpot_n$ such that $\psi_n(\sigma)\in G(n|\alpha)$ for any leaf $\sigma\in\leaf \tactual_{n|\alpha}$.
\end{itemize}
\end{enumerate}
\end{lemma}

\begin{proof}
Straightforward.
\end{proof}

\begin{notation}
Each index $e$ yields an indexed tuple $\Phi_e=(\tpot_n,\nu_n,\psi_n)_{n\in\om}$.
\begin{itemize}
\item If $\Phi_e$ is a total tree, then we write $\Phi_e^+(n)\down$.
\[\Phi_e^+(n)=\{\psi_n(\sigma):\sigma\in\leaf \tpot\}.\]
\item If $\zeta$ is a function choosing a $(F;\nu)$-branching subtree $\tactual^\zeta_{n|\alpha}\subseteq\tpot_n$ for each $n$ and $\alpha$, then:
\begin{align*}
\Phi_{e|\zeta}^F(n|\alpha)\down&\iff\forall\sigma\in\leaf\tactual^\zeta_{n|\alpha}.\ \psi(\sigma)\down,\\
\Phi_{e|\zeta}^F(n|\alpha)&=\{\psi_n(\sigma):\sigma\in\leaf\tactual^\zeta_{n|\alpha}\}.
\end{align*}
\end{itemize}
\end{notation}

%

%
%
%
%

Of course, we must show that $\tgw$ is transitive.
As with the transitivity of the standard multi-query Weihrauch reduction \cite{HiJo16,Kih23}, the idea itself is quite trivial, but its concrete description is somewhat cumbersome.

\begin{prop}\label{prop:transitive}
$\tgw$ is a preorder.
\end{prop}

\begin{proof}
The reflexivity is obvious.
We only prove the transitivity.
For the idea, see also Observation \ref{obs:tgw-preorder}.
For potted bilayer problems
\begin{align*}
F\pcolon X_0\times\Lambda_0\tto Y_0,& & G\pcolon X_1\times\Lambda_1\tto Y_1, & &H\pcolon X_2\times\Lambda_2\tto Y_2
\end{align*}
we assume $F\tgw G\tgw H$.
\begin{itemize}
\item Then there are a total tree $\Phi=(\tpot_n,\nu_n,\psi_n:n\in\dom_+(F))$ witnessing $F\tgw G$, and $\Psi=(\tpot'_n,\nu'_n,\psi'_n:n\in\dom_+(G))$ witnessing $G\tgw H$.
\item We construct a total tree $\Gamma=(\tpot_n'',\nu_n'',\psi_n'':n\in\dom_+(F))$ witnessing $F\tgw H$ and an auxiliary multifunction $\kappa_n\pcolon Y_1^{<\om}\tto Y_2^{<\om}$.
Here, $\kappa_n(\sigma)=\emptyset$ is also possible.
\item The empty string $\ep\in\tpot_n\subseteq Y_1^{<\om}$ is transformed into the empty string $\kappa_n(\ep)=\{\ep\}\subseteq\tpot''_n\subseteq Y_2^{<\om}$.
For $\sigma\in\tpot_n$, we inductively assume that $\kappa_n(\sigma)\subseteq \tpot_n''\subseteq Y_2^{<\om}$ has already been defined.
\item To each internal node $\sigma\in\internal\tpot_n$, a query $\nu_n(\sigma)=q\in\dom_+(G)$ is assigned.
Insert here the computation tree for $\Psi^H(q)$ used to compute $G(q)$:
\begin{align*}
&\{\gamma\fr\tau:\gamma\in\kappa_n(\sigma)\mbox{ and }\tau\in\tpot'_q\}\subseteq\tpot''_n\\
&\gamma\in\kappa_n(\sigma)\mbox{ and }\tau\in\internal\tpot'_q\implies \nu''_n(\gamma\fr\tau)=\nu'_q(\tau)\\
&\gamma\in\kappa_n(\sigma)\mbox{ and }\rho\in\leaf\tpot'_q\mbox{ and }\psi'_q(\rho)=m\implies \gamma\fr\rho\in\kappa_n(\sigma\fr m).
\end{align*}
\item To each leaf $\sigma\in\leaf\tpot_n$, an output $\psi_n(\sigma)\in\cod_+(F)$ is assigned.
For each $\gamma\in\kappa_n(\sigma)$, put $\psi''_n(\gamma)=\psi_n(\sigma)$.
In particular, each $\gamma\in\kappa_n(\sigma)$ is a leaf of $\tpot_n''$.
\end{itemize}

{\em Verification:}
{\em (Potential).}
\begin{itemize}
\item $\tpot''_n\subseteq Y_2^{<\om}$ is a full-branching well-founded tree:
\begin{itemize}
\item {\em Full-branching.}
Since we are simply concatenating full-branching trees $\tpot'_q\subseteq Y_2^{<\om}$ one after another, $\tpot''_n$ is also a full-branching tree.
\item {\em Well-foundedness.}
Each string in $\tpot''_n$ is a concatenation of leaves of $\tpot_q'$ for various queries $q$.
Moreover, a string in $\tpot_n$ must be extended each time a concatenation is performed.
Since both $\tpot_n$ and each $\tpot_q'$ are well-founded, this shows that $\tpot''_n$ is also well-founded.

For more details,
each string $\eta\in\tpot''_n$ is of the form $\eta=\rho_0\fr\rho_1\fr\dots\fr\rho_k\fr\tau$.
Here, for $\sigma=\pair{m_\ell}_{\ell\leq k}\in\tpot_n$ and $q_i=\nu_n(\pair{m_\ell}_{\ell<i})$, we have $\rho_i\in\leaf\tpot'_{q_i}$ and $\tau\in\tpot'_{q_{k+1}}$.
If $\alpha$ is an infinite path through $\tpot''_n$, since each $\tpot_{q_i}'$ is well-founded, $\alpha$ must be in the form of a concatenation of infinitely many leaves $\alpha=\rho_0\fr\rho_1\fr\dots$
In this case, there must exist an infinite sequence $\pair{m_\ell}_{\ell\in\om}\in\tpot_n$ making queries $q_i=\nu_n(\pair{m_\ell}_{\ell<i})$.
This is impossible by well-foundedness of $\tpot_n$.

\end{itemize}
\item $\nu''_n(\sigma)\in\dom_+(H)$:
This is because $\nu''_n$ is simply copying some $\nu'_q$.
\begin{itemize}
\item $\nu''_n(\gamma\fr\tau)=\nu'_q(\tau)\in\dom_+(H)$.
\end{itemize}
\item $\psi''_n(\gamma)\in\cod_+(F)$:
This is because $\psi''_n$ is simply copying some $\psi_n$.
\begin{itemize}
\item We have $\gamma\in\kappa_n(\sigma)$ for some $\sigma\in\leaf\tpot_n$, and then $\psi''_n(\gamma)=\psi_n(\sigma)\in\cod_+(F)$.
\end{itemize}
\end{itemize}

{\em (Actual).}
\begin{itemize}
\item The \aaa computation tree is also given by similar concatenations.
Let $\tactual_n\subseteq\tpot_n$ be a $(G;\nu_n)$-branching tree solving $F(n)$, and $\tactual_q'\subseteq\tpot_q'$ be a $(H;\nu'_q)$-branching tree solving $G(q)$.
\item For each \aaa instance $n\in\dom(F)$, an $(H;\nu''_n)$-branching tree $\tactual_n''\subseteq\tpot_n''$ solving $F(n)$ is given as follows:
To each internal node $\sigma\in\internal\tactual_n$, the query $\nu_n(\sigma)=q\in\dom(G)$ is assigned.
Then
\[\{\gamma\fr\tau:\gamma\in\kappa_n(\sigma)\mbox{ and }\tau\in\tactual'_q\}\subseteq\tactual''_n.\]
\item At each internal node $\gamma\fr\tau$ of $\tactual''_n$, some $q\in\dom(G)$ is queried, so we have $\nu'_q(\tau)\in\dom(H)$.
Therefore, we get $\nu''_n(\gamma\fr\tau)=\nu'_q(\tau)\in\dom(H)$.
\item For each leaf $\gamma$ of $\tactual''_n$, we have $\gamma\in\kappa_n(\sigma)$ for some $\sigma\in\leaf\tactual_n$.
Hence, $\psi_n''(\gamma)=\psi_n(\sigma)\in F(n)$.
\end{itemize}

Note also that our construction of a computation tree $\Gamma$ for $F\tgw H$ from the computation trees $\Phi$ for $F\tgw G$ and $\Psi$ for $G\tgw H$ is effective.
In other words, there is a total computable function $t\colon\om^2\to\om$ which, given indices $d$ for $\Phi$ and $e$ for $\Psi$, returns an index $t(d,e)$ for $\Gamma$.
\end{proof}

\subsection{Partial Weihrauch reducibility}

Some of the results regarding total Weihrauch reducibility can be reduced to the discussion on traditional Weihrauch reducibility.
For this reason, we introduce traditional Weihrauch reducibility \cite{BGP21} here.
In the traditional context, a (bilayer) problem does not have the notion of \ppp domain; therefore, the \ppp rule in the {\tt Merlin}-{\tt Arthur}-{\tt Nimue} game described in Section \ref{sec:MAN-game} is removed.
Traditional many-query Weihrauch reducibility is often denoted by $\leq_{\sf gW}$.
In this article, we write it as $\leq_{\sf W}^{\sf G}$, where ${\sf g}$ and ${\sf G}$ stand for a game, as in Definition \ref{def:total-G-Weihrauch}.

\begin{lemma}
Let $F\pcolon {X}\times\Lambda\tto Y$ and $G\pcolon {D}\times\Xi\tto E$ be bilayer functions.
Then the following two conditions are equivalent:
\begin{enumerate}
\item $G$ is multi-query Weihrauch reducible to $F$.
\item Given $n\in D$, there is an effective procedure returning $(\nu_n,\psi_n)$ such that:
\begin{itemize}
\item $\nu_n\pcolon Y^{<\om}\to X$ is a computable function.
\item $\psi_n\pcolon Y^{<\om}\to E$ is a computable function.
\end{itemize}
Furthermore, for each instance $(n|\alpha)\in\dom(G)$, there is an $(F;\nu_n)$-branching tree $\tactual_{n|\alpha}\subseteq Y^{<\om}$ such that for any leaf $\sigma\in\leaf \tactual_{n|\alpha}$ we have $\psi_n(\sigma)\down\in G(n|\alpha)$.
\end{enumerate}
\end{lemma}

\begin{proof}
Straightforward.
\end{proof}

Using traditional G.Weihrauch reducibility, one can refine Theorem \ref{MP-equial-partial-computable}.
Recall that $\sii{\sf C}_{1/\om}$ corresponds to Markov's principle (Proposition \ref{prop:MP-as-Sigma-1-choice}).

\begin{theorem}\label{thm:MP-remove}
If $F\tgw G+\sii{\sf C}_{1/\om}$ then $F\gw G$.
\end{theorem}

\begin{proof}
Put $M={\sf MP}_{\sf PR}$.
Assume $F\tgw G+M$ via $e$.
\begin{itemize}
\item For each \ppp instance $n\in\dom_+(F)$, the index $e$ yields a \ppp computation tree $(\tpot^e_n,\nu^e_n,\psi^e_n)$.
Hereafter, the superscript $e$ will be omitted.
\item For each \aaa instance $(n|\alpha)\in\dom(F)$, the assumption gives an \aaa computation tree $\tactual_{n|\alpha}\subseteq \tpot_n$.
We inductively construct a leaf $\rho(n|\alpha)\in\leaf\tactual_{n|\alpha}$ by using the $G$-computation $\Gamma_e^G(n|\alpha)$ as follows:
\begin{itemize}
\item Let $\rho_0$ be the root (the empty string) of $\tactual_{n|\alpha}$.
Inductively assume that we have already given a node $\rho_s\in \tactual_{n|\alpha}$ of length $s$.
\item At an internal node $\rho_s\in\internal\tactual_{n|\alpha}$, if $\nu_n(\rho_s)$ makes a query $q$ to $G$, we must 
have $q\in\dom(G)$ since this is an \aaa computation.

After seeing this, $\Gamma_e^G(n|\alpha)$ makes the same query to $G$; then $\Gamma_e^G(n|\alpha)$ recieves a response $r\in G(q)$ from $G$.
Then, put $\rho_{s+1}=\rho_s\fr r$, which inductively ensures $\rho_{s+1}\in\tactual_{n|\alpha}$.
\item At an internal node $\rho_s\in\internal\tactual_{n|\alpha}$, if $\nu_n(\rho_s)$ makes a query $q$ to $M$, we must have $q\in{\rm dom}(M)$ since this is an \aaa computation.

After seeing this, $\Gamma^G$ wait for a stage $t$ at which the computation $\varphi_q(0)$ halts.
As $q\in{\rm dom}(M)$ implies $\varphi_q(0)\down$, we eventually recognize such a stage $t$, and by the definition of $M={\sf MP}_{\sf PR}$, we have $t\in M(q)$.
Then. put $\rho_{s+1}=\rho_s\fr t$, which inductively ensures $\rho_{s+1}\in\tactual_{n|\alpha}$.
\item If $\rho_s$ is a leaf of $\tactual_{n|\alpha}$, put $\rho(n|\alpha)=\rho_s$.
\end{itemize}
\item This construction produces a leaf $\rho(n|\alpha)\in\leaf\tactual_{n|\alpha}$ of the \aaa computation tree.
Finally, $\Gamma_e^G(n|\alpha)$ outputs $\psi_n(\rho(n|\alpha))$.
Since $(\tactual_{n|\alpha},\nu_n,\psi_n)$ solves $F(n|\alpha)$, we get $\psi_n(\rho(n|\alpha))\in F(n|\alpha)$.
\item Each computation path of $\Gamma_e^G(n|\alpha)$ chooses a leaf $\rho(n|\alpha)$ and outputs $\psi_n(\rho(n|\alpha))\in F(n|\alpha)$.
Therefore, we obtain $\Gamma_e^G(n|\alpha)\subseteq F(n|\alpha)$, which shows $F\leq_{\sf W}^{\sf G}G$.
\end{itemize}
%
\end{proof}

\begin{definition}
We say that a tree $T\subseteq X^{<\om}$ is {\bf $n$-fat} if each internal node $\sigma\in \internal T$ has at least $n$ immediate successors;
that is, there are $n$ distinct elements $i_1,i_2,\dots,i_n\in X$ such that $\sigma\fr i_k\in T$ for each $k\leq n$.
\end{definition}

Regarding the forward direction of the following, although Kihara \cite[Proposition 4.4]{Kih23} essentially shows the contrapositive, we provide a direct proof here because a uniform algorithm for the forward direction will be needed later.

\begin{theorem}\label{MP-WLEM-single-pc}
For a partial function $f\pcolon\om\to\om$:
\[f\tgw\sii{\sf C}_{1/\om}+{\sf C}_{2/3}\iff\mbox{$f$ is computable.}\]
\end{theorem}

\begin{proof}
($\Leftarrow$)
By Theorem \ref{MP-equial-partial-computable}, if $f$ is computable then $f\tgw\sii{\sf C}_{1/\om}$; in particular, $f\tgw\sii{\sf C}_{1/\om}+{\sf C}_{2/3}$.

\medskip

($\Rightarrow$)
Assume $f\tgw\sii{\sf C}_{1/\om}+{\sf C}_{2/3}$ via $e$.
By Theorem \ref{thm:MP-remove}, we get $f\gw{\sf C}_{2/3}$, and moreover, one can effectively compute an index of such a reduction from $e$.
This satisfies the following:
\begin{itemize}
\item For each instance $n\in\dom(f)$, an index $e$ yields a strategy $(\nu^e_n,\psi^e_n)$.
Hereafter, the superscript $e$ will be omitted.
\item Since \aaa computation tree $\tactual_n\subseteq 3^{<\om}$ is finitely-branching and well-founded, by weak K\"onig's lemma, it is a finite tree.
In addition, it is ${\sf C}_{2/3}$-branching; that is, a $2$-fat tree.
Furthermore, for each leaf $\sigma\in\leaf\tactual_n$, the computation $\psi_n(\sigma)$ halts with output $f(n)$.
Therefore, $\psi_n$ outputs a unique value on the leaves of the $2$-fat finite tree $\tactual_n$.
\end{itemize}

We construct an algorithm $\Theta_e(n)$ computing $f(n)$.
This is given by the following brute-force algorithm:
\begin{enumerate}
\item Search for a $2$-fat finite tree $S \subseteq 3^{<\om}$ satisfying the following:
$\psi_n$ is constant on the set $\leaf S$ of leaves.
\item If such an $S$ is found, this means that there is $c$ such that $\psi_n(\sigma)\down=c$ for any $\sigma\in\leaf S$.
Compute such a value $c$; then put $\Theta_e(n)=c$.
\end{enumerate}

{\em Verification:}
\begin{itemize}
\item $S=\tactual_n$ satisfies the condition in the item (1); hence, the algorithm eventually reaches at (2) for some $S$.
Hence, $\Theta_e(n)$ must always be defined.
\item Although the $S$ found in (2) is not necessarily the same as $\tactual_n$, since both $S$ and $\tactual_n$ are $2$-fat subtrees of $3^{<\om}$, they share a common leaf $\sigma\in \leaf S\cap\leaf\tactual_n$.
Our definition gives $\Theta_e(n)=\psi_n(\sigma)$, and since $\sigma\in\leaf\tactual_n$ we also have $\psi_n(\sigma)=f(n)$.
\end{itemize}

Consequently, $f$ is computable by the algorithm $\Theta_e$.
One can effectively find an index of $\Theta_e$ from $e$, so hereafter, we refer to this $\Theta_e$ as $\varphi_{\ep(e)}$.
\end{proof}

If we remove $\sii{\sf C}_{1/\om}$, we can guarantee total computability.

\begin{theorem}\label{MP-WLEM-single-pc2}
For a partial function $f\pcolon\om\to\om$:
\[f\tgw{\sf C}_{2/3}\iff\mbox{$f$ is total computable.}\]
\end{theorem}

\begin{proof}
Assume $f\tgw{\sf C}_{2/3}$ via $e$.
This satisfies the following:
\begin{itemize}
\item Since the \ppp domain of $f$ is $\om$, for each \ppp instance $n\in\om$, the index $e$ yields a total ${\sf C}_{2/3}$-computation tree $(\tpot^e_n,\nu^e_n,\psi^e_n)$.
Hereafter, the superscript $e$ will be omitted.
\begin{itemize}
\item {\em Potential.} $\tpot_n\subseteq 3^{<\om}$ is a full $3$-branching well-founded tree.
By weak K\"onig's lemma, this is finite tree.
At each leaf of this tree, $\psi_n$ defines some value.
\item {\em Actual.} For each \aaa instance $n\in\dom(f)$, we have an \aaa computation tree $\tactual_n\subseteq\tpot_n$ such that $\psi_n$ outputs $f(n)$ on its leaves.
\item Since this \aaa computation tree is $2$-fat, this implies that $\psi_n$ is constant on the leaves of some $2$-fat subtree of $\tpot_n$.
\end{itemize}
\end{itemize}

We construct an algorithm $\Theta_e$ computing a total extension of $f$.
This is given by the following brute-force algorithm:
\begin{enumerate}
\item Search for a $2$-fat finite tree $S\subseteq \tpot_n$ satisfying the following:
$\psi_n$ is constant on the set $\leaf S$ of leaves.
\item If such an $S$ is found, this means that there is $c$ such that $\psi_n(\sigma)\down=c$ for any $\sigma\in\leaf S$.
Compute such a value $c$; then put $\Theta_e(n)=c$.
\item If the value of $\psi_n$ is not uniquely determined on the leaves for any $2$-fat finite subtree $S\subseteq \tpot_n$, then put $\Theta_e(n)=0$.
This is possible since $\tpot_n$ is finite and $\psi_n$ is total on its leaves.
\end{enumerate}

{\em Verification:}
\begin{itemize}
\item  For each \ppp instance $n\in\om$, since $\tpot_n$ is a finite tree, there are finitely many $S\subseteq \tpot_n$; therefore, determining whether (2) or (3) holds can be done in finite steps.
Thus, $\Theta_e(n)$ is always defined.
\item For each \aaa instance $n\in\dom(f)$, the \aaa computation tree $S=\tactual_n$ satisfies the condition in the item (1); hence, the algorithm eventually reaches at (2) for some $S$.
Hence, $\Theta_e(n)$ must always be defined.
\item Although the $S$ found in (2) is not necessarily the same as $\tactual_n$, both $S$ and $\tactual_n$ are $2$-fat subtrees of $3^{<\om}$, they share a common leaf $\sigma\in \leaf S\cap\leaf\tactual_n$.
Our definition gives $\Theta_e(n)=\psi_n(\sigma)$, and since $\sigma\in\leaf\tactual_n$ we also have $\psi_n(\sigma)=f(n)$.
\end{itemize}

Consequently, $f$ is computable by the algorithm $\Theta_e$.
One can effectively find an index of $\Theta_e$ from $e$, so hereafter, we refer to this $\Theta_e$ as $\varphi_{\ep(e)}$.
\end{proof}

\section{Computability (Separations)}\label{sec:computable-separation}

We first show that the $\Sigma_1$-co-unique choice problem $\sii{\sf C}_{\om-1/\om}$ (which is linked to ${\sf MP}_\om^\lor$ by Theorem \ref{thm:disjunctive-Markov-counique}) is not total computable.

\begin{theorem}\label{thm:MPomega-not-total-computable}
$\sii{\sf C}_{\om-1/\om}\not\tgw{\rm id}_\om$.
\end{theorem}

\begin{proof}
By Proposition \ref{prop:below-id-total-computable-choice}, it suffices to show $\sii{\sf C}_{\om-1/\om}$ is not total computable.
Let $\varphi_k\colon\om\to\om$ be a total computable function.
By using the recursion theorem, construct the following algorithm $e$ involving a self-reference:
\begin{enumerate}
\item Wait for the computation $\varphi_k(e)$ to halt.
\item If we recognize this, enumerate all elements other than $\varphi_k(e)$ into $S_e$.
Then $S_e=\om\setminus\{\varphi_k(e)\}$.
\end{enumerate}

{\em Verification:}
\begin{itemize}
\item By totality of $\varphi_k$, the computation $\varphi_k(e)$ must halt.
Hence, the algorithm $e$ must reach at (2), so we have $S_e=\om\setminus\{\varphi_k(e)\}$, which means $e\in\dom(\sii{\sf C}_{\om-1/\om})$.
However, we have $\varphi_k(e)\not\in S_e=\sii{\sf C}_{\om-1/\om}(e)$; hence, $\varphi_k$ cannot be a total extension of $\sii{\sf C}_{\om-1/\om}$.
%
\end{itemize}

Consequently, we get $\sii{\sf C}_{\om-1/\om}\not\tgw{\rm id}_\om$.
%
%
\end{proof}

We next show that $\sii{\sf C}_{1/\om}$ (linked to Markov's principle ${\sf MP}$) is not derivable from ${\sf C}_{1/2}$ (linked to the weak law of excluded middle ${\sf WLEM}$).

\begin{theorem}\label{thm:MP-WLEM-separation}
$\sii{\sf C}_{1/\om}\not\tgw{\sf C}_{1/2}$.
In particular, $\sii{\sf C}_{1/\om}\not\tgw\pii{\sf C}_{1/2}+\sii{\sf C}_{1/2}$.
\end{theorem}

We use the following lemma, which is proven by a standard compactness argument.

\begin{lemma}\label{lem:finite-codomain-CT-maj-comp}
Assume that the codomain of a potted bilayer problem $\mathcal{U}$ is finite.
Then, for any potted problem $F\pcolon\om\tto\om$, if $F\tgw\mathcal{U}$ then $F$ has a solution in the range bounded by a total computable function $g$; in other words,
\[\dom_+(F)\subseteq\dom(g);\qquad x\in\dom(F)\implies \exists y\leq g(x).\ y\in F(x).\]
\end{lemma}

\begin{proof}
Assume that $\mathcal{U}$ always takes values less than $k$, and its \ppp codomain is $H$ with $k\subseteq H\subseteq\om$.
In this case, $h:=H\cap k$ is a finite set.
Assume that $F\tgw\mathcal{U}$ via $e$.
%
\begin{itemize}
\item For each \ppp instance $n\in\dom_+(F)$, the index $e$ yields a \ppp computation tree $(\tpot^e_n,\nu^e_n,\psi^e_n)$.
Hereafter, the superscript $e$ will be omitted.
\item Since $H_n:=\tpot_n\cap h^{<\om}$ is a finitely-branching well-founded tree, by weak K\"onig's lemma, it is a finite tree.
Since the index $e$ specifies a computation procedure for $n\mapsto \tpot_n$, and $h$ is a fixed finite set, $n\mapsto H_n$ is also computable.
\item At each leaf $\sigma\in\leaf H_n$, the value $\psi_n(\sigma)$ is defined.
By finiteness of $H_n$, the set $\{\psi_n(\sigma):\sigma\in H_n\}$ is also finite.
\end{itemize}

Then, consider the following algorithm $\Xi_e$:
\begin{enumerate}
\item
For each \ppp instance $n\in\dom_+(F)$, computes the finite tree $H_n$.
Wait for seeing $\psi_n(\sigma)\down$ for all leaves $\sigma\in\leaf H_n$.
\item After observing the termination of $\psi_n$ on all leaves, we get the finite set $R_n=\{\psi_n(\sigma):\sigma\in\leaf H_n\}$.
Let $\Xi_e(n)$ be the maximum value $\max R_n$.
\end{enumerate}

{\em Verification:}
\begin{itemize}
\item For each \ppp instance $n\in\dom_+(F)$, $\psi_n$ is total on $\leaf \tpot_n$; in particular, for each $\sigma\in\leaf H_n\subseteq\leaf \tpot_n$, the computation $\psi_n(\sigma)$ halts at stage $s_\sigma$.
Since $H_n$ is finite, at the maximum stage $s=\max\{s_\sigma:\sigma\in H_n\}$, we recognize the termination of $\psi_n$ on all leaves; hence, the algorithm proceeds to (2).
Therefore, $\Xi_e(n)$ eventually outputs some value.

\item For each \aaa instance $n\in\dom(F)$, the \aaa computation tree $\tactual_n$ induced by {\tt Nimue}'s strategy $\zeta$ must satisfy $\tactual_n\subseteq H_n$; hence, we get
\[\Phi_{e|\zeta}^\mathcal{U}(n)=\psi_n[\leaf\tactual_n]\subseteq \psi_n[\leaf H_n]=R_n.\]
Therefore, regardless of $\zeta$, for any $y\in\Phi_{e|\zeta}^\mathcal{U}(n)$, we have $y\leq\max R_n=\Xi_e(n)$.
\item 
By our assumption $F\tgw\mathcal{U}$ via $e$, there is $\zeta$ such that $\Phi_{e|\zeta}^\mathcal{U}(n)\subseteq F(n)$.
Since $F$ is nonempty, there is some $y\in\Phi_{e|\zeta}^\mathcal{U}(n)\subseteq F(n)$; hence, we get $y\leq\Xi_e(n)$ and $y\in F(n)$.
\end{itemize}

Consequently, $\varphi_{\beta(e)}:=\Xi_e$ gives a total computable upper bound for a solution of $F$.
%
\end{proof}

\begin{proof}[Proof (Theorem \ref{thm:MP-WLEM-separation})]
We show that Markov's principle $\sii{\sf C}_{1/\om}$ has no total computable upper bound.
\begin{itemize}
\item Let $g\colon\om\to\om$ be a total computable function.
\item By the recursion theorem, there is $e$ such that the $\Sigma_1$ set $S_e=\{g(e)+1\}$ has an index $e$.
\item Then, $e\in\dom(\sii{\sf C}_{1/\om})$, and for any $y\leq g(e)$, we have $y\not\in\sii{\sf C}_{1/\om}(e)=S_e$; hence, $g$ does not give an upper bound.
\end{itemize}

Since the codomain of ${\sf C}_{1/2}$ is finite, by Lemma \ref{lem:finite-codomain-CT-maj-comp}, we obtain $\sii{\sf C}_{1/\om}\not\tgw{\sf C}_{1/2}$.
\end{proof}

The following combinatorial fact is easy to prove but quite useful.

\begin{fact}[see {\cite[Proposition 2.9]{CeHi08}} and {\cite[Fact 4.4]{Kih23}}]\label{fact:fat-tree-lemma}
Let $T$ be a $(n+1)$-fat finite tree.
Then, for any function $\psi\colon \leaf T\to n$, there exists a $2$-fat subtree $U\subseteq T$ such that $\psi$ is constant on the set $\leaf U$ of its leaves.
\end{fact}

Using this fact, we show that $\sii{\sf C}_{k-1/k}$ (linked to Markov's principle ${\sf MP}^\lor_k$) is not derivable from ${\sf C}_{k/k+1}$ (linked to ${{\sf WLEM}}_{k+1}$).

\begin{theorem}\label{thm:separation-MPlorn-WLEMn}
$\sii{\sf C}_{k-1/k}\not\tgw{{\sf C}}_{k/k+1}$.
\end{theorem}

\begin{proof}
We assume $\sii{\sf C}_{k-1/k}\not\tgw{{\sf C}}_{k/k+1}$ via $\Phi(n)=(\tpot_n,\nu_n,\psi_n)$.
We may assume that the potted codomains of $\sii{\sf C}_{k-1/k}$ and ${\sf C}_{k/k+1}$ are $k$ and $k+1$.
\begin{itemize}
\item Since $\tpot_n\subseteq(k+1)^{<\om}$ is a finitely-branching well-founded tree, it is a finite tree.
\item For each \aaa instance $n\in\dom(\sii{\sf C}_{k-1/k})$, an \aaa computation tree $\tactual_n\subseteq \tpot_n$ is $k$-fat.
\item At each leaf $\sigma\in\leaf \tpot_n$, the computation $\psi_n(\sigma)$ always halts and $\psi_n(\sigma)<k$.
\end{itemize}
By using the recursion theorem, construct the following algorithm $e$ involving a self-reference:
\begin{enumerate}
\item Wait for seeing $\psi_e(\sigma)\down$ for all leaves $\sigma\in\leaf \tpot_e$.
Before observing this, keep $S_e=\emptyset$.
\item If we recognize this,
Applying Fact \ref{fact:fat-tree-lemma} to $\psi_e\colon\leaf{\tpot_e}\to k$, we see that there must exist a $2$-fat subtree $U\subseteq \tpot_e$ such that $\psi_e$ is constant on its leaves.

\item By brute-force, search for such a finite tree $U$.
Then there is a unique $c<k$ such that $\psi_e(\sigma)=c$ for any leaf $\sigma\in\leaf{U}$.
Then we set the $\Sigma_1$ set $S_e$ as $S_e=\{n<k:n\not=c\}$.
\end{enumerate}

{\em Verification:}
\begin{itemize}
\item Since $e\in\om=\dom_+(\sii{\sf C}_{k-1/k})$ is a \ppp instance, $\psi_e$ is total on $\leaf \tpot_e$; thus, for any $\sigma\in\leaf \tpot_e$, the computation $\psi_e(\sigma)$ halts at some stage $s_\sigma$.
Since $\tpot_e$ is finite, at the maximum stage $s=\max\{s_\sigma:\sigma\in \tpot_e\}$, we recognize the termination of $\psi_e$ on all leaves; hence, the algorithm proceeds to (2).
Moreover, the algorithm eventually finds an $U$ as in Fact \ref{fact:fat-tree-lemma}, so the algorithm proceeds to (3); hence, $S_e$ becomes a co-singleton, which means $e\in\dom(\sii{\sf C}_{k-1/k})$.
\item Since $U\subseteq(k+1)^{<\om}$ is $2$-fat and $\tactual_e\subseteq(k+1)^{<\om}$ is $k$-fat, there must be a common leaf $\sigma\in \leaf U\cap \leaf\tactual_e$.
As $\sigma\in \leaf U$, by our construction of $S_e$, we have $\psi(\sigma)\not\in S_e$ while $\sigma\in\leaf\tactual_e$.
\end{itemize}

Consequently, we obtain $\Phi^\U(e)\not\subseteq S_e=\sii{\sf C}_{k-1/k}(e)$.
This shows $\Sigma_1\text{-}{\sf C}_{k-1/k}\not\leq_{\sf tGW}{\sf C}_{k/k+1}$.
\end{proof}

\begin{cor}\label{cor:LLPOk-not-reduce-WLEMkp1}
$\pii{\sf C}_{\geq k-1/k}\not\tgw{\sf C}_{k/k+1}$.
\end{cor}

\begin{proof}
Observation \ref{obs:basic-relations-choice} and Theorem \ref{thm:basic-Sigma-1-counique} imply $\sii{\sf C}_{k-1/k}\tW\pii{\sf C}_{\geq k-1/k}$; hence the assertion directly follows from Theorem \ref{thm:separation-MPlorn-WLEMn}.
\end{proof}

The above is a trivial consequence of Theorem \ref{thm:separation-MPlorn-WLEMn}, but we emphasize here that it does not hold with respect to one-query Weihrauch reducibility $\leq_{\sf W}$ (\cite[Proposition 4.3]{Kih23}).
In fact, combining Proposition \ref{prop:LLPOn-reducible-WLEMnplusone} and Theorem \ref{thm:MP-remove}, we get $\pii{\sf C}_{\geq k-1/k}\leq_{\sf W}^{\sf G}{\sf C}_{k/k+1}$.

\begin{prop}\label{prop:separation-Eff-LLPO-n-nplusone}
$\pii{\sf C}_{\geq n-1/n}\not\tgw\pii{\sf C}_{\geq n/n+1}+\sii{\sf C}_{1/\om}$.
\end{prop}

\begin{proof}
Otherwise, by Theorem \ref{thm:MP-remove} we get $\pii{\sf C}_{\geq n-1/n}\not\leq_{\sf W}^{\sf G}\pii{\sf C}_{\geq n/n+1}$, which contradicts with a known fact (\cite[Proposition 7.18]{BGP21}).
\end{proof}

\begin{prop}\label{prop:separation-EFF-LLPO-n-WLEM}
$\pii{\sf C}_{\geq n-1/n}\not\tgw{\sf C}_{n+1/n+2}+\sii{\sf C}_{1/\om}$.
\end{prop}

\begin{proof}
Otherwise, by Theorem \ref{thm:MP-remove}, we get $\pii{\sf C}_{\geq n-1/n}\gw{\sf C}_{n+1/n+2}$, which contradicts with a known fact (\cite[Proposition 4.5]{Kih23}).
\end{proof}

%

\begin{theorem}\label{thm:separation-WLEMn-nplusone}
${\sf C}_{n-1/n}\not\tgw{\sf C}_{n/n+1}+F$ for any potted problem $F\pcolon\om\tto\om$.
\end{theorem}

\begin{proof}
${\sf C}_{n-1/n}\not\gw{\sf C}_{n/n+1}+F$ is shown in Kihara \cite[Proposition 4.6]{Kih23}.
\end{proof}

\begin{theorem}\label{thm:separation-DML-WLEM}
${\sf C}_{1/\om}\not\tgw{\sf C}_{1/2}+\sii{\sf C}_{1/\om}$.
\end{theorem}

\begin{proof}
Assume ${\sf C}_{1/\om}\tgw{\sf C}_{1/2}+\sii{\sf C}_{1/\om}$.
We may assume that the \ppp codomain of ${\sf C}_{1/2}$ is $2$.
\begin{itemize}
\item The public instance of ${\sf C}_{1/\om}$ is unique, so all computations take place within a single computation tree $\Phi=(\tpot,\nu,\psi)$.
For each secret instance $c\in\om$, $\psi$ outputs a value ${\sf C}_{1/\om}(\ast|c)=c$ on all leaves of an \aaa computation tree $\tactual_c\subseteq \tpot$.
\item $\tactual_c$ is at most $2$-branching:
A node with a $\sii{\sf C}_{1/\om}$-query is $1$-branching, and a node with ${\sf C}_{1/2}$-query is $2$-branching.
Therefore, $\tactual_c$ is a finite-branching well-founded tree, so it is a finite tree.
\end{itemize}

The key point here is that since the computable strategy $\nu$ does not depend on the secret data $c$, a query assigned to a $\sii{\sf C}_{1/\om}$-node do not depend on $c$.
Now, we inductively define a subtree $H\subseteq\tpot$ as follows:
\begin{enumerate}
\item The empty string belongs to $H$.
\item At an internal node $\sigma\in H\cap\internal\tpot$, if $\nu(\sigma)$ makes a query to $\sii{\sf C}_{1/\om}$, for a unique solution $\sii{\sf C}_{1/\om}(\nu(\sigma))=n$, we put $\sigma\fr n\in H$.
\item If $\nu(\sigma)$ makes a query to ${\sf C}_{1/2}$, then put $\sigma\fr 0,\sigma\fr 1\in H$.
\end{enumerate}

This tree $H$ has the following properties:
\begin{itemize}
\item $H\subseteq\tpot$ is an at most $2$-branching well-founded tree, so it is a finite tree.
\item For each $c\in\om$, we have $\tactual_c\subseteq H$.
\end{itemize}

Based on this observation, we select a value $c\in\om$ as follows:
\begin{itemize}
\item For any leaf $\sigma\in\leaf\tpot$, the computation $\psi(\sigma)$ halts.
By finiteness of $H\subseteq \tpot$, the set $V=\{\psi(\sigma):\sigma\in \leaf H\}$ is finite.
Therefore, there must exist a number $c\not\in V$, so we select such a $c$.
\end{itemize}

{\em Verification:}
\begin{itemize}
\item An \aaa computation tree depends on $c$, but regardless of $c$, we must have $\tactual_c\subseteq H$.
Therefore, we have $\psi(\sigma)\in V$ for any leaf $\sigma\in\leaf{\tactual_c}\subseteq\leaf H$.
\item However, we have selected $c\not\in V$, which implies $\psi(\sigma)\not=c$.
\end{itemize}
Consequently, we obtain ${\sf C}_{1/\om}\not\tgw{\sf C}_{1/2}+\sii{\sf C}_{1/\om}$.
\end{proof}

The above result is interesting.
This is because, in fact, it is possible to use a typical {\tt while} loop to generate unbounded queries to ${\sf C}_{1/2}$ to compute a solution to ${\sf C}_{1/\om}$.

\begin{prop}\label{prop:DML-gw-WLEM}
${\sf C}_{1/\om}\gw{\sf C}_{1/2}$.
\end{prop}

\begin{proof}
We perform the following play:
\begin{itemize}
\item {\em Round $0$.} {\tt Merlin}'s opening move is $(\ast|c)$.
\item {\em Round $n$.}
\begin{itemize}
\item If either $n=0$ or {\tt Merlin}'s last move is $0$, then {\tt Arthur} asks {\tt Nimue} to make a query.
{\tt Nimue} reads the secrete data $c$, and produces the truth value of the equality $n=c$.
That is, if this is true, {\tt Arthur}-{\tt Nimue} make a query $(\ast|1)$ to ${\sf C}_{1/2}$; otherwise, they make a query $(\ast|0)$ to ${\sf C}_{1/2}$.
\item If {\tt Merlin}'s last move is $1$, then {\tt Arthur} declares to terminate the game and outputs the current round $n$.
\end{itemize}
\end{itemize}

{\em Verification:}
\begin{itemize}
\item As ${\sf C}_{1/2}(\ast|i)=i$, the game terminates at the first round $n$ such that {\tt Arthur}-{\tt Nimue} make a query $(\ast|1)$.
In this case, the output is $n$.
\item 
Since this occurs only when $n=c$, the output must be $c={\sf C}_{1/\om}(\ast|c)$.
\end{itemize}
\end{proof}

This also implies that the reverse direction of Theorem \ref{thm:MP-remove} does not hold.
In other words, $\gw$ cannot be characterized as $\tgw$ equipped with Markov's principle ${\sf C}_{1/\om}$.
Intuitively, within a reduction $\tgw$ equipped with Markov's principle, while a plain unbounded search is possible this does not means that a $G$-relative unbounded search is possible.

\begin{cor}
There exist potted bilayer problems $F$ and $G$ such that $F\gw G$ but $F\not\tgw G+\sii{\sf C}_{1/\om}$.
\end{cor}

\begin{proof}
Consider $F={\sf C}_{1/\om}$ and $G={\sf C}_{1/2}$.
Then the assertion follows from Theorem \ref{thm:separation-DML-WLEM} and Proposition \ref{prop:DML-gw-WLEM}.
\end{proof}

Next, we make a slight modification to Lemma \ref{lem:finite-codomain-CT-maj-comp}.

\begin{definition}
A partial function $g\pcolon \om\to\om$ is {\em lower semicomputable} if there is a computable function $\tilde{g}\colon\om^2\to\om$ such that $g(n)=\max_{t\in\om} g(n,t)$ for any $n\in\dom(g)$.
\end{definition}

\begin{lemma}\label{lem:finite-codomain-CT-maj-comp2}
Assume that the codomain of a potted bilayer problem $\mathcal{U}$ is finite.
Then, for any potted problem $F\pcolon\om\tto\om$, if $F\tgw\sii{\sf C}_{1/\om}+\mathcal{U}$ then $F$ has a solution in the range bounded by a lower semicomputable function $g$; in other words,
\[\dom(F)\subseteq\dom(g);\qquad x\in\dom(F)\implies \exists y\leq g(x).\ y\in F(x).\]
\end{lemma}

\begin{proof}
Assume that $\mathcal{U}$ always takes values less than $k$, and its \ppp codomain is $H$ with $k\subseteq H\subseteq\om$.
In this case, $h:=H\cap k$ is a finite set.
Assume that $F\tgw\sii{\sf C}_{1/\om}+\mathcal{U}$ via $e$.
\begin{itemize}
\item For each \ppp instance $n\in\dom_+(F)$, the index $e$ yields a \ppp computation tree $(\tpot^e_n,\nu^e_n,\psi^e_n)$.
Hereafter, the superscript $e$ will be omitted.
\item We inductively construct a subtree $H_n\subseteq \tpot_n$ as follows:
\begin{enumerate}
\item The empty set belongs to $H_n$.
\item For a internal node $\sigma\in H_n\cap\internal \tpot_n$, if $\nu_n(\sigma)$ makes a query to $\U$, then we declare that $\sigma$ is a full $h$-branching node; that is, we put $\sigma\fr i\in H_n$ for each $i\in h$.
\item For a internal node $\sigma\in H_n\cap\internal \tpot_n$, if $\nu_n(\sigma)$ makes a query $q$ to $\sii{\sf C}_{1/\om}$ and $\sii{\sf C}_{1/\om}(q)$ is a singleton $\{c\}$, then put $\sigma\fr c\in H_n$.
\end{enumerate}
\item For each \aaa instance $n\in\dom(F)$, an \aaa computation tree must satisfy $\tactual_n\subseteq H_n\subseteq \tpot_n$.
Since $H_n$ is a finitely-branching well-founded tree, it is a finite tree.
\item At each leaf $\sigma\in\leaf \tpot_n$, the value $\psi_n(\sigma)$ is defined.
By finiteness of $H_n$, the set $\{\psi_n(\sigma):\sigma\in H_n\cap\leaf \tpot_n\}$ is finite.
\end{itemize}

We construct an algorithm $\Xi_e$ approximating this construction as follows:
\begin{enumerate}
\item
At stage $t$, compute the subtree $H_n[t]\subseteq \tpot_n$ in the same way as above.
Here, at the step (3), we put $\sigma\fr c\in H_n[t]$ only when $S_q=\{c\}$ at stage $t$.
Otherwise, we do not perform any action.
\item 
Compute the values of $\psi_n(\sigma)$ for each leaf $\sigma\in H_n[t]\cap\leaf\tpot_n$, and obtain the finite set $R_n[t]=\{\psi_n(\sigma):\sigma\in H_n[t]\cap\leaf \tpot_n\}$.
Then $\Xi_e(n,t)$ computes the maximum value $\max R_n[t]$.
\end{enumerate}

{\em Verification:}
Let $n\in\dom(F)$ be an \aaa instance.
\begin{itemize}
\item 
One can see $H_n=\bigcup_{t\in\om}H_n[t]$; that is, for any $\sigma\in H_n$, there is a stage $t_\sigma$ such that $\sigma\in H_n[t_\sigma]$.
\item Since $H_n$ is finite, at the maximum stage $s=\max\{t_\sigma:\sigma\in H_n\}$, we get $\tactual_n\subseteq H_n=H_n[s]$.
\item In particular, we get $\bigcup_{t\in\om}R_n[t]=R_n[s]$, so $g(n):=\sup_t\Xi_e(n,t)$ takes the finite value $\max R_n[s]$.
\item The \aaa computation tree induced by {\tt Nimue}'s strategy $\zeta$ must satisfy $\tactual_n\subseteq H_n$; hence, we get
\[\Phi_{e|\zeta}^\mathcal{U}(n)=\psi_n[\leaf\tactual_n]\subseteq \psi_n[\leaf H_n]=R_n.\]

Therefore, regardless of $\zeta$, for any $y\in\Phi_{e|\zeta}^\mathcal{U}(n)$, we have $y\leq\max R_n[s]=g(n)$.
\item By our assumption $F\tgw\mathcal{U}$ via $e$, there is $\zeta$ such that $\Phi_{e|\zeta}^\mathcal{U}(n)\subseteq F(n)$.
Since $F$ is nonempty, there is some $y\in\Phi_{e|\zeta}^\mathcal{U}(n)\subseteq F(n)$; hence, we get $y\leq g(n)$ and $y\in F(n)$.
\end{itemize}

Consequently, $g$ gives a lower semicomputable upper bound for a solution of $F$.
\end{proof}

\section{Lawvere-Tierney Modality}\label{sec:LT-topology}
We have completed {\em computability-theoretic} separation proofs for various principles, but these are strictly computability-theoretic results; we have {\bf not} yet established {\em logical} separations.
The focus from here on is  how to translate computable separations into logical separations.
The key notion is a Lawvere-Tierney modality, and it is through this that logical separations are achieved in the subsequent sections.

For the basics of realizability topos theory, see van Oosten \cite{vOBook}.
For general topos theory, see \cite{SGL,elephant}.

\subsection{Tripos}

Recall from Definition \ref{def:potted-set} that a potted set $\paset{A}$ is a pair $\paset{A}=(A_-;A_+)$ such that $A_-\subseteq A_+\subseteq \N$ and $0\in A_+$.
We sometimes call $n\in A_+$ a \ppp realizer, and $n\in A_-$ an \aaa realizer.
Let $\Omega$ be the collection of all potted sets equipped with the following structure:

\begin{definition}
Heyting operations on $\Omega$ are defined as follows:
\begin{itemize}
\item {\em Top.} $\top=(\N;\N)$.
\item {\em Bottom.} $\bot=(\emptyset;\N)$.
\item {\em Join.} $\paset{A}\times \paset{B}=(A_-\times B_-;\ A_+\times B_+)$.
\item {\em Meet.} $\paset{A}+\paset{B}=(A_-+B_-;\ A_++B_+)$.
\item {\em Implication.} $\paset{A}\totarr \paset{B}=(A_-\arr B_-;\ (A_-\arr B_-)\cap (A_+\arr B_+))$.
\end{itemize}
Here, define $A\arr B=\{e\in\om:\forall n\in A.\ \varphi_e(n)\down\in B\}$ for $A,B$.
We assume that $0$ is a code of a trivial total function (i.e., ${\rm id}_\om$), which ensures $0\in(A_-\arr B_-)$.
As usual, $A\arr\bot$ is abbreviated as $\neg A$.
\end{definition}

We also use the following set operations:
\begin{definition}\label{def:set-operations-on-pasets}
Let $\paset{A},\paset{B},\paset{A}_i$ be potted sets.
\begin{itemize}
\item {\em Inclusion} $\paset{A}\subseteq \paset{B}$ if and only if $A_-\subseteq B_-$ and $A_+\subseteq B_+$.
\item {\em Union.} $\bigcup_{i\in I}\paset{A}_i=(\bigcup_{i\in I}(A_i)_-;\ \bigcup_{i\in I}(A_i)_+)$.
\item {\em Intersection.} $\bigcap_{i\in I}\paset{A}_i=(\bigcap_{i\in I}(A_i)_-;\ \bigcap_{i\in I}(A_i)_+)$.
\end{itemize}
\end{definition}


\begin{definition}
For each set $X$, we call a function $\varphi\colon X\to\Omega$ a {\bf predicate} on $X$.
By equipping with the pointwise Heyting operations, one can think of $\Omega^X$ as a Heyting pre-algebra as well.
\begin{itemize}
\item We say that such a predicate $\varphi$ is {\bf realizable} if $\varphi(x)$ has a common \aaa realizer for any $x\in X$; that is, $\bigcap_{x\in X}\varphi(x)_-$ is nonempty.
\item The set of all predicates $\Omega^X$ on $X$ is equipped with the {\em realizability preorder} $\leq_X$:
\[
\varphi\leq_X\psi\iff\bigcap_{x\in X}(\varphi(x)\arr\psi(x))_-\not=\emptyset.
\]
That is, $\varphi(x)\arr\psi(x)$ is realizable.
\item For a function $f\colon X\to Y$, and a predicate $\varphi\in\Omega^Y$, the {\em substitution} $\varphi\cdot f\in\Omega^X$ is defined as $(\varphi\cdot f)(x)=\varphi(f(x))$.
\end{itemize}  
\end{definition}

\begin{remark}
By the above convention, if $\varphi$ is not realizable, then an index $0$ of the trivial function realizes the negation $\neg\varphi$.
\end{remark}

For preordered sets $P,Q$, a function $\alpha\colon P\to Q$ is monotone if $x\leq_Py$ implies $\alpha(x)\leq_Q\alpha(y)$.

\begin{definition}
A {\bf natural transformation} on $\Omega$ is an assignment of a monotone function $\alpha_X\colon \Omega^X\to\Omega^X$ to each set $X$ which commutes with substitutions; that is, $\alpha_Y(\varphi)\cdot f\equiv_X\alpha_X(\varphi\cdot f)$ holds for any function $f\colon X\to Y$ and predicate $\varphi\in\Omega^Y$.
\end{definition}

\begin{obs}\label{obs:monotone-to-transformation}
Let $\alpha\colon\Omega\to\Omega$ be a monotone function.
Define $\alpha_X(\varphi)(x)=j(\varphi(x))$ for any $\varphi\in\Omega^X$ and $x\in X$.
Then, the assignment $X\mapsto \alpha_X$ is a natural transformation on $\Omega$.
\end{obs}

\begin{proof}
For each $f\colon X\to Y$ and $\varphi\in\Omega^Y$, we have $\alpha_Y(\varphi)(f(x))=\alpha(\varphi(f(x)))=\alpha_X(\varphi\cdot f)(x)$.
\end{proof}

\begin{definition}\
A function $j\colon \Omega\to \Omega$ is a {\bf Lawvere-Tierney modality} if it is an {internal closure operator}; that is,
\begin{enumerate}
\item {\em Monotone.} 
$(p\arr q)\arr(j(p)\arr j(q))$ is realizable.
\item {\em Inflationary.}
$p\arr j(p)$ is realizable.
\item {\em Idempotent.}
$j(j(p))\arr j(p)$ is realizable.
\end{enumerate}

As in Observation \ref{obs:monotone-to-transformation}, a LT-modality $j$ always yields a natural transformation $X\mapsto j_X$ endowed with some ${\tt m},\eta,\mu\in\om$ such that for any set $X$,
\begin{enumerate}
\item {\em Monotone.} ${\tt m}\in\bigcap_{\varphi,\psi\in\Omega^X}\bigcap_{x\in X}\big[(\varphi(x)\arr\psi(x))\arr[j_X\varphi(x)\arr j_X\psi(x)]\big]$.
\item {\em Inflationary.} $\eta\in\bigcap_{\varphi\in\Omega^X}\bigcap_{x\in X}(\varphi(x)\arr j_X\varphi(x))$.
\item {\em Idempotent.} $\mu\in\bigcap_{\varphi\in\Omega^X}\bigcap_{x\in X}[j_Xj_X\varphi(x)\arr j_X\varphi(x)]$.
\end{enumerate}
Here, $j_X(\varphi)$ is abbreviated as $j_X\varphi$.
Such a transformation is called a {\bf closure transformation} \cite{vOBook}.
Hereafter, $j_X$ is often abbreviated as $j$.
\end{definition}

\begin{definition}
For a LT-modality $j$ and a formula $\varphi$, inductively define $\eval{\varphi}^j$ as follows:
\begin{itemize}
\item $\eval{\top}^j=\top$.
\item $\eval{\bot}^j=j({\bot})$.
\item $\eval{p\land q}^j=\eval{p}^j\times\eval{q}^j)$.
\item $\eval{p\lor q}^j=\eval{p}^j+\eval{q}^j)$
\item $\eval{p\to q}^j=\eval{p}^j\arr j(\eval{q}^j)$.
\item $\eval{\exists x\varphi(x)}^j=\bigcup_{x\in X}\eval{\varphi(x)}^j$
\item $\eval{\forall x\varphi(x)}^j=\bigcap_{x\in X}j(\eval{\varphi(x)}^j)$
\end{itemize}

We say that a predicate $\varphi$ is {\bf $j$-realizable} if $j(\eval{\varphi}^j)$ has an \aaa realizer.
This is called the {\bf Kuroda $j$-translation} \cite{vdB19}.
\end{definition}

Note that $j$ is involved only with $\bot,\to,\forall$.
In this article, we consider only the case where $j$ is $\neg\neg$-dense, that is, where $j(\bot)=\bot$.
In this case, $j$ is involved only with $\to,\forall$.

\begin{obs}\label{obs:negneg-real-truth}
Let $j$ be $\neg\neg$-dense and $\varphi$ be a formula.
Then $\varphi$ is true if and only if the $\neg\neg$-translation $j(\eval{\varphi}^{\neg\neg})$ is realizable (by the trivial realizer $0$).
\end{obs}

\begin{remark}
A LT-modality/closure transformation can be regarded as an operation that transforms the truth values of predicates.
Its role here is to transform the semantics of a {\em computable universe} into the semantics of an {\em oracle-relative computable universe}.
In other words, a LT-modality serves the role of an oracle relativization \cite{Kih23,Kih22,AhBa26}.
This view is developed in Sections \ref{sec:topology-oracle} and \ref{sec:topos-sheafsubtopos}.
\end{remark}

\subsection{LT-topologies as oracles}\label{sec:topology-oracle}

The key idea is that a LT-modality is an oracle.
There is a way to explain the link between a LT-modality and an oracle in abstract terms.
As we have already seen, a multi-query oracle computation is a free monad $\lambda V.\mu S.(\sum_{x\in X}S^{F(x)}+V)$.
Similarly, a closure operator can also be understood as a monad.
Indeed, each oracle yields a ``free LT-modality'' $j_F(p)=\mu s.((\exists x\in X.\ F(x)\arr s)\lor p)$ and this is expected to provide a semantics for the $F$-relative computable universe.
See also Ahman-Bauer \cite{AhBa26}.

Once we realize that the shape of a free LT-modality is identical to that of a free monad, we see that it can be described as an oracle computation.
To be explicit, for an oracle $F$, $j_F(P)$ can be described as the set of all {\em codes of strategies for solving $P$ using the oracle $F$}.
In the concrete analysis, it is precisely this explicit description that is important; for example, the structural analysis of the LT topologies on the effective topos has advanced rapidly and significantly thanks to the discovery of this explicit description \cite{LvO13,Kih23,KiNg26}.

While the above is a plausible abstract idea, unfortunately our category is not quite that perfect, so we set aside the abstract discussion for now and treat it simply as the underlying idea.
Let us actually get to work and verify that this idea works correctly even within the setting of total computability.

\begin{terminology}
Let us review the terminology related to a computation tree.
Let $\paset{P}$ be a potted set, and $F\pcolon \paset{X}\times\Lambda\tto Y$ be a potted bilayer problem.
A \ppp computation tree of type $P_+$ is the following triple $(\tpot,\nu,\psi)$:
\begin{enumerate}
\item $\tpot\subseteq Y^{<\om}$ is a nonempty full-branching well-founded tree.
\item $\nu\colon \internal \tpot\to X_+$ is a total labeling function on the internal nodes.
\item $\psi\colon \leaf \tpot\to P_+$ is a total labeling function on the leaves.
\end{enumerate}

An \aaa computation tree by $F$ is an $(F;\nu)$-branching subtree $\tactual\subseteq\tpot$; that is,
\begin{enumerate}\setcounter{enumi}{3}
\item The empty string belongs to $\tactual$.
\item To each internal node $\sigma\in \tactual\cap\internal\tpot$, a query $\nu(\sigma)\in X_-$ is assigned, and
\[\{n\in Y:\sigma\fr n\in \tactual\}\in \overline{F}_{\nu(\sigma)}.\]
In other words,
there is a secret $\alpha\in\Lambda$ such that each $n\in F(\nu(\sigma)|\alpha)$ provides an immediate successor of $\sigma$.
\end{enumerate}

We say that a \ppp computation tree $(\tpot,\nu,\psi)$ $F$-solves $P_-$ if:
\begin{enumerate}\setcounter{enumi}{5}
\item There is an $(F;\nu)$-branching subtree $\tactual\subseteq\tpot$ such that $\psi(\rho)\down\in P_-$ for any leaf $\rho\in\leaf\tactual$.
\end{enumerate}
\end{terminology}



\begin{notation}
For each computation $\Phi_e=(\tpot^e,\nu^e,\psi^e)$ given by an index $e$,
\begin{itemize}
\item If $\Phi_e$ is a \ppp computation tree, we write $\Phi_e^+\down$.
It is of the type $P_+$ if 
\[\Phi_e^+=\psi^e[\leaf\tpot]=\{\psi^e(\sigma):\sigma\in\leaf \tpot\}\subseteq P_+.\]
\item For each $(F;\nu)$-branching subtree $\tactual\subseteq\tpot$, we write:
\begin{align*}
&\Phi_{e|\tactual}^F\down\iff\forall\sigma\in\leaf\tactual.\ \psi^e(\sigma)\down,\\
&\Phi_{e|\tactual}^F=\psi^e[\leaf\tactual]=\{\psi^e(\sigma):\sigma\in\leaf\tactual\}.
\end{align*}
\end{itemize}
\end{notation}

\begin{definition}
Let $\paset{P}$ be potted set, and $F\colon\paset{X}\times\Lambda\tto Y$ be a potted bilayer problem.
\begin{itemize}
\item Let $F^\solve_+(P_+)$ denote the set of all code of \ppp computation trees $(\tpot,\nu,\psi)$ of type $P_+$.
\item Let $F^\solve(P_-)$ denote the set of all codes of such trees that $F$-solve $P_-$.
\end{itemize}
\end{definition}

\begin{definition}
Let $F$ be a potted bilayer problem.
Then, $j_F\colon\Omega\to\Omega$ is defined as follows:
\begin{itemize}
\item {\em Potential.} $j_F(\paset{P})_+=F^\solve_+(P_+)$.
\item {\em Actual.} $j_F(\paset{P})_-=F^\solve(P_-)$.
\end{itemize}
\end{definition}

\begin{theorem}
Let $F\colon\paset{X}\times\Lambda\tto Y$ be a potted bilayer problem.
Then, $j=j_F$ is a LT-modality.
\end{theorem}

\begin{proof}
\underline{\em Monotone:}
Let $e$ be a realizer for $\paset{A}\totarr \paset{B}$.
Consider the following algorithm $u$:
\begin{itemize}
\item Given a code of a computation $(\nu,\psi)$, transform the output by $e$ to obtain a code for the computation $(\nu,\varphi_e\circ\psi)$.
\end{itemize}

{\em Verification:}
First, let us focus on the following observation.
\begin{itemize}
\item $(\nu,\psi)$ and $(\nu,\varphi_e\circ\psi)$ have the same domains, so they shares a \ppp computation tree $\tpot$.
\item They also shares a query-labeling function $\nu$; so their $(F;\nu)$-branching subtrees are also equivalent.
\end{itemize}

{\em Potential.}
\begin{itemize}
\item Each \ppp realizer $e\in(\paset{A}\totarr \paset{B})_+$ transforms an element of $A_+$ into an element of $B_+$.
\item The above algorithm $ue$ recieves a code $a\in j(\paset{A})_+$ of a \ppp computation tree of type $A_+$, and returns a code of a \ppp computation tree of type $B_+$:
\begin{itemize}
\item As $\Phi^+_a\down\subseteq A_+$, we have $\Phi^+_{uea}=\varphi_e\circ\Phi^+_{a}\down\subseteq B_+$.
\end{itemize}
Thus, $uea\in j(\paset{B})_+$.
\item Therefore, we get $ue\in j(\paset{A})_+\arr j(\paset{B})_+=(j(\paset{A})\totarr j(\paset{B}))_+$.
\end{itemize}

{\em Actual.}
\begin{itemize}
\item Each \aaa realizer $e\in(\paset{A}\totarr \paset{B})_-$ transforms, in addition, an element of $A_-$ into an element of $B_-$.
\item Therefore, the above algorithm receives a code $a\in j(\paset{A})_-$ of a computation solving $A_-$, and returns a code of a computation solving $B_-$:
\begin{itemize}
\item There is an $(F;\nu)$-branching subtree $\tactual\subseteq\tpot$ such that $\Phi^F_{a|\tactual}\down\subseteq A_-$; hence, we get $\Phi^F_{uea|\tactual}=\varphi_e\circ\Phi_a^F\down\subseteq B_-$.
\end{itemize}
Thus, $uea\in j(\paset{B})_-$.
\item Therefore, we get $ue\in(j(\paset{A})\arr j(\paset{B}))_-$.
\end{itemize}

\medskip
\noindent
\underline{\em Inflationary:}
Let $a$ be a realizer for $\paset{A}$.
Consider the following algorithm $\eta$:
\begin{itemize}
\item Given $a$, returns a code of a computation which immediately outputs $a$ without making a query.
\end{itemize}

{\em Verification:}
\begin{itemize}
\item {\em Potential.} 
For each $a\in A_+$, $\eta a$ yields a \ppp computation of type $A_+$; hence, $\eta a\in j(\paset{A})_+$.
\begin{itemize}
\item As $a\in A_+$, we have $\Phi_{\eta a}^+\down=\{a\}\subseteq A_+$.
\end{itemize}
\item {\em Actual.}
For each $a\in A_-$, $\eta a$ $F$-solves $A_-$; hence, $\eta a\in j(\paset{A})_-$.
\begin{itemize}
\item As $a\in A_-$, we have $\Phi_{\eta a}^F\down=\{a\}\subseteq A_-$.
\end{itemize}
\end{itemize}

\medskip
\noindent
\underline{\em Idempotent:}
Let $a$ be a realizer for $jj(\paset{A})$.
\begin{itemize}
\item {\em Potential.}
Each \ppp realizer $a\in jj(\paset{A})_+$ is a code of a \ppp computation $(\nu,\psi)$ of type $j(\paset{A})_+$, whose output is a code of a \ppp computation of type $A_+$.
\end{itemize}

Let $\mu$ be an algorithm that composes these computations.
Formally, the algorithm $\mu$ is defined as follows:
\begin{itemize}
\item Each $a\in jj(\paset{A})_+$ induces a \ppp computation tree $(\tpot,\nu,\psi)$ of type $j(\paset{A})_+$.
To each leaf $\rho\in\leaf\tpot$, an output $\psi(\rho)\in j(A_+)$ is assigned.
This output also codes a \ppp computation tree $(\tpot_\rho,\nu_\rho,\psi_\rho)$ of type $A_+$.
By composing them, we construct the following computation tree $(\tpot^\mu,\nu^\mu,\psi^\mu)$:
\begin{align*}
&\tpot^\mu=\tpot\cup\{\rho\fr\sigma:\rho\in \leaf \tpot\mbox{ and }\sigma\in \tpot_\rho\}\\
&\nu^\mu(\sigma)=\nu(\sigma)\mbox{ for $\sigma\in \internal \tpot$}\\
&\nu^\mu(\rho\fr \sigma)=\nu_\rho(\sigma)\mbox{ for $\rho\in\leaf \tpot$ and $\sigma\in \internal \tpot_\rho$}\\
&\psi^\mu(\rho\fr \rho')=\nu_\rho(\rho')\mbox{ for $\rho\in\leaf\tpot$ and $\rho'\in \leaf \tpot_\rho$}
\end{align*}
\item Let $\mu a$ be a code of the tree $(\tpot^\mu,\nu^\mu,\psi^\mu)$.
\end{itemize}

{\em Verification:}
{\em Potential.}
\begin{itemize}
\item Since $\tpot^\mu$ is obtained by attaching a full-branching well-founded tree $\tpot_\rho$ to each leaf $\rho$ of the full-branching well-founded tree $T$, the resulting tree $\tpot^\mu$ is also a full-branching well-founded tree.

\item Since $\nu,\nu_\rho$ are $(X_+\times\Lambda)$-valued, the function $\nu^\mu$ is also $(X_+\times\Lambda)$-valued.
Moreover, $\nu^\mu$ is defined on $\internal \tpot^\mu$, the internal nodes.
\item Since $\psi_\rho$ is $A_+$-valued, the function $\psi^\mu$ is also $A_+$-valued.
Moreover, $\psi^\mu$ is defined on $\leaf \tpot^\mu$, the leaves.
\item Consequently, $(T^\mu,\nu^\mu,\psi^\mu)$ is a \ppp computation tree of type $A_+$; hence, $\mu a\in j(\paset{A})_+$.
\end{itemize}

{\em Actual.}
\begin{itemize}
\item Each \aaa realizer $a\in jj(\paset{A})_-$ is a code of a computation $(\nu,\psi)$ that $F$-solves $j(\paset{A})_-$, whose output is a code that $F$-solves $A_-$.
\begin{itemize}
\item There is an $(F;\nu)$-branching subtree $\tactual\subseteq\tpot$ such that the output $\psi(\rho)$ at each leaf $\rho\in\leaf\tactual$ codes a computation $(\tpot_\rho,\nu_\rho,\psi_\rho)$ that $F$-solves $A_-$.
\item That is, there is an $(F;\nu)$-branching subtree $\tactual_\rho\subseteq\tpot_\rho$ such that $\psi(\rho')\in A_-$ for any $\rho'\in\leaf\tactual_\rho$.
\end{itemize}
\item 
By joining these trees together, we obtain a new tree $\tactual^\mu\subseteq\tpot^\mu$.
\[\tactual^\mu=\tactual\cup\{\rho\fr\sigma:\rho\in \leaf \tactual\mbox{ and }\sigma\in \tactual_\rho\}.\]
Then, $\tactual^\mu\subseteq\tpot^\mu$ is $(F;\nu^\mu)$-branching.
This is because, for any $\sigma\in\internal\tactual^\mu$:
\begin{itemize}
\item If $\sigma\in\internal\tactual$ then $\sigma$ is an $\overline{F}_{\nu(\sigma)}$-branching node, and $\nu^\mu(\sigma)=\nu(\sigma)$.
\item If $\sigma=\rho\fr\tau$ for some $\rho\in\leaf\tactual$ and $\tau\in\internal\tactual_\rho$, then $\tau$ is an $\overline{F}_{\nu_\rho(\tau)}$-branching node in $\tactual_\rho$.
Since $\nu^\mu(\rho\fr\tau)=\nu_\rho(\tau)$, this $\sigma=\rho\fr\tau$ is an $\overline{F}_{\nu^\mu(\sigma)}$-branching node in $\tactual^\mu$.
\end{itemize}
\item $\psi^\mu$ is $A_-$-valued on $\leaf\tactual^\mu$:
\begin{itemize}
\item Each leaf $\sigma\in\leaf\tactual^\mu$ is of the form $\sigma=\rho\fr\rho'$ for some $\rho\in\leaf\tactual$ and $\rho'\in\leaf\tactual_{\rho}$.
Then $\psi^\mu(\rho\fr\rho')=\psi_\rho(\rho')\in A_-$.
\end{itemize}
Therefore, $(\tpot^\mu,\nu^\mu,\psi^\mu)$ $F$-solves $A_-$; that is, $\mu a\in j(\paset{A})_-$.
\end{itemize}

This shows that $\mu$ realizes $jj(\paset{A})\to j(\paset{A})$.
\end{proof}

\begin{definition}
For each LT-modality $j,k\colon\Omega\to\Omega$:
\[j\leq k\iff \bigcap_{P\in\Omega}\big(j(P)\arr k(P)\big)\mbox{ has an \aaa realizer.}\]
\end{definition}

\begin{theorem}
For any potted bilayer problems $F,G$:
\[F\tgw G\iff j_F\leq j_G.\]
\end{theorem}

\begin{proof}
($\Rightarrow$)
Think of $P\in\Omega$ as a potted problem $\dot{P}(\ast)=P_-$, where $\dom_+(\dot{P})=\om$, $\dom_-(\dot{P})=\{\ast\}$ and $\cod_+(\dot{P})=P_+$.
\begin{itemize}
\item Then, an \aaa realizer $a$ for $j_F(P)$ can be identified with an index of a reduction for $\dot{P}\tgw F$.
\item If $e$ is an index of a reduction for $F\tgw G$, by transitivity (Proposition \ref{prop:transitive}), we get an index $t(a,e)$ of a reduction for $\dot{P}\tgw G$.
\item That is, $a\mapsto t(a,e)$ gives an \aaa realizer for $\bigcap_P(j_F(P)\arr j_G(P))$.
\end{itemize}

($\Leftarrow$)
Let $a$ be an \aaa realizer for $\bigcap_P(j_F(P)\arr j_G(P))$.
\begin{itemize}
\item For each \ppp instance $n\in\dom_+(F)$, define the potted set $P_n$ as follows:
\begin{itemize}
\item If $n\in\dom(F)$ then put $P_n=(F(n);\cod_+(F))$; otherwise, $P_n=(\emptyset;\cod_+(F))$.
\end{itemize}
\item Let $q_n$ be an algorithm which makes the query $n$ to $F$, and outputs an response from $F$ as is.
\begin{itemize}
\item If $n\in\dom_+(F)$ then it potentially outputs an element in $\cod_+(F)$, and if $n\in\dom(F)$ then it actually outputs an element in $F(n)$.
\item In the latter case, this means that $\Phi_{q_n}^F$ solves $P_n$.
Therefore, in any case, $q_n$ is an \aaa realizer for $j_F(P_n)$.
\end{itemize}
\item Applying $a$ to this, we obtain $aq_n\in j_G(P_n)$.
\begin{itemize}
\item If $n\in\dom_+(F)$ then $\Phi_{aq_n}^+$ is a \ppp computation tree of type $\cod_+(F)$, and if $n\in\dom(F)$ then $\Phi_{aq_n}^G$ solves $F(n)$.
\item Consequently, $n\mapsto aq_n$ gives an index of a reduction for $F\tgw G$.
\end{itemize}
\end{itemize}
\end{proof}

Hereafter, $j_F$-realizability is referred to simply as {\em $F$-realizability}, and the corresponding \aaa realizer is referred to as an {\em $F$-\aaa realizer}.
Note that $j_F$ is always $\neg\neg$-dense by the nonemptyness condition $F(x)\not=\emptyset$.

\subsection{Sheaf subtopos}\label{sec:topos-sheafsubtopos}
The structure consists of the truth-order on {\em predicates} $(\Omega^X,\leq_X)$ and {\em substitutions} (commonly called a {\bf tripos} \cite{vOBook}) merely provides a logical structure.
In particular, the logical structure of total computability used here is called the {\bf Grayson tripos}.
A tripos, by itself, has no objects.
Therefore, we typically add the following kind of objects.

\begin{definition}[\cite{vO97}]
An {\bf $\Omega$-set} is a set $X$ equipped with an internal equivalence relation $\approx_X\colon X^2\to\Omega$; that is, reflexivity, symmetricity, transitivity are realizable.
\end{definition}

\begin{example}
Define $\approx_{\sf N}\colon\om^2\to\Omega$ as follows:
\[
(n\approx_{\sf N} m)=
\begin{cases}
(\{n\};\om)&\mbox{ if }n=m,\\
(\emptyset;\om)&\mbox{ if }n\not=m.
\end{cases}
\]

Then ${\sf N}:=(\om,\approx_{\sf N})$ forms an $\Omega$-set, called the {\bf natural numbers object}.
\end{example}

We make a brief comment here:
\begin{itemize}
\item For an $\Omega$-set $X$, $(x\approx_Xx)$ is treated as the set of all {\em $X$-codes} of $x$.
For example, the unique ${\sf N}$-code of a natural number $n$ is $n$ itself.
\item In $\N$, if the external equality $=$ fails, then the internal equality $\approx_{\sf N}$ cannot be realizable.
Such an $\Omega$-set is usually called an {\em assembly}.
Therefore, an assembly has only information on codes.
\end{itemize}

\begin{fact}
The category of $\Omega$-sets and functional relations is a topos.
\end{fact}
This is known as the {\em tripos-to-topos construction}.
In our setting, the resulting topos of total computability is known as the {\bf Grayson topos} \cite{vO97}.
By adding a bit more structure to satisfy extensionality, one obtains the modified realizability tripos and the associated modified realizability topos.
However, within the scope of our discussion, there is no difference from the Grayson topos.
Therefore, in this article, we often refer to the Grayson topos as the modified realizability topos.

The interpretation of quantification in the category of $\Omega$-sets is given by:
\begin{itemize}
\item $\eval{\exists x^X\varphi(x)}=\bigcup_{x\in X}[(x\approx_Xx)\times\eval{\varphi(x)}]$.
\item $\eval{\forall x^X\varphi(x)}=\bigcap_{x\in X}[(x\approx_Xx)\arr\eval{\varphi(x)}]$.
\end{itemize}

If each quantification applies only to the natural numbers, and if $j$ is $\neg\neg$-dense (i.e., $j(p)_-$ is nonempty for any $p$), then the Kuroda $j$-translation can be concisely described as follows:
\begin{itemize}
\item $\eval{\exists x^\N\varphi(x)}=\bigcup_{x\in \om}[j(x\approx_{\sf N} x)\times\eval{\varphi(x)}^j]$.
\item $\eval{\forall x^\N\varphi(x)}=\bigcap_{x\in \om}[j(x\approx_{\sf N} x)\arr j(\eval{\varphi(x)}^j)]$.
\end{itemize}

Thus, $j$-realizability for arithmetical formulas can also be defined in the same way as before.
Note that here we distinguish between the symbols: $\N$ is the type of natural numbers, $\sfN$ is its interpretation in the category, and $\om$ is the set of all meta-natural numbers.

\begin{remark}
It is known that a closure transformation on a tripos corresponds to a Lawvere-Tierney topology on the corresponding topos \cite{vOBook}.
Furthermore, for each LT-topology $j$, there is a notion of a $j$-sheaf, and the set of all $j$-sheaves forms a subtopos of the original topos.
Then, $j$-realizability provides a semantics for the $j$-sheaf subtopos.
For computability-theoretic aspects of sheaves, see also \cite{AhBa26}.
\end{remark}

Although this article does not explain the notion of a sheaf, we offer a few comments below regarding what the natural numbers object is and what its exponential object is, as these are necessary for interpreting second-order arithmetic.

\begin{example}
${\sf N}_j:=(\om,j\circ{\approx_\N})$ is the natural numbers object in the $j$-sheaf subtopos.
If $j=j_F$ for some potted bilayer problem $F$, this means:
\[
(n\approx_{{\sf N}_j} m)=
\begin{cases}
(F^\solve\{n\};F^\solve_+\om)&\mbox{ if }n=m,\\
(\emptyset;F^\solve_+\om)&\mbox{ if }n\not=m.
\end{cases}
\]

Using this, we can verify that $j$-realizability for the arithmetical formulas mentioned above provides the arithmetical semantics for the $j$-sheaf subtopos.
\end{example}

Since the assemblies form an exponential ideal, the exponential object ${\sf N}_j^{{\sf N}_j}$ in the $j$-sheaf subtopos is also an assembly.

\begin{example}\label{exa:pre-exponential}
If $j=j_F$, the underlying set of ${\sf N}_j^{{\sf N}_j}$ consists of all functions $f\colon\om\to\om$ which are total $F$-computable as $\sfN_j\to\sfN_j$; that is, there is an index $e$ such that:
\begin{itemize}
\item {\em Potential.}
Given a \ppp $\sfN_j$-code $a$, i.e.~$a\in F^\solve_+\om$, $\Phi_e^+(a)$ outputs a \ppp $\sfN_j$-code.
\item {\em Actual.}
Given an \aaa $\sfN_j$-code $a$ of $n$, i.e.~$a\in F^\solve\{n\}$, $\Phi_e^F(a)$ outputs an \aaa $\sfN_j$-code of $f(n)$.
\end{itemize}
Such an $e$ is an \aaa $\sfN_j^{\sfN_j}$-code of $f$.
The set of \ppp codes is $F^\solve_+\om\arr F^\solve_+\om$.
\end{example}
However, using the technique below, this can be simplified.

\subsection{Translation techniques}

In the logic of potted sets, the set operations introduced in Definition \ref{def:set-operations-on-pasets} are used in the interpretation of predicate logic.
It should be noted that the LT-modality obtained from an oracle also satisfies monotonicity with respect to set operations.

\begin{obs}\label{obs:pot-truth-inclusion}
Let $j=j_F$, and $\paset{A},\paset{B},\paset{B}_i$ be potted sets.
\begin{enumerate}
\item $\paset{A}\subseteq \paset{B}$ implies $j(\paset{A})\subseteq j(\paset{B})$.
\item $\bigcup_{i\in I}j(\paset{B}_i)\subseteq j(\bigcup_{i\in I}\paset{B}_i)$.
\end{enumerate}
\end{obs}

\begin{proof}
(1)
\begin{itemize}
\item {\em Potential.}
Each \ppp realizer $a\in j(\paset{A})_+$ is a code of a \ppp computation tree of type $A_+$.
As $A_+\subseteq B_+$, in particular, $a$ is a code of a \ppp computation tree of type $B_+$; that is, $a\in j(\paset{B})_+$.
\item {\em Actual.}
In addition, if $a\in j(\paset{A})_-$ then $a$ $F$-solves $A_-$.
As $A_-\subseteq B_-$, in particular, $a$ $F$-solves $B_-$; therefore, $a\in j(\paset{B})_-$.
\end{itemize}

(2)
For each $k\in I$, since $\paset{B}_k\subseteq\bigcup_{i\in I}\paset{B}_i$ by the item (1), we get $j(\paset{B}_k)\subseteq j(\bigcup_{i\in I}\paset{B}_i)$.
Therefore, we get $\bigcup_{k\in I}j(\paset{B}_k)\subseteq j(\bigcup_{i\in I}\paset{B}_i)$.
\end{proof}

This leads to the following useful lemma:
\begin{lemma}\label{lem:AE-thm-simplify}
For potted sets $\paset{A},\paset{B}_i$, realizers for the following two formulas are mutually convertible.
Furthermore, such a conversion does not depend on $\paset{A},\paset{B}_i$.
\begin{enumerate}
\item $\paset{A}\arr j(\bigcup_{i\in I}\paset{B}_i)$.
\item $j(\paset{A})\arr j(\bigcup_{i\in I}j(\paset{B}_i))$.
\end{enumerate}
\end{lemma}

\begin{proof}
(1)$\arr$(2):
\begin{itemize}
\item Assume that the formula (1) is realizable by $e\in \paset{A}\arr j(\bigcup_{i\in I}\paset{B}_i)$.
\item By monotonicity, we have $ue\in j(\paset{A})\arr jj(\bigcup_{i\in I}\paset{B}_i)$.
\item By inflation, for any $i\in I$, we have $\eta\in \paset{B}_i\arr j(\paset{B}_i)$; thus, $\eta\in\bigcup_{i\in I}\paset{B}_i\arr\bigcup_{i\in I}j(\paset{B}_i)$.
\item By monotonicity again, we have $u\eta\in j(\bigcup_{i\in I}\paset{B}_i)\arr j(\bigcup_{i\in I}j(\paset{B}_i))$.
\item Combining the above realizers with idempotence:
\[
j(\paset{A})\xarr{ue}jj\left(\bigcup_{i\in I}\paset{B}_i\right)\xarr{\mu}j\left(\bigcup_{i\in I}\paset{B}_i\right)\xarr{u\eta} j\left(\bigcup_{i\in I}j(\paset{B}_i)\right).
\]
This realizes the formula (2).
\end{itemize}

(2)$\arr$(1):
\begin{itemize}
\item Assume that the formula (2) is realizable by $e\in j(\paset{A})\arr j(\bigcup_{i\in I}j(\paset{B}_i))$.
\item By Observation \ref{obs:pot-truth-inclusion} (2), we have $\bigcup_{i\in I}j(\paset{B}_i)\subseteq j(\bigcup_{i\in i}\paset{B}_i)$.
Let $i$ be a code of this inclusion map.
\item By monotonicity, we get $ui\in j(\bigcup_{i\in I}j(\paset{B}_i))\arr jj(\bigcup_{i\in I}\paset{B}_i)$.
\item Combining the above realizers with inflation and idempotence:
\[
\paset{A}\xarr{\eta}j(\paset{A})\xarr{e}j\left(\bigcup_{i\in I}j(\paset{B}_i)\right)\xarr{ui} jj\left(\bigcup_{i\in I}\paset{B}_i\right)\xarr{\mu}j\left(\bigcup_{i\in I}\paset{B}_i\right).
\]
This realizes the formula (1).
\end{itemize}
\end{proof}

\begin{example}\label{exa:AE-thm-realizer-simplify}
The Kuroda $j$-translation of $\forall x^\N\exists y^\N A(x,y)$ according to the definition is:
\[
\bigcap_{x\in\om}\left(j(x\approx_\sfN x)\arr j\left(\bigcup_{y\in\N}\big(j(y\approx_\sfN y)\times A^j(x,y)\big)\right)\right)
\]

{\em Discussion:}
\begin{itemize}
\item By the $\land$-preservation, there is an effective conversion between $j(y\approx_\sfN y)\times A^j(x,y)$ and $j((y\approx_\sfN y)\times A^j(x,y))$.
\item by applying Lemma  \ref{lem:AE-thm-simplify}, the above formula is rewritten as follows:\[
\bigcap_{x\in\om}\left((x\approx_\sfN x)\arr j\left(\bigcup_{y\in\N}\big((y\approx_\sfN y)\times A^j(x,y)\big)\right)\right)
\]
\end{itemize}

Since this rewriting will be used frequently later on, we provide a detailed description for $j=j_F$ here.
Since $(x\approx_\sfN x)=(\om;\{x\})$, one can think of a $j$-realizer for $\forall x^\N\exists y^\N A(x,y)$ as:
\begin{itemize}
\item {\em Potential.}
\[
\bigcap_{x\in\om}\left(\om\arr F_+^\solve\left(\bigcup_{y\in\om}\big(\om\times A_+^j(x,y)\big)\right)\right)
\]
That is, each \ppp realizer is a code of a total computable multifunction $\Phi$ which, given an input $x \in \om$, outputs the pair of a value $y\in\om$ and a \ppp $j$-realizer of $A(x,z)$ for some $z$.
\item {\em Actual.}
It is the intersection of the above \ppp part with the following:
\[
\bigcap_{x\in\om}\left(\{x\}\arr F^\solve\left(\bigcup_{y\in\om}\big(\{y\}\times A_-^j(x,y)\big)\right)\right)
\]
That is, given $x\in\om$, the total multifunction $\Phi^F$ outputs the pair of a value $y\in\om$ and an \aaa $j$-realizer for $A(x,y)$.
\end{itemize}
\end{example}

\begin{example}\label{exa:func-on-nat}
If $j=j_F$, an $F$-computable function ${\sf N}_j\to{\sf N}_j$ can be identified with a computable function ${\sf N}_j\to{\sf N}_{jj}$, where ${\sf N}_{jj}=(\om,j\circ{\approx_{\sfN_j}})=(\om,j\circ j\circ{\approx_{\sfN}})$, by idempotence of $j$, it is merely a computable function of type ${\sf N}_j\to{\sf N}_j$.
The above argument shows that a computable function of type ${\sf N}_j\to{\sf N}_j$ can be identified with that of type ${\sf N}\to{\sf N}_j$.
Again, this is just an $F$-computable function of type ${\sf N}\to{\sf N}$.

In summary, the underlying set of the exponential object $\sfN_j^{\sfN_j}$ in Example \ref{exa:pre-exponential} can be replaced with the set of all total $F$-computable functions $f\colon\om\to\om$; that is, there is an index $e$ such that:
\begin{itemize}
\item {\em Potential.}
Given $n\in\om$, $\Phi_e^+(n)$ halts.
\item {\em Actual.}
Given $n\in\om$, $\Phi_e^F(n)$ outputs $f(n)$.
\end{itemize}
Such an $e$ is an \aaa $\sfN_j^{\sfN_j}$-code of $f$ (which will also be called an $F$-code).
The set of \ppp codes is $\om\arr \om$.
\end{example}

\section{Realizability}\label{sec:main-realizability}

\subsection{\bflogic{MP} and its variants}

We now address the main topic of this article: the analysis of Markov's principle and its weak variants.
Markov's principle, also known as the principle of double negation elimination for $\Sigma_1$ formulas, is formulated as follows.

\begin{definition}
Markov's principle
$\bflogic{MP}$ is the following principle:
\[\forall \alpha\in 2^\N.\ (\neg\forall n\in\N.\ \alpha(n)=0\ \longrightarrow\ \exists n\in\N\ \alpha(n)\not=0).\]
\end{definition}

\begin{definition}
$n$-disjunctive Markov's principle
$\bflogic{MP}^\lor_n$ is the following principle:
\begin{align*}
U(\alpha)&\equiv
\forall u,v\in n\times\N.\ (\alpha(u)\not=0\land \alpha(v)\not=0)\to u=v,
\\
E(\alpha)&\equiv
\neg\forall u\in n\times\N.\ \alpha(u)=0
\\
\bflogic{MP}^\lor_n&\equiv\forall \alpha\in 2^{n\times\N}.\ 
[U(\alpha)\land E(\alpha)
\longrightarrow
\exists k<n\forall s\in\N\ \alpha(k,s)=0].
\end{align*}
\end{definition}

\begin{notation}
For each $k,s$, put $\alpha_k(s)=\alpha(k,s)$.
If $\alpha(n)=0$ for any $n\in\N$, then we write $\alpha=0^\infty$.
\end{notation}

\begin{obs}\label{obs:V-true-realizable-trivial}
$U(\alpha)$ is true if and only if it is $F$-realizable, and such a realizer does not depend on $\alpha$.
\end{obs}

\begin{proof}~
\begin{itemize}
\item Assume that $U(\alpha)$ has \aaa realizer $a$.
If both $\alpha(u)\not=0$ and $\alpha(v)\not=0$ are true, then applying $a$ to their realizers yields a realizer for $u=v$.
As the underlying object is an assembly, this means that $u=v$ is true.
\item Conversely, assume that $U(\alpha)$ is true.
If $u=v$ is true, then a code of $u$ is a realizer; so if $U(\alpha)$ is true then the first projection which, given codes of $u,v$, returns a code $u$ gives an \aaa realizer.
\end{itemize}
\end{proof}

Recall that $\sii{\sf C}_{1/\om}$ is linked to Markov's principle ${\sf MP}$.

\begin{theorem}\label{thm:MP-in-MP-sheaf}
Assume $\sii{\sf C}_{1/\om}\tgw F\tgw\sii{\sf C}_{1/\om}+\pii{\sf C}_{\geq 1/2}$.
Then $\bflogic{MP}$ is $F$-realizable.
\end{theorem}

\begin{proof}
{\em Realizer for $\bflogic{MP}$:}
Based on the argument in Example \ref{exa:AE-thm-realizer-simplify}, an $F$-realizer $r$ for $\bflogic{MP}$ is described as follows.
\begin{itemize}
\item {\em Instance.}
An $F$-code $a$ of a sequence $\alpha\in 2^\N$ is, by Example \ref{exa:func-on-nat}, a code of $F$-computation which, given $n\in\om$, returns $\alpha(n)$; that is, $\Phi^F_a(n)=\alpha(n)$.
\item {\em Potential.} $0$ is a \ppp realizer for any proposition; in particular, for $\neg\forall n\alpha(n)=0$.
Hence, $\Phi^+_r(a,0)$ halts.
\item {\em Actual.} Note that a $\Pi_1$ formula is true iff it is realizable.
Hence, if $\neg\forall n\alpha(n)=0$ is true, then it is realizable, and $0$ is its \aaa realizer since $\neg$ is the outmost connective.
In this case, $\Phi_r^F(a,0)$ outputs an existential witness for the conclusion of $\bflogic{MP}$; that is, for any $s\in \Phi_r^F(a,0)$, we have $\alpha(s)\not=0$.
\end{itemize}

\medskip
To show that $\bflogic{MP}$ is $F$-realizable, let an $F$-code $e$ for $\alpha\in 2^\N$ be given.
\begin{itemize}
\item That is, $\alpha\tgw F$ via some $e$.
By the assumption, we have $\alpha=\Phi_e^F\tgw F\tgw G:=\sii{\sf C}_{1/\om}+\pii{\sf C}_{\geq 1/2}$, so by uniform transitivity (see the last paragraph of the proof of Proposition \ref{prop:transitive}) there is a total computable function $t$ such that $\alpha\tgw G$ via $t(e)$.
That is, we have $\Phi_{t(e)}^{G}(n)=\Phi_e^{F}(n)=\alpha(n)$ for any $n\in\N$.
\item Using a computable function in Theorem \ref{LLPO-equial-partial-computable}, we get $\varphi_{\ep(t(e))}=\alpha$.
\end{itemize}

By Proposition \ref{prop:MP-as-Sigma-1-choice}, our assumption implies ${\sf MP}\tgw F$.
Thus, there is an $F$-computable process simulating ${\sf MP}(\ep(t(e)))$.
Formally, perform the following process:
\begin{itemize}
\item Assume ${\sf MP}\tgw F$ via $u$.
\begin{itemize}
\item {\em Potential.}
For each \ppp instance $a\in\dom_+({\sf MP})$, since $u$ gives a total computation tree of type $\dom_+({\sf MP})\to \cod_+({\sf MP})$, the \ppp computation $\Phi_u^+(a)$ halts.
\item {\em Actual.} For each \aaa instance $a\in\dom({\sf MP})$, we have $\Phi_u^F(a)\subseteq{\sf MP}(a)$.
\end{itemize}
\item Given a code $e$, one can compute an index $r$ such that $\Phi_r^+(e,0)=\Phi_u^+(\ep(t(e)))$.
\end{itemize}

{\em Verification:}
\begin{itemize}
\item {\em Potential.}
If $e$ is a code of a sequence $\alpha\in 2^\N$, then we have $\varphi_{\ep(t(e))}=\alpha$ as seen above, $\ep(t(e))$ is a code of a total computable function; that is, $\ep(t(e))\in\dom_+({\sf MP})$.
Hence, the computation $\Phi_r^+(e,0)=\Phi_u^+(\ep(t(e)))$ halts as seen above.
\item {\em Actual.}
If $\neg\exists n\alpha(n)=0$ is true, the code $\ep(t(e))$ of $\varphi_{\ep(t(e))}=\alpha$ belongs to $\dom({\sf MP})$.
Hence, for any $s\in\Phi_r^F(e,0)=\Phi_u^F(\ep(t(e)))\subseteq {\sf MP}(\ep(t(e)))$ we have $\alpha(s)\not=0$.
\end{itemize}

By the above characterization of a realizer for $\bflogic{MP}$, this means that $r$ is an $F$-realizer for $\bflogic{MP}$.
\end{proof}

\begin{theorem}\label{thm:MPn-reducible-to-realizable}
Assume $\sii{\sf C}_{n-1/n}\tgw F\tgw\sii{\sf C}_{1/\om}+\pii{\sf C}_{\geq 1/2}$.
Then $\bflogic{MP}_n^\lor$ is $F$-realizable.
\end{theorem}

\begin{proof}
The argument is similar to the proof of Theorem \ref{thm:MP-in-MP-sheaf}.

{\em Realizer for $\bflogic{MP}_n^\lor$:}
Based on the argument in Example \ref{exa:AE-thm-realizer-simplify}, an $F$-realizer $r$ for $\bflogic{MP}_n^\lor$ is described as follows.
\begin{itemize}
\item {\em Instance.}
An $F$-code $a$ of a sequence $\alpha\in 2^{n\times\N}$ is, by Example \ref{exa:func-on-nat}, a code of $F$-computation which, given $(k,s)$, returns $\alpha(k,s)$; that is, $\Phi^F_a(k,s)=\alpha(k,s)$.
\item {\em Actual.} By Observation \ref{obs:V-true-realizable-trivial}, if $U(\alpha)$ is true, then a code $\pi$ of the first projection is an \aaa realizer.
If $E(\alpha)$ is true, it is realizable and $0$ is its \aaa realizer as before.
Hence, if the premise $U(\alpha)\land E(\alpha)$ is true, $\pair{\pi,0}$ is its \aaa realizer.
In this case, $\Phi_r^F(a,\pair{\pi,0})$ outputs an existential witness for the conclusion of $\bflogic{MP}_n^\lor$; that is, for any $k\in \Phi_r^F(a,\pair{\pi,0})$, we have $\alpha_k=0^\infty$.
\item {\em Potential.} One can see $\pi\in U(\alpha)_+$.
Thus, $\pair{\pi,0}$ is a \ppp realizer for $U(\alpha)\land E(\alpha)$.
Hence, $\Phi^+_r(a,\pair{\pi,0})$ halts whenever $\alpha$ is total.
\end{itemize}

\medskip
To show that $\bflogic{MP}_n^\lor$ is $F$-realizable, let an $F$-code $e$ for $\alpha\in 2^{n\times\N}$ be given.
\begin{itemize}
\item That is, $\alpha\tgw F$ via some $e$.
By the assumption, we have $\alpha=\Phi_e^F\tgw F\tgw G:=\sii{\sf C}_{1/\om}+\pii{\sf C}_{\geq 1/2}$, so by uniform transitivity (Proposition \ref{prop:transitive}) there is a total computable function $t$ such that $\alpha\tgw G$ via $t(e)$.
That is, we have $\Phi_{t(e)}^{G}(k,s)=\Phi_e^{F}(k,s)=\alpha(k,s)$ for any $(k,s)$.
\item Using a computable function in Theorem \ref{LLPO-equial-partial-computable}, we get $\varphi_{\ep(t(e))}=\alpha$.
\end{itemize}

By Theorem \ref{thm:disjunctive-Markov-counique}, our assumption implies ${\sf MP}_n^\lor\tgw F$.
Thus, there is an $F$-computable process simulationg ${\sf MP}_n^\lor(\ep(t(e)))$.
Formally, perform the following process:
\begin{itemize}
\item Assume ${\sf MP}_n^\lor\tgw F$ via $u$.
\begin{itemize}
\item {\em Potential.} For each \ppp instance $a\in\dom_+({\sf MP}_n^\lor)$, since $u$ gives a total computation tree of type $\dom_+({\sf MP}_n^\lor)\to\cod_+({\sf MP}_n^\lor)$, the \ppp computation $\Phi_u^+(a)$ halts.
\item {\em Actual.} For each \aaa instance $a\in\dom({\sf MP}_n^\lor)$, we have $\Phi_u^F(a)\subseteq{\sf MP}_n^\lor(a)$.
\end{itemize}
\item Given a code $e$, one can compute an index $r$ such that $\Phi_r^+(e,\pair{\pi,0})=\Phi_u^+(\ep(t(e)))$.
\end{itemize}

{\em Verification:}
\begin{itemize}
\item {\em Potential.}
If $e$ is a code of a sequence $\alpha\in 2^{n\times\N}$, then we have $\varphi_{\ep(t(e))}=\alpha$ as seen above, $\ep(t(e))$ is a code of a total computable function; that is, $\ep(t(e))\in\dom_+({\sf MP}_n^\lor)$.
Hence, the computation $\Phi_r^+(e,\pair{\pi,0})=\Phi_u^+(\ep(t(e)))$ halts as seen above.
\item {\em Actual.}
If $U(\alpha)\land E(\alpha)$ is true, the code $\ep(t(e))$ of $\varphi_{\ep(t(e))}=\alpha$ belongs to $\dom({\sf MP}_n^\lor)$.
Hence, for any $k\in\Phi_r^F(e,\pair{\pi,0})=\Phi_u^F(\ep(t(e)))\subseteq {\sf MP}_n^\lor(\ep(t(e)))$ we have $\alpha_k=0^\infty$.
\end{itemize}

By the above characterization of a realizer for $\bflogic{MP}_n^\lor$, this means that $r$ is an $F$-realizer for $\bflogic{MP}_n^\lor$.
\end{proof}

\begin{remark}
The upper bound $\sii{\sf C}_{1/\om}+\pii{\sf C}_{\geq 1/2}$ in Theorems \ref{thm:MP-in-MP-sheaf} and \ref{thm:MPn-reducible-to-realizable} is used solely to obtain a computable code for a given $\alpha\in 2^\N$.
Thus, by Theorem \ref{MP-WLEM-single-pc}, the upper bound can be replaced with $\sii{\sf C}_{1/\om}+{\sf C}_{2/3}$.
\end{remark}

\begin{prop}\label{prop:MP-in-MP-sheaf2}
Assume $F\tgw\sii{\sf C}_{1/\om}+{\sf C}_{2/3}$.
\begin{enumerate}
\item If $\sii{\sf C}_{1/\om}\tgw F$, then $\bflogic{MP}$ is $F$-realizable.
\item If $\sii{\sf C}_{n-1/n}\tgw F$, then $\bflogic{MP}^\lor$ is $F$-realizable.
\end{enumerate}
\end{prop}

\begin{remark}
By the same reason, this upper bound is unnecessary for the realizability of the variant $\mathbf{MP}_{\rm TM}$, in which the premise of Markov's principle is modified to ``{for any $\alpha \in 2^\mathbb{N}$, {\em if $\alpha$ is computable}, then...}''
It has been shown that this variant $\mathbf{MP}_{\rm TM}$ is equivalent to the primitive recursive Markov's principle $\mathbf{MP}_{\rm PR}$ \cite{CFKRR}:
\[
\bflogic{MP}_{\rm PR}\colon
\qquad
\forall e\in\N.~(\neg\forall s\in\N.\ \varphi_e(0)[s]{\uparrow}\ \longrightarrow\ \exists s\in\N.\ \varphi_e(0)[s]\down)
\]

Clearly, $\bflogic{MP}$ implies $\bflogic{MP}_{\rm PR}$.
\end{remark}

\begin{obs}\label{obs:MP-in-MP-sheaf3}
If $\sii{\sf C}_{1/\om}\tgw F$, then $\bflogic{MP}_{\rm PR}$ is $F$-realizable.
\end{obs}

Next, let us show the converse.

\begin{theorem}\label{thm:MP-realizable-to-reducible}
If $\bflogic{MP}$ is $F$-realizable, then $\sii{\sf C}_{1/\om}\tgw F$.
\end{theorem}

\begin{proof}
Let $r$ be an $F$-realizer for $\bflogic{MP}$.
\begin{itemize}
\item Recalling the specific description of a realizer for $\bflogic{MP}$ in the proof of Theorem \ref{thm:MP-in-MP-sheaf}, for any $F$-relative code $e$ of a total computable function, $\Phi_r^+(e,0)$ halts.
\item Given a computable function $\varphi_e$, its index $\eta(e)$ of an $F$-relative computation can be obtained by considering a computation that does not make any queries to $F$.
\item Then compute an index $m$ such that $\Phi_m^+(e)=\Phi_r^+(\eta(e),0)$.
\end{itemize}

{\em Verification:}
\begin{itemize}
\item {\em Potential.}
Assume that a code $e\in{\rm dom}_+({\sf MP})$ of a total computable function is given.
\begin{itemize}
\item Then $\eta(e)$ is an $F$-relative index of $\alpha=\varphi_e$; that is, $\Phi_{\eta(e)}^F(n)=\alpha(n)$.
\item Since $\eta(e)$ is an $F$-code of a total computable function, the computation $\Phi_m^+(e)=\Phi_r^+(\eta(e),0)$ must halt.
\end{itemize}
\item {\em Actual.}
For each \aaa instance $e\in\dom({\sf MP})$, for $\alpha=\varphi_e$, we have $\neg\forall n\alpha(n)=0$.
Since $\eta(e)$ is an $F$-code of $\alpha$, and $r$ is a realizer for $\bflogic{MP}$, for any $n\in\Phi_m^F(e)=\Phi_r^F(\eta(e),0)$, we must have $\alpha(n)=0$.
Therefore, we get $\Phi_m^F(e)\subseteq{\sf MP}(e)$.
\end{itemize}

Consequently, $m$ is a code of a reduction for ${\sf MP}\tgw F$.
By Proposition \ref{prop:MP-as-Sigma-1-choice}, we get the conclusion.
\end{proof}

\begin{remark}
Regarding Theorem \ref{thm:MP-realizable-to-reducible}, it is obvious that $\bflogic{MP}$ can be replaced by $\bflogic{MP}_{\rm PR}$.
\end{remark}

\begin{theorem}\label{thm:MPnlor-realizable-to-above}
If $\bflogic{MP}_n^\lor$ is $F$-realizable, then $\sii{\sf C}_{n-1/n}\tgw F$.
\end{theorem}

\begin{proof}
The argument is similar to the proof of Theorem \ref{thm:MP-realizable-to-reducible}.
Let $r$ be an $F$-realizer for $\bflogic{MP}_n^\lor$.
\begin{itemize}
\item Recalling the specific description of a realizer for $\bflogic{MP}_n^\lor$ in the proof of Theorem \ref{thm:MPn-reducible-to-realizable}, for any $F$-relative code $e$ of a total computable function, $\Phi_r^+(e,\pair{\pi,0})$ halts.
\item Given a computable function $\varphi_e$, its index $\eta(e)$ of an $F$-relative computation can be obtained by considering a computation that does not make any queries to $F$.
\item Then compute an index $m$ such that $\Phi_m^+(e)=\Phi_r^+(\eta(e),\pair{\pi,0})$.
\end{itemize}

{\em Verification:}
\begin{itemize}
\item {\em Potential.}
Assume that a code $e\in{\rm dom}_+({\sf MP}_n^\lor)$ of a total computable function $\alpha\in 2^{n\times\N}$ is given.
\begin{itemize}
\item Then $\eta(e)$ is an $F$-relative index of $\alpha=\varphi_e$; that is, $\Phi_{\eta(e)}^F(k,s)=\alpha(k,s)$.
\item Since $\eta(e)$ is an $F$-code of a total computable function, the computation $\Phi_m^+(e)=\Phi_r^+(\eta(e),\pair{\pi,0})$ must halt.
\end{itemize}
\item {\em Actual.}
For each \aaa instance $e\in\dom({\sf MP}_n^\lor)$, for $\alpha=\varphi_e$, the premise $U(\alpha)\land E(\alpha)$ is true.
As seen in Theorem \ref{thm:MPn-reducible-to-realizable}, $\pair{\pi,0}$ is its \aaa realizer.
Since $\eta(e)$ is an $F$-code of $\alpha$, and $r$ is a realizer for $\bflogic{MP}_n^\lor$, for any $k\in\Phi_m^F(e)=\Phi_r^F(\eta(e),0)$, we must have $\alpha_k=0^\infty$.
Therefore, we get $\Phi_m^F(e)\subseteq{\sf MP}_n^\lor(e)$.
\end{itemize}

Consequently, $m$ is a code of a reduction for ${\sf MP}_n^\lor\tgw F$.
By Theorem \ref{thm:disjunctive-Markov-counique}, we get the conclusion.
\end{proof}

\subsection{\bflogic{LLPO} and its variants}
Richman \cite{Ri90} introduced $\bflogic{LLPO}_n$, a weaker variant of $\bflogic{LLPO}$, the lessor limited principle of omniscience.
Note that $\bflogic{LLPO}$ is equivalent to $\bflogic{LLPO}_2$.

\begin{definition}
$\bflogic{LLPO}_n$ is the following principle:
\[
\forall \alpha\in 2^{n\times\N}.\ 
[U(\alpha)
\longrightarrow
(\exists k<n\forall s^\N\ \alpha(k,s)=0)].
\]
\end{definition}


Recall from Proposition \ref{prop:LLPO-AcoUC} that $\pii{\sf C}_{\geq n-1/n}$ is linked to ${\sf LLPO}_n$.
The following is similar to the standard argument on Lifschitz realizability.

\begin{theorem}\label{thm:LLPO-reducible-to-realizable}
Assume $\pii{\sf C}_{\geq n-1/n}\tgw F\tgw\sii{\sf C}_{1/\om}+\pii{\sf C}_{\geq 1/2}$.
Then $\bflogic{LLPO}_n$ is $F$-realizable.
\end{theorem}

\begin{proof}
The argument is similar to the proof of Theorem \ref{thm:MPn-reducible-to-realizable}.

{Realizer for $\bflogic{LLPO}_n$:}
Based on the argument in Example \ref{exa:AE-thm-realizer-simplify}, an $F$-realizer $r$ for $\bflogic{LLPO}_n$ is described as follows.
\begin{itemize}
\item {\em Instance.}
An $F$-code $a$ of a sequence $\alpha\in 2^{n\times\N}$ is, by Example \ref{exa:func-on-nat}, a code of $F$-computation which given $(k,s)$, returns $\alpha(k,s)$; that is, $\Phi^F_a(k,s)=\alpha(k,s)$.
\item {\em Actual.} As seen in Theorem \ref{thm:MPn-reducible-to-realizable}, if $U(\alpha)$ is true, then it has an \aaa realizer $\pi$.
In this case, $\Phi_r^F(a,\pi)$ outputs an existential witness for the conclusion of $\bflogic{LLPO}_n$; that is, for any $k\in \Phi_r^F(a,0)$, we have $\alpha_k=0^\infty$.
\item {\em Potential.} One can see $\pi\in U(\alpha)_+$.
Hence, $\Phi^+_r(a,\pi)$ halts whenever $\alpha$ is total.
\end{itemize}

\medskip
To show that $\bflogic{LLPO}_n$ is $F$-realizable, let an $F$-code $e$ for $\alpha\in 2^{n\times\N}$ be given.
\begin{itemize}
\item That is, $\alpha\tgw F$ via some $e$.
By the assumption, we have $\alpha=\Phi_e^F\tgw F\tgw G:=\sii{\sf C}_{1/\om}+\pii{\sf C}_{\geq 1/2}$, so by uniform transitivity (Proposition \ref{prop:transitive}) there is a total computable function $t$ such that $\alpha\tgw G$ via $t(e)$.
That is, we have $\Phi_{t(e)}^{G}(k,s)=\Phi_e^{F}(k,s)=\alpha(k,s)$ for any $(k,s)$.
\item Using a computable function in Theorem \ref{LLPO-equial-partial-computable}, we get $\varphi_{\ep(t(e))}=\alpha$.
\end{itemize}

By Proposition \ref{prop:LLPO-AcoUC}, our assumption implies ${\sf LLPO}_n\tgw F$.
Thus, there is an $F$-computable process simulating ${\sf LLPO}_n(\ep(t(e)))$.
Formally, perform the following process:
\begin{itemize}
\item Assume ${\sf LLPO}_n\tgw F$ via $u$.
\begin{itemize}
\item {\em Potential.} For each \ppp instance $a\in\dom_+({\sf LLPO}_n)$, since $u$ gives a total computation tree of type $\dom_+({\sf LLPO}_n)\to\cod_+({\sf LLPO}_n)$, the \ppp computation $\Phi_u^+(a)$ halts.
\item {\em Actual.} For each \aaa instance $a\in\dom({\sf LLPO}_n)$, we have $\Phi_u^F(a)\subseteq{\sf LLPO}_n(a)$.
\end{itemize}
\item Given a code $e$, one can compute an index $r$ such that $\Phi_r^+(e,\pi)=\Phi_u^+(\ep(t(e)))$.
\end{itemize}

{\em Verification:}
\begin{itemize}
\item {\em Potential.}
If $e$ is a code of a sequence $\alpha\in 2^{n\times\N}$, then we have $\varphi_{\ep(t(e))}=\alpha$ as seen above, $\ep(t(e))$ is a code of a total computable function; that is, $\ep(t(e))\in\dom_+({\sf LLPO}_n)$.
Hence, the computation $\Phi_r^+(e,\pi)=\Phi_u^+(\ep(t(e)))$ halts as seen above.
\item {\em Actual.}
If $U(\alpha)$ is true, the code $\ep(t(e))$ of $\varphi_{\ep(t(e))}=\alpha$ belongs to $\dom({\sf LLPO}_n)$.
Hence, for any $k\in\Phi_r^F(e,\pi)=\Phi_u^F(\ep(t(e)))\subseteq {\sf LLPO}_n(\ep(t(e)))$ we have $\alpha_k=0^\infty$.
\end{itemize}

By the above characterization of a realizer for $\bflogic{LLPO}_n$, this means that $r$ is an $F$-realizer for $\bflogic{LLPO}_n$.
\end{proof}

\begin{remark}
Again, the upper bound $\sii{\sf C}_{1/\om}+\pii{\sf C}_{\geq 1/2}$ in Theorem \ref{thm:LLPO-reducible-to-realizable} is used solely to obtain a computable code for a given $\alpha\in 2^\N$.
Thus, by Theorem \ref{MP-WLEM-single-pc}, the upper bound can be replaced with $\sii{\sf C}_{1/\om}+{\sf C}_{2/3}$.
\end{remark}

\begin{prop}\label{prop:LLPO-reducible-to-realizable2}
If $\pii{\sf C}_{\geq n-1/n}\tgw F\tgw\sii{\sf C}_{1/\om}+{\sf C}_{2/3}$, then $\bflogic{LLPO}_n$ is $F$-realizable.
\end{prop}

%
%

Next, let us show the converse.

\begin{theorem}\label{thm:LLPOn-realizable-to-above}
If $\bflogic{LLPO}_n$ is $F$-realizable, then $\pii{\sf C}_{\geq n-1/n}\tgw F$.
\end{theorem}

\begin{proof}
The argument is similar to the proof of Theorem \ref{thm:MP-realizable-to-reducible}.
Let $r$ be an $F$-realizer for $\bflogic{LLPO}_n$.
\begin{itemize}
\item Recalling the specific description of a realizer for $\bflogic{LLPO}_n$ in the proof of Theorem \ref{thm:LLPO-reducible-to-realizable}, for any $F$-relative code $e$ of a total computable function, $\Phi_r^+(e,\pi)$ halts.
\item Given a computable function $\varphi_e$, its index $\eta(e)$ of an $F$-relative computation can be obtained by considering a computation that does not make any queries to $F$.
\item Then compute an index $m$ such that $\Phi_m^+(e)=\Phi_r^+(\eta(e),\pi)$.
\end{itemize}

{\em Verification:}
\begin{itemize}
\item {\em Potential.}
Assume that a code $e\in{\rm dom}_+({\sf LLPO}_n)$ of a total computable function $\alpha\in 2^{n\times\N}$ is given.
\begin{itemize}
\item Then $\eta(e)$ is an $F$-relative index of $\alpha=\varphi_e$; that is, $\Phi_{\eta(e)}^F(k,s)=\alpha(k,s)$.
\item Since $\eta(e)$ is an $F$-code of a total computable function, the computation $\Phi_m^+(e)=\Phi_r^+(\eta(e),\pi)$ must halt.
\end{itemize}
\item {\em Actual.}
For each \aaa instance $e\in\dom({\sf LLPO}_n)$, for $\alpha=\varphi_e$, the premise $U(\alpha)$ is true.
As seen in Theorem \ref{thm:MPn-reducible-to-realizable}, $\pi$ is its \aaa realizer.
Since $\eta(e)$ is an $F$-code of $\alpha$, and $r$ is a realizer for $\bflogic{LLPO}_n$, for any $k\in\Phi_m^F(e)=\Phi_r^F(\eta(e),0)$, we must have $\alpha_k=0^\infty$.
Therefore, we get $\Phi_m^F(e)\subseteq{\sf LLPO}_n(e)$.
\end{itemize}

Consequently, $m$ is a code of a reduction for ${\sf LLPO}_n\tgw F$.
By Proposition \ref{prop:LLPO-AcoUC}, we get the conclusion.
\end{proof}


\subsection{\bflogic{WLEM} and its variants}\label{sec:WLEM}

The weak law of excluded middle $\bflogic{WLEM}$ states that one of two incompatible principles, $A$ and $\neg A$, is false.
The principle $\bflogic{WLEM}_n$ states that one of $n$ mutually incompatible principles is false.

\begin{definition}\label{def:WLEM_n-def}
For each $n<\om$, then $\bflogic{WLEM}_n$ is the following principle:
\[
\left(\neg\bigvee_{\substack{k,\ell<n\\ k\not=\ell}}A_k\land A_\ell\right)\longrightarrow\bigvee_{k<n}\neg A_k.
\]

If $n=\om$, then $\bflogic{WLEM}_\om$ stands for the following principle:
\[
\big(\neg\exists k,\ell\in\N.\ (k\not=\ell \land A_k\land A_\ell)\big)\longrightarrow\exists k\in\N\ \neg A_k.
\]
\end{definition}

We link the co-unique choice principle ${\sf C}_{n-1/n}$ to $\bflogic{WLEM}_n$.

\begin{theorem}\label{thm:WLEMn-realizable-iff-reducible}
Let $F$ be a potted bilayer problem, and $n\leq\om$.
\[
{\sf C}_{n-1/n}\tgw F\iff\mbox{$\bflogic{WLEM}_n$ is $F$-realizable.}
\]
\end{theorem}

\begin{proof}
We describe the case where $k < \omega$, but the same applies when $k=\omega$.

{\em Realizer for $\bflogic{WLEM}_k$:}
\begin{itemize}
\item {\em Premise.}
Since the outmost connective of $\neg\vee_{n\not=m<k}A_m\land A_n$ is $\neg$, this is $F$-realizable iff $0$ is an $F$-\aaa realizer.
Moreover, this is $F$-realizable iff there is at most one $i$ such that $A_i$ is $F$-realizable.
\end{itemize}

Based on the argument in Example \ref{exa:AE-thm-realizer-simplify}, an $F$-realizer $r$ for $\bflogic{WLEM}_k$ is described as follows.
\begin{itemize}
\item {\em Potential.} $\Phi_r^+(0)$ outputs a number less than $k$.
\item {\em Actual.}
If there is at most one $F$-realizable one among the $A_i$'s, then some \aaa computation tree for $\Phi_r^F(0)$ outputs an $i<k$ such that $\neg A_i$ is $F$-realizable.
\begin{itemize}
\item That is, there is an \aaa computation tree $(\tactual_k,\nu_k,\psi_k)$ such that if $\sigma\in\leaf\tactual_k$ and $\psi_k(\sigma)=i$ then $\neg A_i$ is $F$-realizable.
\end{itemize}
\end{itemize}

For $A=(A_i)_{i<k}$, consider the following set:
\[N(A)=\{i<k:\neg A_i\mbox{ is $F$-realizable}\}.\]

To summarize the above, the following are equivalent.
\begin{enumerate}
\item $r$ is an $F$-\aaa realizer for $\bflogic{WLEM}_k$.
\item $\Phi_r^+(0)$ is a total computation tree; and if there is at most one $F$-realizable one among the $A_i$'s, then some \aaa computation tree for $\Phi_r^F(0)$ solves ${\sf C}_{\geq k-1/k}(\ast|N(A))$.
\end{enumerate}


\medskip
We show the equivalence $\iff$ in the statement of the theorem.

$\Rightarrow$:
Proposition \ref{prop:coUC-equal-WLEM}, we may assume ${\sf C}_{\geq k-1/k}\tgw F$ via $r$.
\begin{itemize}
\item {\em Potential.}
Any subset of $k$ is a secret \ppp instance of ${\sf C}_{k-1/k}$.
In particular, $N(A)\subseteq k$ can be thought of as a secret \ppp instance, so the computation $\Phi_r^+(0)$ halts.
\item {\em Actual.}
If $(A_i)_{i<k}$ satisfies the premise of $\bflogic{WLEM}_k$, then $(\ast | N(A))\in\dom({\sf C}_{\geq k-1/k})$.
Therefore, for each secret instance $N(A)$, there is an \aaa computation tree for $\Phi_r^F(0)$ which solves ${\sf C}_{\geq k-1/k}(\ast|N(A))$.
\end{itemize}

Consequently, by the above (2)$\Rightarrow$(1), $r$ is an $F$-\aaa realizer for $\bflogic{WLEM}_k$.

\medskip
$\Leftarrow$:
Let $r$ be an $F$-\aaa realizer for $\bflogic{WLEM}_k$.
For each $a,i<k$, consider the following formula:
\[
B^a_i\equiv (i=a)
\]

\begin{itemize}
\item $i\not=a$ is true iff $\neg B^a_i$ is true iff $0$ is an \aaa realizer for $\neg B^a_i$.
\item 
For any two distinct $u$ and $v$, it is impossible for $u=a$ and $v=a$ to hold simultaneously.
In other words, it is impossible for $B_u^a\land B_v^a$ to hold.
Therefore, the premise of $\bflogic{WLEM}_k$ for $(B^i_a)_{i<k}$ is always realizable.
\end{itemize}

To show ${\sf C}_{k-1/k}\tgw F$, proceed with the following algorithm:
\begin{itemize}
\item {\em Potential.}
Regardless of a \ppp instance $(\ast|A)$, it simulates $\Phi_r^+(0)$ as is.
This always halts.
\item {\em Actual.} 
For each \aaa instance $(\ast|A)\in\dom({\sf C}_{k-1/k})$, $A\subseteq k$ is a co-singleton, so it is of the form $A=k\setminus\{a\}$.

Since $r$ is an $F$-realizer for $\bflogic{WLEM}_k$, for the instance $(B^a_i)_{i<k}$, there is an \aaa computation tree $\tactual_{r|a}\subseteq \tpot_r$ for $\Phi_r^F(0)$ which returns an \aaa realizer for $\bigvee_{i<k}\neg B^a_i$.
This output must be $i\not=a$, so this computation solves ${\sf C}_{k-1/k}(\ast|A\setminus\{a\})$.
\end{itemize}

Consequently, we obtain ${\sf C}_{k-1/k}\tgw F$.
\end{proof}

The weak law of excluded middle $\bflogic{WLEM}$ ($=\bflogic{WLEM}_2$) is better known as {\em propositional de Morgan's law} $\neg(A\land B)\to \neg A\lor \neg B$.
We next consider de Morgan's law on the natural numbers, denoted by $\bflogic{DML}_\N$:
\begin{definition}\label{def:DML-nat}
$\bflogic{DML}_\N$ is the following principle:
\[
\neg\forall n\in\N\  A(n)\longrightarrow \exists n\in\N\ \neg A(n).
\]
\end{definition}


\begin{theorem}\label{thm:DML-reducible-iff-DNE-realizable}
Let $F$ be a potted bilayer problem.
\begin{align*}
{\sf C}_{1/\om}\tgw F&\iff\mbox{$\bflogic{DML}_\N$ is $F$-realizable.}
\end{align*}
\end{theorem}

\begin{proof}
Assume ${\sf C}_{1/\om}\tgw F$.
By Proposition \ref{prop:UC-equal-DML}, we have ${\sf C}_{\geq 1/\om}\tgw F$.

%
%

{\em Realizer for $\bflogic{DML}_\N$:}
Under the assumption ${\sf C}_{1/\om}\tgw F$, the following equivalence holds:
\[\mbox{$A(n)$ is $F$-realizable for all $n\in\N$}\iff\mbox{$\forall n\in\N.A(n)$ is $F$-realizable.}\]

The backward direction ($\Leftarrow$) is obvious, so we only show the forward direction ($\Rightarrow$):
\begin{itemize}
\item Let $r(n)$ be an $F$-\aaa realizer for $A(n)$.
\item As in Example \ref{exa:AE-thm-realizer-simplify}, it suffices to give an $F$-computation which, given $n\in\om$, returns an $F$-\aaa realizer for $A(n)$.
\item We thinking of this total function $r\colon\om\to\om$ as a potted problem.
Then, by Proposition \ref{prop:DML-upper-bound}, we get $r\tW{\sf C}_{1/\om}$.
\item By transitivity (Proposition \ref{prop:transitive}), we have $r\tgw F$.
This reduction gives an $F$-realizer for $\forall n\in\N.A(n)$.
\end{itemize}

Based on the above observation, we analyze a realizer for $\bflogic{DML}_\N$.
\begin{itemize}
\item {\em Premise.}
Since the outmost connective of $\neg\forall n\in\N.A(n)$ is $\neg$, this is $F$-realizable iff $0$ is an $F$-\aaa realizer.
Moreover, by the above observation, this is $F$-realizable iff $\forall n^\N A(n)$ is not $F$-realizable iff $A(n)$ is not $F$-realizable for some $n\in\N$.
\item {\em Realizer.}
As in Example \ref{exa:AE-thm-realizer-simplify}, if $r$ is an $F$-\aaa realizer for $\bflogic{DML}_\N$ then both $\Phi_r^+(0)$ and $\Phi_r^F(0)$ always halt, and if $A(n)$ is not $F$-realizable for some $n\in\N$, then it outputs such an $n$.
\end{itemize}

To construct a ${\sf C}_{\geq 1/\om}$-realizer for $\bflogic{DML}_\N$, we consider the following strategy $r$:
\begin{itemize}
\item {\tt Merlin} secretly plays an $A$.
In response, {\tt Arthur} asks {\tt Nimue} to make a query, and {\tt Nimue} plays the following secret move $N(A)$.
\[N(A)=\{n\in\N:A(n)\mbox{ is not $F$-realizable}\}.\]
\item In response to {\tt Merlin}'s answer $m$ in the next round, {\tt Arthur} declares $m$ as the output.
\end{itemize}

{\em Verification:}
\begin{itemize}
\item {\em Potential.}
{\tt Arthur} always declares the termination of the game, and outputs a natural number.
That is, $\Phi_r^+(0)$ always returns a \ppp realizer for $\exists n^\N \neg A(n)$.
\item {\em Actual.}
If the premise $\neg\forall n^\N A(n)$ is realizable, then $N(A)\not=\emptyset$ holds, so {\tt Nimue}'s query $N(A)$ is contained in ${\sf C}_{\geq 1/\om}$.
Then {\tt Arthur} outputs {\tt Merlin}'s answer $m\in N(A)$ as is, and this is an \aaa realizer of $\exists n^\N\neg A(n)$.
That is, for $G={\sf DML}_\om$, $\Phi^G_r(0)$ must output an \aaa realizer for $\exists n^\N\neg A(n)$.
\item As $G\tgw F$, this $G$-\aaa realizer can be transformed into an $F$-\aaa realizer.
\end{itemize}

This shows that $\bflogic{DML}_\N$ is $F$-realizable.

\medskip

$\Leftarrow$:
Let $r$ be an $F$-\aaa realizer for $\bflogic{DML}_\N$.
For each $a\in\N$, consider the following formula:
\[
B_a(n)\equiv (n\not=a)
\]

\begin{itemize}
\item $n=a$ is true iff $\neg B_a(n)$ is true iff $0$ is an \aaa realizer for $\neg B_a(n)$.
\item Therefore, the premise of $\bflogic{DML}_\N$ for $B_a$ is always realizable
\end{itemize}

To show ${\sf C}_{1/\om}\tgw F$, proceed with the following algorithm:
\begin{itemize}
\item {\em Potential.} Regardless of a \ppp instance $(\ast|A)$, it simulates $\Phi_r^+(0)$ as is.
This always halts.
\item {\em Actual.}
For each \aaa instance $(\ast|A)\in{\sf C}_{1/\om}$, $A$ is a singleton, so it is of the form $A=\{a\}$.

Since $r$ is an $F$-realizer for $\bflogic{DML}_\N$, for the instance $B_a$, there is an \aaa computation tree $\tactual_{r|a}\subseteq \tpot_r$ wgucg returns an \aaa realizer for $\exists n\neg B_a(n)$.
This output must be $n=a$, so this computation solves ${\sf C}_{1/\om}(\ast|\{a\})$.
\end{itemize}

Consequently, we obtain ${\sf C}_{1/\om}\tgw F$.
\end{proof}

\begin{remark}
In the effective topos, both ${\sf C}_{1/2}$ and ${\sf C}_{1/\om}$ are the strongest non-trivial oracles which induce the double negation topology.
In the total computable setting, we have seen in Theorem \ref{thm:separation-DML-WLEM} that ${\sf C}_{1/2}<_{\sf tW}^{\sf G}{\sf C}_{1/\om}$ holds.
The question here is whether ${\sf C}_{1/\om}$ induces the double negation topology.
By Proposition \ref{prop:DML-upper-bound}, one might expect ${\sf C}_{1/\om}$ to be the strongest non-trivial oracle; however:
\begin{itemize}
\item To realize $\neg\neg A\to A$, even when $A$ is not \aaa realizable, we must output an \ppp realizer for $A$.
\item However, $(\neg\neg A\to A)_+$ does not preserve any information about $A_+$, and ${\sf C}_{1/\om}$ has no way to restore it.
\item In general, an oracles with a fixed \ppp codomain, such as ${\sf C}_{1/\om}$ has no way to handle different \ppp realizers for each distinct $A$.
\end{itemize}

In fact, it seems correct to generally introduce oracles generated from logical principles as $\Omega$-valued functions, which have no fixed \ppp codomain.
\end{remark}


\subsection{$\bflogic{CT}_0$ and its variants}\label{sec:Church-thesis}

There is a previous work on the separation of strengths of the weak Markov principles using Kripke models \cite{HeLu16}.
However, unlike the Kripke model (the topological model), the realizability model automatically guarantees the validity of principles regarding computability.
Here, we consider some variants of {\bf Church's thesis for constructivists}.

First, $\bflogic{CT}_0$ is a principle stating that any total multifunction is computable.
\begin{definition}
$\bflogic{CT}_0$ is the following principle:
\[
\forall x\in\N\exists y\in\N.\ A(x,y)
\to
\exists e\in\N\forall x\in\N.\ {\varphi_e(x)\downarrow}\land A(x,\varphi_e(x))
].
\]

The principle obtained by replacing $\exists y$ with $\exists!y$ is called $\bflogic{CT}_0!$.
\end{definition}

Next, $\bflogic{nCT}_0$  is a principle stating that any partial multifunction is computable.
\begin{definition}
$\bflogic{nCT}_0$ is the following principle:
\begin{align*}
&\forall x\in\N.\ (\neg D(x)\to \exists y\in\N.\ A(x,y))\\
\longrightarrow\exists e\in\N&\forall x\in\N.\ [\neg D(x)\to {\varphi_e(x)\downarrow}\land A(x,\varphi_e(x))].
\end{align*}

The principle obtained by replacing $\exists y$ with $\exists!y$ is called $\bflogic{nCT}_0!$.
\end{definition}

Note that this is different from $\bflogic{ECT}_0$ (extended Church's thesis), which uses the notion of almost negativity; however, almost negativity performs well only for partial computability and is not a natural notion for total computability.

$\bflogic{CT}_0$ also does not quite align well with the total computability setting, in that its realizer requires an unbounded search for halting $\varphi_e(x)\down$ for computations.
For this reason, we introduce the following weaker principle.

\begin{definition}
$\bflogic{wCT}_0$ is the following principle\footnote{
Note that this is different from what is commonly referred to as $\bflogic{WCT}_0$.
In $\bflogic{WCT}_0$, $\neg\neg$ appears before $\exists e$.
}:
\begin{align*}
\forall x\in\N\exists y\in\N.\ A(x,y)
\to
\exists e\in\N\forall x\in\N.\ \neg\neg[\varphi_e(x)\down\land A(x,\varphi_e(x))].
\end{align*}

The principle obtained by replacing $\exists y$ with $\exists!y$ is called $\bflogic{wCT}_0!$.
\end{definition}

\begin{definition}
$\bflogic{wnCT}_0$ is the following principle:
\begin{align*}
&\forall x\in\N.\ (\neg D(x)\to \exists y\in\N.\ A(x,y))\\
\longrightarrow\exists e\in\N&\forall x\in\N.\ [\neg D(x)\to \neg\neg ({\varphi_e(x)\downarrow}\land A(x,\varphi_e(x)))].
\end{align*}

The principle obtained by replacing $\exists y$ with $\exists!y$ is called $\bflogic{wnCT}_0!$.
\end{definition}


Relativization to an oracle weaker than Markov's principle preserves weak Church's thesis.

\begin{theorem}\label{thm:MP-vs-CT0-realizability}
If $F\tgw\sii{\sf C}_{1/\om}$ then $\bflogic{wnCT}_0$ is $F$-realizable.
\end{theorem}

\begin{proof}
Assume $F\tgw G:=\sii{\sf C}_{1/\om}$.
For each index $e$, we have $\Phi_e^F\tgw F\tgw G$, so by transitivity (Proposition \ref{prop:transitive}), there is a total computable function $t$ such that $\Phi_{t(e)}^{G}$ refines $\Phi_e^{F}$.

\medskip
{\em Realizer for $\bflogic{wnCT}_0$.}
Let $a$ be an $F$-realizer for the premise $\forall x^\N(\neg D(x)\to \exists y^\N A(x,y))$ of $\bflogic{wnCT}_0$.
As in Example \ref{exa:AE-thm-realizer-simplify}, its $F$-realizer is described as follows:
\begin{itemize}
\item Since the outmost connective $\neg D(x)$ is $\neg$, we only need to consider $0$ as an \ppp/\aaa realizer.
Thus, $\Phi_a^+(x,0)$ yields a realizer for $\exists yA(x,y)$.
\item {\em Potential.}
For each $x\in\om$, we have $\Phi^+_{a}(x,0)\down$, and for any $\pair{y,z}\in\Phi^+_{a}(x,0)$, $z$ is a \ppp realizer for $A(x,y')$ for some $y'$.
\item {\em Actual.}
For each $x\in\om$, if $\neg D(x)$ is true, we have $\Phi^{F}_{a}(x,0)\down$, and for any $\pair{y,z}\in\Phi^{F}_{a}(x,0)$, $z$ is a $F$-\aaa realizer for $A(x,y)$.
\end{itemize}

Regarding the conclusion of $\bflogic{wnCT}_0$, since the part following $\exists e^\N$ starts with the negation, it suffices to provide an existential witness $e$.
\begin{itemize}
\item If $a$ is an \aaa realizer for the premise, and if $\neg D(x)$ is \aaa realizable, then $\Phi_a^F(x,0)$ halts.
\item Therefore, using a computable function in Theorem \ref{MP-equial-partial-computable}, we obtain $\varphi_{\delta(t(a))}(x)\down\in\Phi_{t(a)}^G(x)\subseteq\Phi_a^F(x,0)$.
\begin{itemize}
\item Its output is a pair, so let $e_i$ be an index of the computation that yields its $i$th coordinate; that is, $\varphi_{\delta(t(a))}(x)=\pair{\varphi_{e_0}(x),\varphi_{e_1}(x)}$.
\item This $e_0$ yields a desired index for an existential witness.
\end{itemize}
\item A realizer for $\bflogic{wnCT}_0$ is an algorithm which, given $a$, returns a code $e_0$ for $\pi_0\circ\varphi_{\delta(t(a))}$.
\end{itemize}

\medskip
\noindent
{\em Verification:}
\begin{itemize}
\item {\em Potential.}
\begin{itemize}
\item Since $t,\delta$ are total computable functions, $\delta(t(a))$ is always defined regardless of $a$.
\item Hence, a code $e_0$ for $\pi_0\circ\varphi_{\delta(t(a))}$ is always defined.
\end{itemize}
\item {\em Actual.}
If $a$ is an $F$-\aaa realizer, then let $x\in\om$ be such that $\neg D(x)$ is realizable.
\begin{itemize}
\item The computation $\varphi_{\delta(t(a))}(x)$ halts, and
its halting stage is an \aaa realizer for $\varphi_{\delta(t(a))}(x)\down$.
\item 
Furthermore, $\varphi_{e_1}(x)$ is an $F$-\aaa realizer for $A(x,\varphi_{e_0}(x))$.
In particular, $A(x,\varphi_{e_0})(x)$ is $F$-realizable.
\item
Consequently, if $\neg D(x)$ is realizable, then so is $\neg\neg(\varphi_{e_0}(x)\down\land A(x,\varphi_{e_0}(x)))$.
Its outmost connective is $\neg$, so $0$ gives its realizer.
Thus, for the part following $\exists^\N e$ in $\bflogic{wnCT}_0$, a code for a constant function that outputs $0$ regardless of the input provides an \aaa realizer for $\bflogic{wnCT}_0$.
\end{itemize}
\end{itemize}
Therefore, $\bflogic{wnCT}_0$ is $F$-realizable.
\end{proof}

If $F$ is exactly Markov's principle, then $\bflogic{w}$ can be eliminated.

\begin{theorem}\label{thm:MP-vs-CT0-realizability-2}
$\bflogic{nCT}_0$ is $\sii{\sf C}_{1/\om}$-realizable.
\end{theorem}

\begin{proof}
The proof is exactly the same as in Theorem \ref{thm:MP-vs-CT0-realizability} up to the construction of $e_0$.
Here, there is no need to go through $t$.
\begin{itemize}
\item By single-valuedness of $F=\sii{\sf C}_{1/\om}$, we can ensure $\Phi_{a}^{F}(x)=\{\varphi_{\delta(a)}(x)\}$.
\item 
Define $e_0$ as in Theorem \ref{thm:MP-vs-CT0-realizability}, and let $e_1$ be an $F$-index of $\pi_1\circ\Phi_{a}^F$.
\begin{itemize}
\item That is, $x\in\dom(\Phi_{a}^F)$ implies $\varphi_{\delta(a)}(x)=\pair{\varphi_{e_0}(x),\Phi^F_{e_1}(x)}$, and $\Phi^F_{e_1}(x)$ is an $F$-realizer for $A(x,\varphi_{e_0}(x))$.
\end{itemize}
\end{itemize}

Next, we construct an upper bound on the number of steps required to compute $\varphi_{e_0}(x)$.
\begin{itemize}
\item Given $x$, compute an index $d(x)$ such that $\varphi_{d(x)}(0)[s]\simeq\varphi_{e_0}(x)[s]$ for any stage $s$ (that is, an algorithm $d(x)$ fully simulates $\varphi_{e_0}(x)$, including computation steps).
\item Make a query $d(x)$ to ${\sf MP}_{\sf PR}$, and receive its solution $s(x)$.
\begin{itemize}
\item As the \ppp domain of ${\sf MP}_{\sf PR}$ is $\om$, we always receive some solution.
\item As ${\sf MP}_{\sf PR}\tW\sii{\sf C}_{1/\om}=F$, a similar result can be obtained using the $F$-computation.
\end{itemize}
\item Let $e_s$ be an index of the $F$-computation which, given an input $x$, first makes a query $d(x)$ to ${\sf MP}_{\sf PR}$, and outputs an response from ${\sf MP}_{\sf PR}$ as is.
\end{itemize}

Then, we show that $e_0,e_1,e_s$ realize the conclusion part of $\bflogic{nCT}_0$.

\medskip
\noindent
{\em Verification:}
\begin{itemize}
\item {\em Potential.}
\begin{itemize}
\item Since $\delta$ is a total computable function, regardless of $a$, $\delta(a)$ is always defined.
Hence, a code $e_0$ for $\pi_0\circ\varphi_{\delta(a)}$ is always defined.
The same applies to $e_1$.
Then, one can also compute a code $e_s$ since it is explicitly constructed from $e_0$.
\item Since both $\Phi_{e_s}^+$ and $\Phi_{e_1}^+$ are total computable, an index of this pair gives a \ppp realizer for the part following $\exists^\N e$ in $\bflogic{nCT}_0$ with an existential witness $e=e_0$.
\end{itemize}
\item {\em Actual.}
If $a$ is an $F$-\aaa realizer, then let $x\in\om$ be such that $\neg D(x)$ is realizable.
\begin{itemize}
\item Then the computations $\Phi_a^F(x)$ and $\varphi_{\delta(a)}(x)=\pair{\varphi_{e_0}(x),\Phi_{e_1}^F(x)}$ halt.
\item By the definition of ${\sf MP}_{\sf PR}$, ${\sf MP}_{\sf PR}(d(x))$ gives an upper bound on the halting stage of $\varphi_{d(x)}(0)$.
By the definition of $d$, this is a halting stage of $\varphi_{e_0}(x)$.
Therefore, $z\in\Phi_{e_s}^F(x)$ implies $\varphi_{e_0}(x)[z]\down$.
\item As mentioned above, $\Phi_{e_1}^F(x)$ gives an \aaa realizer for $A(x,\varphi_{e_0}(x))$.
\item Consequently, whenever $\neg D(x)$ is realizable, an index of the pair $\Phi_{e_s}^G,\Phi_{e_1}^G$ gives an \aaa realizer for the part following $\exists^\N e$ in $\bflogic{nCT}_0$ with an existential witness $e=e_0$.
\end{itemize}
\end{itemize}

Therefore, $\bflogic{nCT}_0$ is $F$-realizable, where recall $F=\sii{\sf C}_{1/\om}$.
\end{proof}

\begin{prop}\label{prop:id-CT0}
$\bflogic{CT}_0$ is ${\rm id}_\om$-realizable.
\end{prop}

\begin{proof}[Proof (Sketch)]
A realizer for the premise of $\bflogic{CT}_0$ yields an index of a total computable function, so we can compute its halting stage.
\end{proof}

The converse of Theorem \ref{thm:MP-vs-CT0-realizability} also holds.
Since we are dealing here with constructive Church's thesis in first-order arithmetic, the problems must also be written as arithmetical formulas.

\begin{theorem}\label{thm:MP-vs-CT0-realizability-rev}
Let $F\pcolon\om\tto\om$ be a potted problem whose graph is arithmetically definable (in classical logic).
\begin{enumerate}
\item If $\bflogic{wnCT}_0$ is $F$-realizable, then $F\tgw\sii{\sf C}_{1/\om}$.
\item Let $F$ be total, i.e., $\dom_+(F)=\om$.
If $\bflogic{wCT}_0$ is $F$-realizable, then $F\tgw\sii{\sf C}_{1/\om}$.
\end{enumerate}
\end{theorem}
\begin{proof}
Assume that the graph of $F$ is defined by an arithmetical formula $G_F(x,y)$ (which describes $y\in F(x)$).
Then, the \ppp domain is also expressed by an arithmetical formula $D_F(x)$.
Now, consider the double negation translations of two formulas $G_F^{\neg\neg}(x,y)$ and $D_F^{\neg\neg}(x)$.
Since the underlying logic is classical, these translations define the same function and the same \ppp domain.
By Observation \ref{obs:negneg-real-truth}, such a negative formula is true iff it is $F$-realizable by $0$ (by $\neg\neg$-density).

\begin{itemize}
\item 
Thus, the premise $\forall x(\neg\neg D_F^{\neg\neg}(x)\to\exists y G_F^{\neg\neg}(x,y))$ is $F$-realizable:
Apply a realizer for $\bflogic{wnCT}_0$ to its trivial realizer for the above premise.
\item This produces an index $e$ such that $G_F^{\neg\neg}(x,\varphi_e(x))$ is $F$-realizable.
This implies that $\varphi_e(x)\in F(x)$ is true, so $\varphi_e$ is a choice function for $F$.
\item Since $\varphi_e(x)$ is a partial computable function, Theorem \ref{MP-equial-partial-computable} implies $\varphi_e\tgw\sii{\sf C}_{1/\om}$; in particular, $F\tgw\sii{\sf C}_{1/\om}$.
\end{itemize}

When $F$ is total, $D_F$ is unnecessary, so the principle $\bflogic{wCT}_0$ is sufficient.
\end{proof}


In the setting of partial computability, $\pii{\sf C}_{\geq 1/2}$-realizability corresponds to the so-called {\em Lifschitz realizability} (see e.g.~\cite{vOBook}), which was introduced to realize the unique Church's thesis ${\sf CT}_0!$ (Lifschitz \cite{Lif79}).
In the setting of total computability, $\pii{\sf C}_{\geq 1/2}$-realizability (total Lischitz realizability) realizes a weak variant to the unique Church's thesis.

\begin{theorem}\label{thm:MP-LLPO-models-CT0u}
If $F\tgw\sii{\sf C}_{1/\om}+\pii{\sf C}_{\geq 1/2}$ then $\bflogic{wnCT}_0!$ is $F$-realizable.
\end{theorem}

\begin{proof}
Assume $F\tgw G:=\sii{\sf C}_{1/\om}+\pii{\sf C}_{\geq 1/2}$.
For each index $e$, we have $\Phi_e^F\tgw F\tgw G$, so by transitivity (Proposition \ref{prop:transitive}), there is a total computable function $t$ such that $\Phi_{t(e)}^{G}$ refines $\Phi_e^{F}$.

\medskip

{\em Realizer for $\bflogic{wnCT}_0!$.}
Let $a$ be an $F$-realizer for the premise $\forall x^\N(\neg D(x)\to\exists! y^\N A(x,y))$ of $\bflogic{wnCT}_0!$.
As in Example \ref{exa:AE-thm-realizer-simplify}, its $F$-realizer is described as follows:
\begin{itemize}
\item Since the outmost connective $\neg D(x)$ is $\neg$, we only need to consider $0$ as an \ppp /\aaa realizer.
Thus, regardless of $x$, $\Phi_a^+(x,0)$ is always defined, and if $\neg D(x)$ is true, $\Phi_a^F(x,0)$ produces a realizer for $\exists ! yA(x,y)$.
\end{itemize}

Here, if we expand the $\exists!y$ part, we see that $\Phi_a^+(x,0)$ is attempting to $F$-realize the following:
\[
\exists y^\N(A(x,y)\land \forall z^\N (A(x,z)\to y=z)).
\]

As in Example \ref{exa:AE-thm-realizer-simplify}, its $F$-realizer is described as follows.
Hereafter, we omit $0$ and write $\Phi_a^F(x)$.
\begin{itemize}
\item {\em Potential.}
For each $x\in\om$, we have $\Phi^+_{a}(x)\down$, and for any $\pair{y,w}\in\Phi^+_{a}(x)$, $\pi_0 w$ is a \ppp realizer for $A(x,y')$ for some $y'$.
\item {\em Actual.}
If $\neg D(x)$ is true, we have $\Phi^F_{a}(x)\down$, and for any $\pair{y,w}\in\Phi^F_{a}(x)$, $\pi_0w$ is an $F$-\aaa realizer for $A(x,y)$.
\begin{itemize}
\item Such a $y$ is unique:
If $\pair{y,v},\pair{z,w}\in\Phi_a^F(x)$ then $\pi_0v$ and $\pi_0w$ give realizers for $A(x,y)$ and $A(x,z)$, respectively.
Then, $\pi_1v(\pi_0w)$ is a realizer for $y=z$.
\end{itemize}
\end{itemize}

\medskip

To realize the conclusion of $\bflogic{wnCT}_0!$, we consider the following process:
\begin{itemize}
\item Put $\Phi_{p(a)}^+=\pi_0\circ\Phi_a^+$.
\begin{itemize}
\item By the above argument, if $a$ is an $F$-\aaa realizer for the premise, then $\Phi_{p(a)}^F$ is single-valued.
\end{itemize}
\item 
Using a computable function $\ep$ in Theorem \ref{LLPO-equial-partial-computable}, we take $e_0:=\ep(t(p(a))$.
\begin{itemize}
\item If $a$ is an $F$-\aaa realizer and $\neg D(x)$ is true ($=$ realizable by $0$), we obtain $\Phi_{p(a)}^F(x)=\Phi_{t(p(a))}^G(x)=\varphi_{\ep(t(p(a)))}(x)$.
\end{itemize}
\item We claim that the process $a\mapsto\ep(t(p(a))=e_0$ gives a realizer for $\bflogic{wnCT}_0!$.
\end{itemize}

\medskip
\noindent
{\em Verification:}
\begin{itemize}
\item {\em Potential.}
\begin{itemize}
\item Since $p,t,\delta$ are total computable functions, $e_0=\ep(t(p(a))$is always defined regardless of $a$.
\end{itemize}
\item {\em Actual.}
If $a$ is an $F$-\aaa realizer, then let $x\in\om$ be such that $\neg D(x)$ is realizable.
\begin{itemize}
\item The computation $\varphi_{e_0}(x)$ halts, and its halting stage is an \aaa realizer for $\varphi_{e_0}(x)\down$.
\item 
Furthermore, for any $\pair{y,w}\in\Phi^F_{a}(x)$, $\pi_0w$ is an $F$-\aaa realizer for $A(x,y)$.
In addition, we must have $y=\varphi_{e_0}(x)$, so $\pi_0w$ is an $F$-\aaa realizer for $A(x,\varphi_{e_0}(x))$.
\item
Consequently, if $\neg D(x)$ is realizable, then so is $\neg\neg(\varphi_{e_0}(x)\down\land A(x,\varphi_{e_0}(x)))$.
Its outmost connective is $\neg$, so $0$ gives its realizer.
Thus, the part following $\exists^\N e$ in $\bflogic{wnCT}_0!$, a code for a constant function that outputs $0$ regardless of the input procides an \aaa realizer for $\bflogic{wnCT}_0!$.
\end{itemize}
\end{itemize}
Therefore, $\bflogic{wnCT}_0!$ is $F$-realizable.
\end{proof}

If $\sii{\sf C}_{1/\om}\tgw F$ then, as in Theorem \ref{thm:MP-vs-CT0-realizability-2}, $\bflogic{w}$ can be eliminated.

\begin{theorem}\label{thm:MP-LLPO-models-CT0u-2}
If $\sii{\sf C}_{1/\om}\tgw F\tgw\sii{\sf C}_{1/\om}+\pii{\sf C}_{\geq 1/2}$ then $\bflogic{nCT}_0!$ is $F$-realizable.
\end{theorem}

\begin{proof}
The proof is exactly the same as in Theorem \ref{thm:MP-LLPO-models-CT0u} up to the construction of $e_0$.
Then incorporate the argument in Theorem \ref{thm:MP-vs-CT0-realizability-2}.
\begin{itemize}
\item Define $e_0$ as in Theorem \ref{thm:MP-LLPO-models-CT0u}, and let $e_1$ be an $F$-index of $\pi_0\circ\pi_1\circ\Phi_{a}^F$.
\begin{itemize}
\item Here, if $\neg D(x)$ is true, for any $\pair{y,w}\in\Phi_a^F(x)$, $\pi_0w$ is an $F$-\aaa realizer for $A(x,y)$.
\item By uniqueness, we must have $y=\varphi_{e_0}(x)$; hence, any element in $\Phi^F_{e_1}(x)=\pi_0\circ\pi_1\circ\Phi^F_a(x)$ is an $F$-\aaa realizer for $A(x,\varphi_{e_0}(x))$.
\end{itemize}
\end{itemize}

Next, the construction of an index $e_s$ computing an upper bound on the halting stage of $\varphi_{e_0}(x)$ is the same as in Theorem \ref{thm:MP-vs-CT0-realizability-2}.
Here, we use ${\sf MP}_{\sf PR}\tW\sii{\sf C}_{1/\om}\tgw F$.
Then, we show that $e_0,e_1,e_s$ realize the conclusion part of $\bflogic{nCT}_0!$.

\medskip
\noindent
{\em Verification:} 
\begin{itemize}
\item {\em Potential.}
\begin{itemize}
\item Since $\delta$ is a total computable function, regardless of $a$, $\delta(a)$ is always defined.
Hence, a code $e_0$ for $\pi_0\circ\varphi_{\delta(a)}$ is always defined.
The same applies to $e_1$.
Then, one can also compute a code $e_s$ since it is explicitly constructed from $e_0$
\item Since both $\Phi_{e_s}^+$ and $\Phi_{e_1}^+$ are total computable, an index of this pair gives a \ppp realizer for the part following $\exists^\N e$ in $\bflogic{nCT}_0!$ with an existential witness $e=e_0$.
\end{itemize}
\item {\em Actual.}
If $a$ is an $F$-\aaa realizer, then let $x\in\om$ be such that $\neg D(x)$ is realizable.
\begin{itemize}
\item 
Then the computations $\varphi_{e_0}(x)$ and $\Phi_{e_1}^F(x)$ halt.
\item By the same argument as in Theorem \ref{thm:MP-vs-CT0-realizability-2}, the computation $\Phi_{e_s}^F(x)$ halts and gives an upper bound on the halting stage of $\varphi_{e_0}(x)$.
That is, $z\in\Phi_{e_s}^F(x)$ implies $\varphi_{e_0}(x)[z]\down$.
\item As mentioned above, $\Phi_{e_1}^F(x)$ gives an \aaa realizer for $A(x,\varphi_{e_0}(x))$.
\item Consequently, whenever $\neg D(x)$ is realizable, an index of the pair $\Phi_{e_s}^F,\Phi_{e_1}^F$ gives an \aaa realizer for the part following $\exists^\N e$ in $\bflogic{nCT}_0!$ with an existential witness $e=e_0$.
\end{itemize}
\end{itemize}
Therefore, $\bflogic{nCT}_0!$ is $F$-realizable.
\end{proof}

\begin{theorem}\label{thm:WLEMthree-CTrealizable}
If $F\tgw {\sf C}_{2/3}$, then $\bflogic{nCT}_0!$ is $F$-realizable.
\end{theorem}

\begin{proof}
The argument is exactly the same as in Theorem \ref{thm:MP-LLPO-models-CT0u}.
Here, we use a computable function $\ep$ in Theorem \ref{MP-WLEM-single-pc2}.

\medskip

Assume $F\tgw G={\sf C}_{2/3}$.
For each index $e$, we have $\Phi_e^F\tgw F\tgw G$, so by transitivity (Proposition \ref{prop:transitive}), there is a total computable function $t$ such that $\Phi_{t(e)}^{G}$ refines $\Phi_e^{F}$.

\medskip

{\em Realizer for $\bflogic{nCT}_0!$}.
Let $a$ be an $F$-realizer for the premise $\forall x^\N(\neg D(x)\to\exists! y^\N A(x,y))$ of $\bflogic{nCT}_0!$.
As in Theorem \ref{thm:MP-LLPO-models-CT0u}, by omitting the input $0$, such an $F$-realizer $a$ is described as follows:
\begin{itemize}
\item {\em Potential.}
For each $x\in\N$, we have $\Phi^+_{a}(x)\down$, and for any $\pair{y,w}\in\Phi^+_{a}(x)$, $\pi_0 w$ is a \ppp realizer for $A(x,y')$ for some $y'$.
\item {\em Actual.}
If $\neg D(x)$ is true, we have $\Phi^F_{a}(x)\down$, and for any $\pair{y,w}\in\Phi^F_{a}(x)$, $\pi_0w$ is an $F$-\aaa realizer for $A(x,y)$.
\begin{itemize}
\item Such a $y$ is unique:
If $\pair{y,v},\pair{z,w}\in\Phi_a^F(x)$ then $\pi_0v$ and $\pi_0w$ give realizers for $A(x,y)$ and $A(x,z)$, respectively.
Then, $\pi_1v(\pi_0w)$ is a realizer for $y=z$.
\end{itemize}
\end{itemize}

\medskip

To realize the conclusion of $\bflogic{nCT}_0!$, we need to construct an index $e$ of a computation, an upper bound $z$ of its halting stage, and a realizer $w$ for $A(x,\varphi_e(x))$ in an $F$-computable manner.

{\em Index $e$.}
\begin{itemize}
\item Put $\Phi_{p(a)}^+=\pi_0\circ\Phi_a^+$.
\begin{itemize}
\item By the above argument, if $a$ is an $F$-\aaa realizer for the premise, then $\Phi_{p(a)}^F$ is single-valued.
\end{itemize}
\item 
Using a computable function $\ep$ in Theorem \ref{MP-WLEM-single-pc2}, we take $e_0:=\ep(t(p(a))$.
\begin{itemize}
\item For each \ppp realizer $a$, as seen above, $\Phi_a^+$ is total, so is $\Phi_{t(p(a))}^+$ (as a $G$-relative computation).
Therefore, by Theorem \ref{MP-WLEM-single-pc2}, $\ep(t(p(a)))$ is an index of a total computable function.
\item If $a$ is an $F$-\aaa realizer and $\neg D(x)$ is true, we obtain $\Phi_{p(a)}^F(x)=\Phi_{t(p(a))}^G(x)=\varphi_{\ep(t(p(a)))}(x)$.
\end{itemize}
\item We claim that the process $a\mapsto\ep(t(p(a))=e_0$ gives a desired index of a computation.
\end{itemize}

{\em Realizer $w$ for $A$.}
\begin{itemize}
\item Let $e_1$ be an $F$-index of $\pi_0\circ\pi_1\circ\Phi_{a}^F$.
\begin{itemize}
\item Here, if $\neg D(x)$ is true, for any $\pair{y,w}\in\Phi_a^F(x)$, $\pi_0w$ is an $F$-\aaa realizer for $A(x,y)$.
\item By uniqueness, we must have $y=\varphi_{e_0}(x)$; hence, any element in $\Phi^F_{e_1}(x)=\pi_0\circ\pi_1\circ\Phi^F_a(x)$ is an $F$-\aaa realizer for $A(x,\varphi_{e_0}(x))$.
\end{itemize}
\end{itemize}

{\em Halting stage $z$}
\begin{itemize}
\item Since $\varphi_{e_0}$ is total, the algorithm $z$ which, given $x\in\N$, computes the halting step $\varphi_{e_0}(x)[z(x)]\down$ is total computable.
Therefore, we can get an index $e_s$ for the $F$-computation of $z$.
\end{itemize}

Then, we show that $e_0,e_1,e_s$ realize the conclusion part of $\bflogic{nCT}_0$.

\medskip
\noindent
{\em Verification:}
\begin{itemize}
\item {\em Potential.}
\begin{itemize}
\item As before, regardless of $a$, a code $e_0=\ep(t(p(a)))$ is always defined.
The same applies to $e_1$.
Then, one can also compute a code $e_s$ since it is explicitly constructed from $e_0$.
\item Since both $\Phi_{e_s}^+$ and $\Phi_{e_1}^+$ are total computable, an index of this pair gives a \ppp realizer for the part following $\exists^\N e$ in $\bflogic{nCT}_0!$ with an existential witness $e=e_0$.
\end{itemize}
\item {\em Actual.}
If $a$ is an $F$-\aaa realizer, then let $x\in\om$ be such that $\neg D(x)$ is realizable.
\begin{itemize}
\item 
Then the computations $\varphi_{e_0}(x)$ and $\Phi_{e_1}^F(x)$ halt.
\item As before, the computation $\Phi_{e_s}^F(x)$ halts and gives an upper bound on the halting stage $z(x)$ of $\varphi_{e_0}(x)$.
\item As mentioned above, $\Phi_{e_1}^F(x)$ gives an \aaa realizer for $A(x,\varphi_{e_0}(x))$.
\item Consequently, whenever $\neg D(x)$ is realizable, an index of the pair $\Phi_{e_s}^F,\Phi_{e_1}^F$ gives an \aaa realizer for the part following $\exists^\N e$ in $\bflogic{nCT}_0!$ with an existential witness $e=e_0$.
\end{itemize}
\end{itemize}

Therefore, $\bflogic{nCT}_0!$ is $F$-realizable.
\end{proof}

\begin{theorem}\label{thm:WLEMthree-CTrealizable2}
If $\sii{\sf C}_{1/\om}\tgw F\tgw\sii{\sf C}_{1/\om}+{\sf C}_{2/3}$, then $\bflogic{nCT}_0!$ is $F$-realizable.
\end{theorem}

\begin{proof}
The argument is exactly the same as in Theorem \ref{thm:MP-LLPO-models-CT0u-2}.
Here, we use a computable function $\ep$ in Theorem \ref{MP-WLEM-single-pc}.
\end{proof}

\subsection{Failure of $\bflogic{CT}_0$}

The reason we have dealt with the artificial principle $\bflogic{wCT}_0$ rather than the usual constructive Church's thesis $\bflogic{CT}_0$ is that, in fact, $\bflogic{CT}_0$ is not $F$-realizable for most {\em computable} oracles $F$.

\begin{theorem}\label{thm:MPn-sheaf-not-CT0}
$\bflogic{CT}_0$ is not $\sii{\sf C}_{n-1/n}$-realizable.
\end{theorem}

\begin{proof}
Consider $F=\pii{\sf C}_{n-1/n}$, which is $\equiv_{\sf tW}$-equivalent to $\sii{\sf C}_{n-1/n}$ by Theorem \ref{thm:basic-Sigma-1-counique}.
Suppose for the sake of contradiction that $\bflogic{CT}_0$ is $F$-realizable by a realizer $r$.
\begin{itemize}
\item Since $F$ is arithmetic, $\Phi_d^F(x)\down=y$ can also be expressed as an arithmetical formula with parameters $(d,x,y)$.
\item For each $d$, let $A_d(x,y)$ be the double negation translation of such a formula.
\item Since its outmost connective is $\neg$, $\Phi_d^F(x)\down=y$ is true iff $0$ is an \aaa realizer for $A_d(x,y)$.
\end{itemize}

For each $d\in\N$, let $i(d)$ be an index of the following computation making a realizer for the premise $\forall x\exists yA_d(x,y)$:
\begin{itemize}
\item Let $\iota$ be a total computable function such that $\iota(y)=\pair{y,0}$.
Then, put $\Phi_{i(d)}^+=\iota\circ\Phi_d^+$.
Then note that $y\in\Phi_d^+(x)$ iff $\pair{y,0}\in\Phi_{i(d)}^+(x)$.
\end{itemize}

{\em Verification:}
\begin{itemize}
\item {\em Potential.}
If $\Phi_d^+$ is total, so is $\Phi_{i(d)}^+$.
\begin{itemize}
\item In particular, $i(d)$ is a \ppp realizer for $\forall x\exists yA_d(x,y)$.
\end{itemize}
\item {\em Actual.} 
If $\Phi_d^F$ is total, so is $\Phi_{i(d)}^F$, and we have $\Phi_{i(d)}^F(x)=\Phi_{d}^F(x)\times\{0\}$.
Then $i(d)$ is an \aaa $F$-realizer for $\forall x\exists yA_d(x,y)$:
\begin{itemize}
\item Given $x\in\om$, if $\pair{y,0}\in\Phi^F_{i(d)}(x)$ then $y\in \Phi_d^F(x)$ holds.
Thus, $0$ realizes $A_d(x,y)$,
so $\pair{y,0}$ is an $F$-\aaa realizer for $\exists yA_d(x,y)$.
\end{itemize}
\end{itemize}

Therefore, if $\Phi_{d}^+$ is a total function, then $i(d)$ is a \ppp realizer of the premise of $\bflogic{CT}_0$, so the computation $\Phi_r^+(i(d))$ halts and outputs a \ppp realizer for the conclusion $\exists e\forall x\ \varphi_e(x)\down\land A_d(x,\varphi_e(x))$.
If $\Phi_{i(d)}^F$ is an \aaa realizer, then $\Phi_r^F(i(d))$ yields an \aaa realizer for the conclusion. 
An explicit description of a realizer for the conclusion is:
\begin{itemize}
\item {\em Potential.}
It is the pair of a natural number $e$ and a code $p$ of a total computable function such that, for any $x\in\om$, the computation $\Phi_p^+(x)$ halts and outputs a pair $\pair{s,a}$.
\item {\em Actual.}
For each $\pair{s,a}\in\Phi_p^F(x)$, we have $\varphi_e(x)[s]\down$, and moreover, $a$ is an \aaa realizer for $A_d(x,\varphi_e(x))$.
In particular, $\varphi_e(x)\in\Phi_d^F(x)$ is true.
\end{itemize}

So far, we have discussed the general behavior of a total function of an index $d$; from here on, we will specify a particular index $d$.
We may assume $\cod_+(F)=n$.

\medskip
\noindent
{\em Construction:}
For each instance $x\in\om$ of $\Phi_d^+$, the \ppp computation tree $(\tpot^d_x,\nu^d_x,\psi^d_x)$ is given as follows:
\begin{itemize}
\item $\tpot_x^d$ is the full $n$-branching tree $n^{\leq x}$ of height $x$.
\item 
The query assigned to an internal node $\sigma\in \internal \tpot_x^d$ refers to the algorithm $\nu^d_x(\sigma)=q(\sigma)$, which will be defined later.
\item The output assigned to an leaf $\sigma\in\leaf \tpot_x^d$ is the leaf $\psi^d_x(\sigma)=\sigma$ itself.
\end{itemize}

Note that this algorithm $d$ refers to an algorithm $q(\sigma)$ (which will be defined later).
This $q(\sigma)$ will refer to $d$, so $d$ is defined using the recursion theorem.
We will ensure that $q(\sigma)$ is always defined; therefore, regardless of $x$, $\Phi_d^+(x)$ is always defined.
In other words, $\Phi_d^+$ is total.
Also, note that this is a computation that outputs the leaves of the computation tree itself.

\medskip
\noindent

For the $F$-computable function given by this code $d$, $\bflogic{CT}_0$ must yield an index $e$ and a halting stage $s$.
We attempt to approximate these values.
In fact, the finiteness of the associated computation trees allows for a certain type of finite approximation.
In details, for the index $d$, we perform the following computation:
\begin{itemize}
\item Since $\Phi_d^+$ is total, the computation $\Phi_r^+(i(d))$ must halt.
\item First, we analyze the \ppp computation tree $(\tpot,\nu,\psi)$ for $\Phi_{r}^+(i(d))$.
As $\cod_+(F)=n$, $\tpot$ is the full $n$-branching tree, which is finite by well-foundedness.
\item Hence one can compute the full information of the finite set $\Phi_r^+(i(d))=\{\psi(\tau):\tau\in \leaf \tpot\}$ of all resulting outputs.
As seen above, each output $\psi(\tau)$ is the pair of an index $e$ and a code $p$ of an $F$-computation, so take the set $E^d=\pi_0\circ\Phi_r^+(i(d))$ consisting of the former indices.
Let $c=c_d$ be the number of elements in this finite set.
\item For each $\pair{e,p}\in\Phi_r^+(i(d))$, let $(\tpot_p,\nu_p,\psi_p)$ be the \ppp computation tree for $\Phi_{p}^+(c+1)$.
Similarly, $\tpot_p$ is also a finite set, one can also compute the full information of the finite set $\Phi_p^+(c+1)=\{\psi_p(\tau):\tau\in\leaf \tpot_p\}$ of all resulting outputs.
\begin{itemize}
\item {\em Actual.}
If $\Phi_d^F$ is total, for any $\pair{e,p}\in\Phi_r^F(i(d))$ and $\pair{s,a}\in\Phi_p^F(x)$, we have $\varphi_e(x)[s]\down$.
\end{itemize}
\item Compute an upper bound on the possible computation stages, $s(d)=\max(\pi_0\circ\Phi_p^+(c+1))$.
\item
Restrict $E^d$ to indices of computations that halt by stage $s(d)$ and output a leaf of the computation tree $\tpot^d_{c+1}$.
\[
H^d=\{e\in E^d:\varphi_e(c+1)[s(d)]\down\in\leaf \tpot^d_{c+1}=n^{c+1}\}.
\]
\begin{itemize}
\item {\em Actual.}
If $\Phi_d^F$ is total, for any $\pair{e,p}\in\Phi_r^F(i(d))$, we have $\varphi_e(c+1)[s_d]\down$.
Since $\Phi_d^F$ is a computation which outputs a leaf of its computation tree, we get $\varphi_e(c+1)\in \Phi_d^F(c+1)\subseteq\leaf \tpot^d_{c+1}$; therefore, $e\in H^d$.
\end{itemize}
\end{itemize}

\medskip
\noindent
The construction of the algorithm $q(\sigma)$:
Based on the above, for each $\sigma\in n^{<\om}$, we construct an algorithm $q(\sigma)$ including a reference to $d$ as follows:
\begin{enumerate}
\item First put $H^d_0=H^d$,
At stage $s\leq c$, we inductively assume that we have constructed $H^d_s\subseteq H^d$.
\item If $H^d_s\not=\emptyset$:
\begin{enumerate}
\item Pick the least element $e\in H^d_s$, and put $H^d_{s+1}=H^d_s\setminus\{e\}$.
\item Since $e\in H^d$, we have $\varphi_e(c+1)[s(d)]\in n^{c+1}$.
Therefore, $\varphi_e(c+1)[s(d)]$ extends $\sigma\fr i$ for some node $\sigma\in n^s$ of height $s$ and $i<n$.
\item Then let $q(\sigma)$ be an index of the $\Pi_1$ co-singleton $P_{q(\sigma)}=n\setminus\{i\}$.
\item 
For other nodes $\tau\in n^s$ of height $s$, we set $P_{q(\tau)}=n\setminus\{0\}$.
\end{enumerate}
\item If either $H^d_s=\emptyset$ or $s>c$:
\begin{itemize}
\item For all nodes $\tau\in n^s$ of height $s$, we set $P_{q(\tau)}=n\setminus\{0\}$.
\end{itemize}
\end{enumerate}

{\em Verification:}
\begin{itemize}
\item By our construction, for any $\sigma\in n^{<\om}$, $P_{q(\sigma)}$ is a co-singleton; hence, $q(\sigma)\in\dom(F)$.
This ensures totality of $\Phi_d^F$.
\item Therefore, $\Phi_r^F(i(d))$ produces an \aaa realizer for the conclusion of $\bflogic{CT}_0$.
In particular, for any $\pair{e,p}\in\Phi_r^F(i(d))$, we get $\varphi_e(c+1)[s(d)]\down\in \Phi_d^F(c+1)$; hence, $e\in H^d$.
\item $H^d$ has at most $c$ elements, and $(H_s^d)_s$ is a decreasing sequence in which elements are removed one by one starting from $H^d$; thus, there is a stage $s$ such that $e=\min H^d_s$, so $e$ is removed from $H^d$ at some stage.
\item $\varphi_e(c+1)[s(d)]$ extends some $\sigma\fr i$; in this case, by the construction (2-c), we get $i\not\in P_{q(\sigma)}$.
This implies $\sigma\fr i\not\in \tpot^d_{c+1}$.
In particular, we obtain $\varphi_e(c+1)[s(d)]\not\in \tpot^d_{c+1}$.
\item
By our construction, for each input $x$, $\Phi_d^F$ outputs a leaf of the computation tree $\tpot^d_x$, we must have $\varphi_e(c+1)\in\Phi_d^F(c+1)\subseteq \tpot^d_{c+1}$.
However, this is impossible, which leads to a contradiction.
\end{itemize}
Consequently, $\bflogic{CT}_0$ is not $F$-realizable.
\end{proof}

Recall from Theorem \ref{thm:WLEMthree-CTrealizable} that $\bflogic{CT}_0!$ is $\sii{\sf C}_{n-1/n}$-realizable for $n\geq 3$.
However, interestingly, we next show that this is false for $n=2$.
Indeed, $\bflogic{CT}_0!$ is not $\pii{\sf C}_{\geq 1/2}$-realizable.
This is a crucial difference from Lifschitz realizability (partial $\pii{\sf C}_{\geq 1/2}$-realizability).

\begin{theorem}\label{thm:LLPO-not-CT0unique}
$\bflogic{CT}_0!$ is neither $\sii{\sf C}_{1/2}$-realizable nor $\pii{\sf C}_{\geq 1/2}$-realizable.
\end{theorem}

\begin{proof}
The argument is the same as in Theorem \ref{thm:MPn-sheaf-not-CT0}.
\begin{itemize}
\item 
Here, note that the previous argument applies not only to $\pii{\sf C}_{1/2}$ but also to $\pii{\sf C}_{\geq 1/2}$.
\end{itemize}

The key point is that, for $n=2$, a co-singleton is also a singleton.
\begin{itemize}
\item 
The previous construction yields a $\Pi_1$ co-singleton $P_{q(\sigma)}\subseteq n$ for each $\sigma\in n^{<\om}$; however, when $n=2$, this is also a $\Pi_1$ singleton.
\item This implies that the \aaa computation tree $\tactual_x^d\subseteq \tpot_x^d$ forms a single chain; thus, $\tactual_x^d$ has a unique leaf.
In particular, there is only one possible value for $\psi_x^d$.
\item Consequently, $\Phi_d^F$ yields a total single-valued function.
Therefore, $d$ is a realizer for the premise $\forall x\exists! yA_d(x,y)$ of $\bflogic{CT}_0!$.
\end{itemize}

The remaining part follows from an argument entirely analogous to that of Theorem \ref{thm:MPn-sheaf-not-CT0}.
\end{proof}

\begin{obs}\label{thm:WLEM-not-CT0unique}
$\bflogic{CT}_0!$ is not ${\sf C}_{1/2}$-realizable.
\end{obs}

\begin{proof}
Let $h\colon\om\to 2$ be a non-computable arithmetical function (e.g.~the characteristic function of the halting problem), and let $H$ be the double negation translation of its graph.
By Proposition \ref{prop:WLEM-upper-bound}, we have $h\leq_{\sf tW}{\sf C}_{1/2}$; hence,
$\forall x\exists ! yH(x,y)$ is ${\sf C}_{1/2}$-realizable.
By non-computability of $h$, $H(x,\varphi_e(x))$ is false for any $e$; in particular, it cannot be realizable relative to any oracle.
\end{proof}

\subsection{Majorization Versions of $\bflogic{CT}_0$}\label{sec:major-Church-thesis}

We introduce the principle $\bflogic{CT}_0^{\rm maj}$, which states that every arithmetically definable function is majorized by some computable function, as follows.
\begin{definition}
$\bflogic{CT}_0^{\rm maj}$ is the following principle:
\[
\forall x\in\N\exists y\in\N.\ A(x,y)
\to
\exists e\in\N\forall x\in\N.\ {\varphi_e(x)\downarrow}\land \exists y\leq\varphi_e(x).\ A(x,y).
].
\]
\end{definition}

\begin{theorem}\label{thm:finite-maj-realizable}
Let $F\pcolon\om\times\Lambda\tto\om$ be a potted bilayer function whose codomain is finite.
Then, $\bflogic{CT}_0^{\rm maj}$ is $F$-realizable.
%
\end{theorem}

\begin{proof}
{\em Premise.}
Let $a$ be an $F$-\aaa realizer for the premise $\forall x\exists y A(x,y)$ of $\bflogic{CT}_0^{\rm maj}$.
As in Example \ref{exa:AE-thm-realizer-simplify}, and similar to Theorem \ref{thm:MP-vs-CT0-realizability}, etc., such an $F$-realizer $a$ is described as follows:
\begin{itemize}
\item 
For any $\pair{y,z}\in\Phi_a^F(x)$, the second value $z$ is an $F$-realizer for $A(x,y)$.
\end{itemize}

{\em Conclusion.}
To realize the conclusion of $\bflogic{CT}_0^{\rm maj}$, it suffices to construct an index $e$ of an upper bound, its halting stage $s$, some $y\leq\varphi_e(x)$, and a realizer for $A(x,y)$.

{\em Index for an upper bound.}
\begin{itemize}
\item 
Using a computable function $\beta$ in Lemma \ref{lem:finite-codomain-CT-maj-comp}, we see that some $\pair{y,z}\leq\varphi_{\beta(a)}(x)$ must gives a solution $\pair{y,z}\in\Phi_{a}^{F}(x)$.
Looking at the verification part of the proof of Lemma \ref{lem:finite-codomain-CT-maj-comp}, we see that, in fact, for any $\pair{y,z}\in\Phi_a^F(x)$, we have $\pair{y,z}\leq\varphi_{\beta(a)}(x)$.
\item 
Since $\varphi_{\beta(a)}$ is a total computable function, we obtain an $F$-index $e_0$ for $\varphi_{\beta(a)}$ by considering a total $F$-computation that does not make queries to $F$.
\end{itemize}

{\em Halting stage.}
\begin{itemize}
\item 
Since $\varphi_{e_0}$ is a total function, for any $x\in\om$, the function $s$ that computes the halting stage $s(x)$, i.e.~$\varphi_{e_0}(x)[s(x)] \downarrow$, is total computable.
Therefore, we can obtain an index $e_s$ for this $F$-computation.
\end{itemize}

{\em Output $y$ and Realizer for $A$.}
\begin{itemize}
\item $\Phi_a^F$ directly produces them.
\begin{itemize}
\item As seen above, for any $\pair{y,z}\in\Phi_a^F(x)$, the second value $z$ is a realizer for $A(x,y)$.
\item As mentioned above, the verification part of the proof of Lemma \ref{lem:finite-codomain-CT-maj-comp} shows that in fact, for any $\pair{y,z}\in\Phi_a^F$, we have $\pair{y,z}\leq\varphi_{\beta(a)}(x)=\varphi_{e_0}(x)$.
In particular, $y\leq\varphi_{e_0}(x)$ is guaranteed.
\end{itemize}
\end{itemize}

Consequently, $\bflogic{CT}_0^{\rm maj}$ is $F$-realizable.
\end{proof}

We next introduce the principle $\bflogic{sCT}_0^{\rm maj}$, which states that every arithmetically definable function is majorized by some lower semicomputable function, as follows.
\begin{definition}
$\bflogic{sCT}_0^{\rm maj}$ is the following principle:
\[
\forall x\in\N\exists y\in\N.\ A(x,y)
\longrightarrow
\exists e\in\N\forall x\in\N\exists y\in\N.\ (A(x,y)\land \neg\neg\exists t\in\N.\ {\varphi_e(x,t)\down}\land y\leq\varphi_e(x)).
\]
\end{definition}

\begin{prop}\label{prop:finite-semi-maj-realizable}
Let $F\pcolon\om\times\Lambda\tto\om$ be a potted bilayer function whose codomain is finite.
Then, $\bflogic{sCT}_0^{\rm maj}$ is $(\sii{\sf C}_{1/\om}+F)$-realizable.
\end{prop}

\begin{proof}
The argument is the same as in Theorem \ref{thm:finite-maj-realizable}.
Here, use Lemma \ref{lem:finite-codomain-CT-maj-comp2} instead of Lemma \ref{lem:finite-codomain-CT-maj-comp}.
\end{proof}

\subsection{Independence of Premise $\bflogic{IP}$}\label{sec:inde-premi}

A prime example of a principle that does not hold in Kleene realizability (i.e., partial computable realizability), but does hold in modified realizability (a kind of total computable realizability) is the principle of {\bf independence of premise}.

\begin{definition}
$\bflogic{IP}$ is the following principle:
\[
\forall x\in\N.[
(\neg D(x)\to\exists y A(x,y))
\to
\exists y\in\N.(\neg D(x)\to A(x,y))
].
\]
\end{definition}

\begin{theorem}\label{thm:totality-vs-IP-realizability}
Let $F\pcolon\om\tto\om$ a potted problem whose graph is arithmetically definable.
Then the following two conditions are equivalent:
\begin{enumerate}
\item $\bflogic{IP}$ is $F$-realizable.
\item $F$ is multi-query total Weihrauch equivalent to a total potted problem (i.e., its \ppp /\aaa domains are $\om$).
\end{enumerate}
\end{theorem}

\begin{proof}
We first give an explicit description of an $F$-realizer for $\bflogic{IP}$.
As in Example \ref{exa:AE-thm-realizer-simplify}, an $F$-realizer $a$ for the premise $\neg D(x)\to\exists y^\N A(x,y)$ is described as follows:
\begin{itemize}
\item {\em Potential.}
Regardless of $x$, $0$ is a \ppp realizer for $\neg D(x)$, so for any $\pair{y,z}\in\Phi_a^+(0)$, the second value $z$ is a \ppp realizer for $A(x,y')$ for some $y'$.
\item {\em Actual.}
If $\neg D(x)$ is realizable, $0$ is its $F$-\aaa realizer, so for any $\pair{y,z}\in\Phi_a^F(0)$, the second value $z$ is an $F$-\aaa realizer for $A(x,y)$.
\end{itemize}

Similarly, an $F$-realizer for the conclusion $\exists y^\N(\neg D(x)\to A(x,y))$ is a pair of a natural number $y$ and an index $b$ such that:
\begin{itemize}
\item {\em Potential.}
As $0$ is a \ppp realizer for $\neg D(x)$, any $z\in\Phi_b^+(0)$ is a \ppp realizer for $A(x,y')$ for some $y'$.
\item {\em Actual.}
If $\neg D(x)$ is realizable, $0$ is its $F$-\aaa realizer, so any $z\in\Phi_b^F(0)$ is an $F$-\aaa realizer for $A(x,y)$.
\end{itemize}


(2)$\Rightarrow$(1):
Without loss of generality, we may assume that $F$ itself is total.
Consider the following process:
\begin{itemize}
\item Let $a$ be an \ppp realizer for the premise $\neg D(x)\to\exists y^\N A(x,y)$.
\item We want to construct an $F$-\ppp realizer for the conclusion $\exists y^\N(\neg D(x)\to A(x,y))$:
\begin{itemize}
\item Each $y\in\pi_0\circ\Phi_a^+(0)$ is a \ppp realizer for the existential witness $y^\N$.
\item 
Let $k(c)$ be an index of the computation that immediately outputs $c$ regardless of the input.
The composition $k\circ\pi_1\circ\Phi_a^+(0)$ appears to yield an $F$-realizer for the part inside the parentheses.
\end{itemize}
\item 
Therefore, as a $F$-computable algorithm for deriving the conclusion from the premise, we consider the process $a\mapsto\pair{\pi_0,k\circ\pi_1}\circ\Phi_a^+(0)$.
\end{itemize}

{\em Verification:}
\begin{itemize}
\item {\em Potential.}
Let $a$ be a \ppp realizer for the premise.
\begin{itemize}
\item Since $\Phi_a^+$, $\pi_i$ and $k$ are total, the computation $\pair{\pi_0,k\circ\pi_1}\circ\Phi_a^+(0)$ must halt.
\item The second value of $\pair{y,b}\in\pair{\pi_0,k\circ\pi_1}\circ\Phi_a^+(0)$ is of the form $b=k(z)$ for some $\pair{y,z}\in\Phi_a^+(0)$.
\item As $0$ is a \ppp realizer for $\neg D(x)$, we have $\Phi_b(0)=\Phi_{k(z)}(0)=z$.
As mentioned above, this is a \ppp realizer for $A(x,y')$ for some $y'$.
\item 
Therefore, any $\pair{y,b}\in\pair{\pi_0,k\circ\pi_1}\circ\Phi_a^+(0)$ is a \ppp realizer for $\exists y^\N(\neg D(x)\to A(x,y))$.
\end{itemize}
\item {\em Actual.}
Let $a$ be an $F$-\aaa realizer for the premise.
\begin{itemize}
\item As $F$ is total, regardless of realizability for $D(x)$, the computation of $\Phi_a^F(0)$ make queries only in the \aaa domain of $F$.
Therefore, $\pair{\pi_0,k\circ\pi_1}\circ\Phi_a^F(0)$ always output some value.
\item 
For the remaining part, by the same argument as above, any such element yields an $F$-\aaa realizer for the conclusion part.
\end{itemize}
\end{itemize}

Consequently, $\bflogic{IP}$ is $F$-realizable.

\medskip

(1)$\Rightarrow$(2):
Since $F$ has an arithmetical graph $G_F$, its domain can also be expressed by an arithmetical formula $D_F$.
Consider the double-negation translations $D_F^{\neg\neg}$ and $G_F^{\neg\neg}$.
We construct an algorithm $a$ which $F$-realizes the premise $\forall x^\N(\neg\neg D_F^{\neg\neg}(x)\to\exists y^\N G_F^{\neg\neg}(x,y))$ of $\bflogic{IP}$:
\begin{itemize}
\item Given $x\in\om$, use this $x$ as a query to $F$.
Consider a total $F$-computable process $\Phi_a^+$ that receives a response $y\in F(x)$ and outputs the pair $(y,0)$ regardless of the input.
\end{itemize}

{\em Verification:}
We show that $a$ is an $F$-realizer for the premise of $\bflogic{IP}$.
\begin{itemize}
\item {\em Potential.}
$0$ is a \ppp realizer for $\neg\neg D_F^{\neg\neg}(x)$.
Then $\Phi_a^+(0)$ makes a query $x$ to $F$, and since $x\in\om=\dom_+(F)$, the process always receives some $y$.
Therefore, regardless of the input, $\Phi_a^+(x)$ always returns an index of a computation that outputs the pair $\pair{y,0}$.
This is a \ppp realizer for $\neg\neg D_F^{\neg\neg}(x)\to\exists y^\N G_F^{\neg\neg}(x,y)$.
\item {\em Actual.}
If $\neg\neg D^{\neg\neg}_F(x)$ is realizable, then $0$ is its \aaa realizer and $D_F(x)$ is true, so $x\in{\rm dom}(F)$.
Thus, the process must receive some $y\in F(x)$.
Therefore, regardless of an \aaa realizer for $\neg\neg D^{\neg\neg}_F(x)$, $\Phi_a^F(x)$ outputs the pair of $y\in F(x)$ and a realizer $0$ for $\neg\neg G_F^{\neg\neg}(x,y)$.
\end{itemize}

Hence, $a$ realizes the premise of $\bflogic{IP}$.
Applying an $F$-realizer for $\bflogic{IP}$ to this $a$, we get an $F$-realizer $b$ for the conclusion $\forall x^\N\exists y^\N(\neg\neg D_F^{\neg\neg}(x)\to G_F^{\neg\neg}(x,y))$.
\begin{itemize}
\item This means that some total extension $\hat{F}$ of $F$ is total $F$-computable.
Indeed, $\hat{F}=\pi_0\circ\Phi_b^F$ is a total extension of $F$:
\begin{itemize}
\item {\em Totality.}
Since any $x\in\om$ is an \aaa realizer for $(x\approx_{\sf N} x)$, the computation $\Phi_b^F(x)$ halts.
Therefore, $\pi_0\circ\Phi_b^F(x)$ is defined.
\item {\em Extension.}
Since any $x\in\om$ is an \aaa realizer for $(x\approx_\sfN x)$, for any $\pair{y,z}\in\Phi_b^F(x)$, the second value $z$ is an $F$-\aaa realizer for $\neg\neg D_F^{\neg\neg}(x)\to \neg\neg G_F^{\neg\neg}(x,y)$.
If $x\in\dom(F)$ then $\neg\neg D_F^{\neg\neg}(x)$ has an \aaa realizer $0$, so $\Phi_z^F(0)$ outputs an $F$-\aaa realizer for $\neg\neg G_F^{\neg\neg}(x,y)$.
In particular, $G_F(x,y)$ is true, which means $y\in F(x)$.
\end{itemize}
\end{itemize}

Consequently, we get $\hat{F}=\pi_0\circ\Phi_b^F$, which implies $\hat{F}\tgw F$.
As $F$ is an extension of $\hat{F}$, we also have $F\tW\hat{F}$.
Combining them, $F$ is total Weihrauch equivalent to the total extension $\hat{F}$.
%
\end{proof}

\begin{cor}\label{cor:LLPOn-sat-IP}
$\bflogic{IP}$ is $\pii{\sf C}_{\geq n-1/n}$-realizable for any $n\leq\om$.
\end{cor}

\begin{proof}
By Theorem \ref{thm:totality-vs-IP-realizability} and Lemma \ref{lem:LLPO-totalizable}.
\end{proof}

\begin{cor}\label{cor:MP-doesnot-sat-IP}
$\bflogic{IP}$ is neither $\sii{\sf C}_{1/\om}$-realizable nor $\sii{\sf C}_{n-1/n}$-realizable.
\end{cor}

\begin{proof}
Let $F\in\{\sii{\sf C}_{1/\om},\sii{\sf C}_{n-1/n}\}_{n\leq\om}$.
Suppose that $\bflogic{IP}$ is $F$-realizable.
\begin{itemize}
\item Theorem \ref{thm:totality-vs-IP-realizability} gives an total extension $\hat{F}$ of $F$ such that $\hat{F}\tgw F$.
\item Since $\hat{F}\tgw\sii{\sf C}_{1/\om}$, by Theorem \ref{MP-equial-partial-computable}, $\hat{F}$ has a computable choice function $g$.
\end{itemize}

However, no total extension of $\sii{\sf C}_{1/\om}$ or $\sii{\sf C}_{n-1/n}$ is computable:
\begin{itemize}
\item Suppose that $\Pi_1\text{-}{\sf coUC}_\om$ has a total computable choice function $g\colon\om\to\om$.
\item By the recursion theorem, one can take an index $e$ such that $P_e=\om\setminus\{g(e)\}$.
\item Then we have $e\in\dom(\Pi_1\text{-}{\sf C}_{\om-1/\om})$ and $g(e)\not\in\Pi_1\text{-}{\sf C}_{\om-1/\om}(e)=P_e$.
\end{itemize}
Consequently, $\bflogic{IP}$ is not $F$-realizable.
\end{proof}

\begin{fact}[see Troelstra {\cite[Theorem 3.4.14]{Tro73}}]\label{fact:IP-ECT-inconsistent}
$\bflogic{IP}$ is incompatible with $\bflogic{ECT}_0$ over Heyting arithmetic.
\end{fact}

\begin{cor}
If $F$ is a total potted problem with an arithmetical graph, then $\bflogic{ECT}_0$ is not $F$-realizable.
\end{cor}

\begin{proof}
By Theorem \ref{thm:totality-vs-IP-realizability} and Fact \ref{fact:IP-ECT-inconsistent}.
\end{proof}

\begin{remark}
On the one hand,
$(\sii{\sf C}_{1/\om}+\pii{\sf C}_{\geq 1/2})$-realizability appears to correspond to Lifshitz realizability.
On the other hand, the behavior of $\pii{\sf C}_{\geq 1/2})$-realizability (total Lifshitz realizability) differs significantly from that of Lifshitz realizability:
By Theorem \ref{thm:LLPO-not-CT0unique} and Corollary \ref{cor:LLPOn-sat-IP}, $\bflogic{IP}$ is $\pii{\sf C}_{\geq 1/2})$-realizable, but $\bflogic{CT}_0!$ is not $\pii{\sf C}_{\geq 1/2})$-realizable.
\end{remark}

Note that the arithmetical definability of the graph is not used in the proof of Theorem \ref{thm:totality-vs-IP-realizability} (2)$\Rightarrow$(1).
In fact, the following holds:

\begin{prop}\label{prop:totality-vs-IP-realizability-bilayer}
Let $F\colon\om\times\Lambda\tto \om$ be a total bilayer potted problem.
Then $\bflogic{IP}$ is $F$-realizable.
\end{prop}

\begin{proof}
The argument is the same as in the proof of Theorem \ref{thm:totality-vs-IP-realizability} (2)$\Rightarrow$(1).
Since {\tt Nimue}'s strategy has no effect on the potential part, it suffices to discuss only the actual part.
\begin{itemize}
\item Let $a$ be an \aaa realizer for the premise $\neg D(x)\to\exists y^\N A(x,y)$; that is, there is $F$-branching tree $\mathbb{T}$ such that $\Phi_{a|\mathbb{T}}^F(0)$ outputs a pair of an existential witness $y$ and an realizer for $A(x.y)$.
\item Then $\pair{\pi_0,k\circ\pi_1}\circ\Phi_{a|\mathbb{T}}^F(0)$ is a realizer for the conclusion $\exists y(\neg D(x)\to A(x,y))$.
\end{itemize}
This shows that $\bflogic{IP}$ is $F$-realizable.
\end{proof}

\begin{cor}\label{cor:IP-wlem}
$\bflogic{IP}$ is ${\sf C}_{k/n}$-realizable for any $k\leq n\leq\om$.
\end{cor}

\begin{proof}
Clearly ${\sf C}_{k/n}$ is equivalent to a total problem $F$ defined by $F(n|\alpha)={\sf C}_{k/n}(\ast|\alpha)$ for any $n\in\om$.
Thus, the assertion follows from Proposition \ref{prop:totality-vs-IP-realizability-bilayer}.
\end{proof}

\subsection{Proof of Main Theorem}\label{sec:proof-main}

Finally, we prove Theorem \ref{thm:main-theorem}; indeed, for $n\geq 2$, each of the following items is realizable with respect to some oracle:
\begin{enumerate}
\item $\neg\bflogic{MP}^\lor_{\om}+\bflogic{CT}_0+\bflogic{IP}$.
\item $\bflogic{MP}_{n+1}^\lor+\neg\bflogic{MP}_{n}^\lor+\bflogic{wCT}_0+\bflogic{CT}_0!+\neg\bflogic{CT}_0+\neg\bflogic{IP}$.
\item $\bflogic{MP}^\lor+\neg\bflogic{MP}_{\rm PR}+\bflogic{wCT}_0+\neg\bflogic{CT}_0!+\neg\bflogic{IP}$.
\item $\bflogic{MP}+\bflogic{CT}_0+\neg\bflogic{IP}$.
\item $\bflogic{LLPO}_{n+1}+\neg\bflogic{MP}_{n}^\lor+\bflogic{CT}_0!+\bflogic{CT}_0^{\rm maj}+\bflogic{IP}$.
\item $\bflogic{LLPO}+\neg\bflogic{MP}_{\rm PR}+\bflogic{wCT}_0!+\neg\bflogic{CT}_0!+\bflogic{CT}_0^{\rm maj}+\bflogic{IP}$.
\item $\bflogic{MP}+\bflogic{LLPO}_{n+1}+\neg\bflogic{LLPO}_{n}+\bflogic{CT}_0!$.
\item $\bflogic{WLEM}_{n+1}+\neg\bflogic{MP}_{n}^\lor+\bflogic{CT}_0!+\bflogic{CT}_0^{\rm maj}+\bflogic{IP}$.
\item $\bflogic{WLEM}+\neg\bflogic{MP}_{\rm PR}+\neg\bflogic{CT}_0!+\bflogic{CT}_0^{\rm maj}+\bflogic{IP}$.
\item $\bflogic{MP}+\bflogic{LLPO}_{n}+\bflogic{WLEM}_{n+1}+\neg\bflogic{LLPO}_{n-1}+\neg\bflogic{WLEM}_{n}+\bflogic{CT}_0!$.
\item $\bflogic{MP}+\bflogic{LLPO}+\bflogic{WLEM}_{n+1}+\neg\bflogic{WLEM}_n+\bflogic{CT}_0!$.
\item $\bflogic{MP}_{\rm PR}+\bflogic{WLEM}+\neg\bflogic{LEM}+\neg\bflogic{CT}_0^{\rm maj}+\bflogic{sCT}_0^{\rm maj}$.
\end{enumerate}

\begin{proof}[Proof of Theorem \ref{thm:main-theorem}]

(1)
It is known that $\bflogic{CT}_0$ and $\bflogic{IP}$ are modified realizable \cite[Corollary 3.4.13]{Tro73}.
For the sake of completeness, we explain the reason here.
Consider ${\rm id}_\om$-realizability.
\begin{itemize}
\item By Theorem \ref{thm:MPomega-not-total-computable}, we have $\sii{\sf C}_{\om-1/\om}\not\tgw{\rm id}_\om$, so by Theorem \ref{thm:MPnlor-realizable-to-above}, $\bflogic{MP}_{\om}^\lor$ is not ${\rm id}_\om$-realizable.
\item By Proposition \ref{prop:id-CT0}, $\bflogic{CT}_0$ is ${\rm id}_\om$-realizable.
\item By Theorem \ref{thm:totality-vs-IP-realizability}, since ${\rm id}_\om$ is total, $\bflogic{IP}$ is ${\rm id}_\om$-realizable.
\end{itemize}

(2)
Consider $F:=\sii{\sf C}_{n/n+1}$.
\begin{itemize}
\item By Theorem \ref{thm:MPn-reducible-to-realizable}, $\bflogic{MP}_{n+1}^\lor$ is $F$-realizable.
\item Theorem \ref{thm:separation-MPlorn-WLEMn} implies $\sii{\sf C}_{n-1/n}\not\tgw\sii{\sf C}_{n/n+1}$; hence, by Theorem \ref{thm:MPnlor-realizable-to-above}, $\bflogic{MP}_{n}^\lor$ is not $F$-realizable.
\item Since $\sii{\sf C}_{n/n+1}\tW\sii{\sf C}_{1/\om}$, by Theorem \ref{thm:MP-vs-CT0-realizability}, $\bflogic{wCT}_0$ is $F$-realizable.
\item Since $n\geq 2$, we have $\sii{\sf C}_{n/n+1}\tW{\sf C}_{n/n+1}\tW{\sf C}_{2/3}$; thus, by Theorem \ref{thm:WLEMthree-CTrealizable}, $\bflogic{CT}_0!$ is $F$-realizable.
\item By Theorem \ref{thm:MPn-sheaf-not-CT0}, $\bflogic{CT}_0$ is not $F$-realizable.
\item By Corollary \ref{cor:MP-doesnot-sat-IP}, $\bflogic{IP}$ is not $F$-realizable.
\end{itemize}

(3)
Consider $F:=\sii{\sf C}_{1/2}$.
\begin{itemize}
\item By Theorem \ref{thm:MPn-reducible-to-realizable}, $\bflogic{MP}^\lor$ is $F$-realizable.
\item Theorem \ref{thm:MP-WLEM-separation} implies $\sii{\sf C}_{1/\om}\not\tgw\sii{\sf C}_{1/2}$; thus, by Theorem \ref{thm:MP-realizable-to-reducible}, $\bflogic{MP}_{\rm PR}$ is not $F$-realizable.
\item Since $\sii{\sf C}_{1/2}\tW\sii{\sf C}_{1/\om}$, by Theorem \ref{thm:MP-vs-CT0-realizability}, $\bflogic{wCT}_0$ is $F$-realizable.
\item By Theorem \ref{thm:LLPO-not-CT0unique}, $\bflogic{CT}_0!$ is not $F$-realizable.
\item By Corollary \ref{cor:MP-doesnot-sat-IP}, $\bflogic{IP}$ is not $F$-realizable.
\end{itemize}

(4)
It is known that these are Kleene realizable.
Here, consider $F:=\sii{\sf C}_{1/\om}$.
\begin{itemize}
\item By Theorem \ref{thm:MP-in-MP-sheaf}, $\bflogic{MP}$ is $F$-realizable.
\item By Theorem \ref{thm:MP-vs-CT0-realizability-2}, $\bflogic{CT}_0$ is $F$-realizable.
\item By Corollary \ref{cor:MP-doesnot-sat-IP}, $\bflogic{IP}$ is not $F$-realizable.
\end{itemize}

(5)
Consider $F:=\pii{\sf C}_{\geq n/n+1}$.
\begin{itemize}
\item By Theorem \ref{thm:LLPO-reducible-to-realizable}, $\bflogic{LLPO}_{n+1}$ is $F$-realizable.
\item Theorem \ref{thm:separation-MPlorn-WLEMn} implies $\sii{\sf C}_{n-1/n}\not\tgw\pii{\sf C}_{\geq n/n+1}$; thus, by Theorem \ref{thm:MPnlor-realizable-to-above}, $\bflogic{MP}_{n}^\lor$ is not $F$-realizable.
\item Since $n\geq 2$, we have $\pii{\sf C}_{\geq n/n+1}\tW{\sf C}_{n/n+1}\tW{\sf C}_{2/3}$; thus, by Theorem \ref{thm:WLEMthree-CTrealizable}, $\bflogic{CT}_0!$ is $F$-realizable.
\item Since the codomain of $\pii{\sf C}_{\geq n/n+1}$ is finite, by Theorem \ref{thm:finite-maj-realizable}, $\bflogic{CT}_0^{\rm maj}$ is $F$-realizable.
\item By Corollary \ref{cor:LLPOn-sat-IP}, $\bflogic{IP}$ is $F$-realizable.
\item Since $\pii{\sf C}_{\geq n/n+1}$ is a potted problem, Theorem \ref{thm:separation-WLEMn-nplusone} implies ${\sf C}_{\om-1/\om}\not\tgw\pii{\sf C}_{\geq n-1/n}$; thus, by Theorem \ref{thm:WLEMn-realizable-iff-reducible}, $\bflogic{WLEM}_\om$ is not $F$-realizable.
\end{itemize}

(6)
Consider $F:=\pii{\sf C}_{\geq 1/2}$.
\begin{itemize}
\item By Theorem \ref{thm:LLPO-reducible-to-realizable}, $\bflogic{LLPO}$ is $F$-realizable.
\item Theorem \ref{thm:MP-WLEM-separation} implies $\sii{\sf C}_{1/\om}\not\tgw\pii{\sf C}_{\geq1/2}$; thus, by Theorem \ref{thm:MP-realizable-to-reducible}, $\bflogic{MP}_{\rm PR}$ is not $F$-realizable.
\item By Theorem \ref{thm:MP-LLPO-models-CT0u}, $\bflogic{wCT}_0!$ is $F$-realizable.
\item By Theorem \ref{thm:LLPO-not-CT0unique}, $\bflogic{CT}_0!$ is not $F$-realizable.
\item Since the codomain of $\pii{\sf C}_{\geq 1/2}$ is finite, by Theorem \ref{thm:finite-maj-realizable}, $\bflogic{CT}_0^{\rm maj}$ is $F$-realizable.
\item By Corollary \ref{cor:LLPOn-sat-IP}, $\bflogic{IP}$ is $F$-realizable.
\item Since $\pii{\sf C}_{\geq 1/2}$ is a potted problem, Theorem \ref{thm:separation-WLEMn-nplusone} implies ${\sf C}_{\om-1/\om}\not\tgw\pii{\sf C}_{\geq 1/2}$; thus, by Theorem \ref{thm:WLEMn-realizable-iff-reducible}, $\bflogic{WLEM}_\om$ is not $F$-realizable.
\end{itemize}

(7)
Consider $F:=\sii{\sf C}_{1/\om}+\pii{\sf C}_{\geq n/n+1}$.
\begin{itemize}
\item By Theorem \ref{thm:MP-in-MP-sheaf}, $\bflogic{MP}$ is $F$-realizable.
\item By Theorem \ref{thm:LLPO-reducible-to-realizable}, $\bflogic{LLPO}_{n+1}$ is $F$-realizable.
\item By Proposition \ref{prop:separation-Eff-LLPO-n-nplusone} and Theorem \ref{thm:LLPOn-realizable-to-above}, $\bflogic{LLPO}_{n}$ is not $F$-realizable.
\item By Theorem \ref{thm:MP-LLPO-models-CT0u-2}, $\bflogic{CT}_0!$ is $F$-realizable.
\item Since $F$ is a potted problem, Theorem \ref{thm:separation-WLEMn-nplusone} implies ${\sf C}_{\om-1/\om}\not\tgw\sii{\sf C}_{1/\om}+\pii{\sf C}_{\geq n/n+1}$; thus, by Theorem \ref{thm:WLEMn-realizable-iff-reducible}, $\bflogic{WLEM}_\om$ is not $F$-realizable.
\end{itemize}

(8)
Consider $F:={\sf C}_{n/n+1}$.
\begin{itemize}
\item By Theorem \ref{thm:WLEMn-realizable-iff-reducible}, $\bflogic{WLEM}_{n+1}$ is $F$-realizable.
\item By Theorem \ref{thm:separation-MPlorn-WLEMn} and Theorem \ref{thm:MPnlor-realizable-to-above}, $\bflogic{MP}_n^\lor$ is not $F$-realizable.
\item By Theorem \ref{thm:WLEMthree-CTrealizable}, $\bflogic{CT}_0!$ is $F$-realizable.
\item Since the codomain of ${\sf C}_{n/n+1}$ is finite, by Theorem \ref{thm:finite-maj-realizable}, $\bflogic{CT}_0^{\rm maj}$ is $F$-realizable.
\item By Corollary \ref{cor:IP-wlem}, $\bflogic{IP}$ is $F$-realizable.
\end{itemize}

(9)
Consider $F:={\sf C}_{1/2}$.
\begin{itemize}
\item By Theorem \ref{thm:WLEMn-realizable-iff-reducible}, $\bflogic{WLEM}$ is $F$-realizable.
\item By Theorem \ref{thm:MP-WLEM-separation} and Theorem \ref{thm:MP-realizable-to-reducible}, $\bflogic{MP}_{\rm PR}$ is not $F$-realizable.
\item By Observation \ref{thm:WLEM-not-CT0unique}, $\bflogic{CT}_0!$ is not $F$-realizable.
\item Since the codomain of ${\sf C}_{1/2}$ is finite, by Theorem \ref{thm:finite-maj-realizable}, $\bflogic{CT}_0^{\rm maj}$ is $F$-realizable.
\item By Corollary \ref{cor:IP-wlem}, $\bflogic{IP}$ is $F$-realizable.
\end{itemize}

(10)
Consider $F:=\sii{\sf C}_{1/\om}+{\sf C}_{n/n+1}$.
\begin{itemize}
\item By Proposition \ref{prop:MP-in-MP-sheaf2} and $n\geq 2$, $\bflogic{MP}$ is $F$-realizable.
\item By Proposition \ref{prop:LLPOn-reducible-WLEMnplusone} and Proposition \ref{prop:LLPO-reducible-to-realizable2}, $\bflogic{LLPO}_{n}$ is $F$-realizable.
\item By Theorem \ref{thm:WLEMn-realizable-iff-reducible}, $\bflogic{WLEM}_{n+1}$ is $F$-realizable.
\item By Proposition \ref{prop:separation-EFF-LLPO-n-WLEM} and Theorem \ref{thm:LLPOn-realizable-to-above}, $\bflogic{LLPO}_{n-1}$ is not $F$-realizable.
\item By Theorem \ref{thm:separation-WLEMn-nplusone} and Theorem \ref{thm:WLEMn-realizable-iff-reducible}, $\bflogic{WLEM}_{n}$ is not $F$-realizable.
\item By Theorem \ref{thm:MP-LLPO-models-CT0u}, $\bflogic{CT}_0!$ is $F$-realizable.
\end{itemize}

(11)
Consider $F:=\sii{\sf C}_{1/\om}+\pii{\sf C}_{\geq 1/2}+{\sf C}_{n/n+1}$.
\begin{itemize}
\item By Proposition \ref{prop:MP-in-MP-sheaf2} and $n\geq 2$, $\bflogic{MP}$ is $F$-realizable.
\item By Proposition \ref{prop:LLPOn-reducible-WLEMnplusone} and $n\geq 2$, we have $F\tgw\sii{\sf C}_{1/\om}+{\sf C}_{2/3}$; thus, by Proposition \ref{prop:LLPO-reducible-to-realizable2}, $\bflogic{LLPO}$ is $F$-realizable.
\item By Theorem \ref{thm:WLEMn-realizable-iff-reducible}, $\bflogic{WLEM}_{n+1}$ is $F$-realizable.
\item By Theorem \ref{thm:separation-WLEMn-nplusone} and Theorem \ref{thm:WLEMn-realizable-iff-reducible}, $\bflogic{WLEM}_{n}$ is not $F$-realizable.
\item By Theorem \ref{thm:MP-LLPO-models-CT0u}, $\bflogic{CT}_0!$ is $F$-realizable.
\end{itemize}

(12)
Consider $F:=\sii{\sf C}_{1/\om}+{\sf C}_{1/2}$.
\begin{itemize}
\item By Observation \ref{obs:MP-in-MP-sheaf3}, $\bflogic{MP}_{\rm PR}$ is $F$-realizable.
\item By Theorem \ref{thm:WLEMn-realizable-iff-reducible}, $\bflogic{WLEM}$ is $F$-realizable.
\item By Theorem \ref{thm:separation-DML-WLEM} and Theorem \ref{thm:DML-reducible-iff-DNE-realizable}, $\bflogic{DML}_\N$ is not $F$-realizable.
In particular, $\bflogic{LEM}$ is not $F$-realizable (since $\bflogic{LEM}$ implies all axioms in classical logic).
\item Since $\bflogic{WLEM}$ implies $\bflogic{WLPO}$ (the $\Pi_1$ law of excluded middle), combined with $\bflogic{MP}$, this implies $\bflogic{LPO}$ (the $\Sigma_1$ law of excluded middle); see e.g.~\cite{Die18}.
Then, clearly, $\bflogic{CT}_0^{\rm maj}$ is not $F$-realizable; for instance, the halting stage function grows faster than any computable function.
That is, for a suitable decidable formula $D$, let $A(x,y)$ be the following formula:
\[\exists z<y.\ D(x,z)\lor\neg\exists zD(x,z).\]
\item By Proposition \ref{prop:finite-semi-maj-realizable}, $\bflogic{sCT}_0^{\rm maj}$ is $F$-realizable.
\end{itemize}
\end{proof}

\section{Future Works}


First, regarding the logical direction:
Hendtlass-Lubarsky's work \cite{HeLu16} is notable for separating the strengths of the hierarchy of Markov's principles while satisfying the {\em axiom of dependent choice} ${\sf DC}$.
The trade-off is that their constructions are quite complicated.
One of our next goals is to separate the strengths of the hierarchy of Markov's principles using simple realizability models while satisfying ${\sf DC}$.

Second, regarding the topos-theoretic direction: 
it is natural to ask whether the effective topos coincides with the modified realizability topos augmented by Markov's principle as an oracle, and whether the Lifschitz realizability topos coincides with the modified realizability topos augmented by the principle ${\sf MP}+{\sf LLPO}$ as an oracle.

\begin{question}
Is the $\sii{\sf C}_{1/\om}$-sheaf subtopos of the Grayson topos (or the modified realizability topos) equivalent to the effective topos?
\end{question}

\begin{question}
Is the $(\sii{\sf C}_{1/\om}+\pii{\sf C}_{\geq 1/2})$-sheaf subtopos of the Grayson topos (or the modified realizability topos) equivalent to the Lifshitz realizability topos?
\end{question}

Regarding the third direction:
As mentioned before, an oracle computability may be viewed as a free monad, and a LT-modality is a certain type of monad.
Meanwhile, there is existing research on the notion of {\bf monadic combinatory algebra} (Cohen et al.~\cite{CGKM}), which can be expected to relate to our work.
For instance, a relational combinatory algebra can be thought of as a combinatory algebra equipped with a multi-valued oracle.
As a specific example, the angelic computation model with ``flip'' may be related to the oracle $\mathsf{C}_{1/2}$.
Cohen et al.~\cite{CAT19} also present the demonic computation model with ``flip'' as an example that does not satisfy the axiom of countable choice ${\sf AC}_\om$, while, in our context, we also have a vast number of oracle realizability models that fail to satisfy ${\sf AC}_\om$; see \cite{Kih20} --- in fact, most of realizability models based on (multi-valued) oracles that arise naturally in computable analysis and reverse mathematics refute ${\sf AC}_\om$.

To what extent is a partial combinatory algebra (or its relative) endowed with a LT-modality related to a monadic combinatory algebra?
However, their work \cite{CGKM} considers monads on $\mathbf{Set}$, which is unsuitable for our setting.
Furthermore, their framework combines a {\bf monad} $M$ on sets with an {\bf $M$-modality} on propositions -- offering perhaps greater customizability than ours, but resulting in a more complicated structure.
A task for future work is to identify a simple, concrete framework that clarifies the relationship between our research and theirs.

\bibliographystyle{plain}
\bibliography{references-MRT}
\end{document}